\documentclass[11pt]{article}

\usepackage[margin=2.5cm]{geometry} 
\usepackage{graphicx} 
\usepackage{amssymb, amsfonts, amsmath, amscd, amsthm} 
\usepackage{mathtools} 
\usepackage{titlesec} 
\usepackage{csquotes} 
\usepackage[shortlabels]{enumitem} 
\usepackage{tikz-cd} 
\usepackage{wrapfig} 
\usepackage{hyperref} 
\usepackage{epigraph} 
\usepackage{changepage} 

\definecolor{researchgate-color}{HTML}{0e5d4e}
\definecolor{arxiv-color}{rgb}{0.70, 0.11, 0.11}

\hypersetup{
    colorlinks   = true,
    linkcolor    = researchgate-color,
    citecolor    = arxiv-color,
    pdftitle     = {A Primer on Coorbit Theory},
}

\numberwithin{equation}{section}

\titleformat*{\section}{\Large \scshape\center}
\titleformat*{\subsection}{\fontsize{14}{14} \sffamily}

\theoremstyle{plain}
\newtheorem{theorem}{Theorem}[section]
\newtheorem*{theorem*}{Theorem}
\newtheorem{lemma}[theorem]{Lemma}
\newtheorem{proposition}[theorem]{Proposition}
\newtheorem{corollary}[theorem]{Corollary}

\theoremstyle{definition}
\newtheorem{definition}[theorem]{Definition}
\newtheorem{example}[theorem]{Example}

\theoremstyle{remark}
\newtheorem*{remark}{Remark}

\newcommand{\Addresses}{{
  \bigskip
  \footnotesize

  \textsc{Department of Mathematical Sciences, Norwegian University of Science and Technology,\\ 7491 Trondheim, Norway.}\par\nopagebreak
  \textit{E-mail addresses}: \texttt{eirik.berge@ntnu.no}
}}

\begin{document}
\pagenumbering{gobble}
\title{
  \Huge{A Primer on Coorbit Theory}
  
  \vspace{0.3cm}
  
  \Large{From Basics to Recent Developments}}
\author{\large{Eirik Berge}}
\date{}
\maketitle
\pagenumbering{arabic}

\begin{abstract}
Coorbit theory is a powerful machinery that constructs a family of Banach spaces, the so-called \textit{coorbit spaces}, from well-behaved unitary representations of locally compact groups. A core feature of coorbit spaces is that they can be discretized in a way that reflects the geometry of the underlying locally compact group. Many established function spaces such as \textit{modulation spaces}, \textit{Besov spaces}, \textit{Sobolev-Shubin spaces}, and \textit{shearlet spaces} are examples of coorbit spaces. \par
The goal of this survey is to give an overview of coorbit theory with the aim of presenting the main ideas in an accessible manner. Coorbit theory is generally seen as a complicated theory, filled with both technicalities and conceptual difficulties. Faced with this obstacle, we feel obliged to convince the reader of the theory's elegance. As such, this survey is a showcase of coorbit theory and should be treated as a stepping stone to more complete sources.
\end{abstract}


\section{Introduction}

Whenever a new mathematical theory is developed, one of two things usually happens: On the one hand, the theory might not be sufficiently interesting. Together with the failure to generate non-trivial results in well-established special cases, this signals a premature end. On the other hand, a newly developed theory might succeed in these endeavours. What follows is a period of flourishing, where researchers from related fields develop the theory to its fullest potential. However, there is a third and more disheartening possibility as well; the theory is wonderful in all regards but is largely left unnoticed by the mathematical community. This was the case for the theory of coorbit spaces, developed in the late '80s in a series of papers \cite{feichtinger1988unified, feichtinger1989banach1, feichtinger1989banach2} by Hans Georg Feichtinger and Karlheinz Gr\"{o}chenig. However, with the turn of the century, interest in coorbit spaces has been growing rapidly. This is due to a plethora of reasons, the most obvious one being the emergence of time-frequency analysis as a central topic in modern harmonic analysis. Many results in time-frequency analysis can be either proven or illuminated by the constructions in coorbit theory. \par
The goal of this survey is to provide an introduction to coorbit theory aimed at non-experts. We have tried to strike a balance between providing sufficient details, while at the same time prioritizing concepts over technicalities. The original papers on coorbit theory are, although insightful, admittedly difficult for novices to digest. More recent sources, e.g.\ \cite{renethesis, felix_thesis, dahlke2015harmonic}, are either not fully devoted to coorbit theory or include technicalities that distract most beginners from the core ideas. This is not intended as critique of the above sources as their main aim is to derive new results. In fact, we have the privilege of dwelling on pedagogical points precisely because we do not aim for novelty. We hope this survey can establish a natural starting point to learn coorbit theory for both students and researchers in neighboring fields.
\newpage
\noindent\textbf{Overview:} Before embarking, we give a brief overview of what coorbit theory is all about. This requires the usage of terminology that might be unfamiliar to the reader; if this causes bewilderment, then skip this part for now and return to it once you have finished reading Chapter \ref{sec: Chapter_2}. We begin with a unitary representation $\pi:G \to \mathcal{H}_{\pi}$ of a locally compact group $G$ on a Hilbert space $\mathcal{H}_{\pi}$. Consider the \textit{wavelet transform} \[\mathcal{W}_{g}:\mathcal{H}_{\pi} \to L^{\infty}(G), \quad \mathcal{W}_{g}(f)(x) := \langle f, \pi(x) g \rangle_{\mathcal{H}_{\pi}},\] where $g, f \in \mathcal{H}_{\pi}$ and $x \in G$. Under some assumptions on the representation $\pi$ and the element $g \in \mathcal{H}_{\pi}$, the transformation $\mathcal{W}_{g}$ is actually an isometry from $\mathcal{H}_{\pi}$ to the Hilbert space $L^{2}(G)$. The inner mechanics of coorbit theory deals with the following two points:
\begin{itemize}
    \item We construct a collection $\mathcal{C}o_{p}(G)$ of Banach spaces for each $1 \leq p \leq \infty$ called \textit{coorbit spaces}. Each space $\mathcal{C}o_{p}(G)$ contains the elements $f \in \mathcal{H}_{\pi}$ such that $\mathcal{W}_{g}(f)$ has a certain decay (depending on $p$) as a function on the group $G$. To make the definition of the coorbit spaces $\mathcal{C}o_{p}(G)$ precise, we will first need to extend the wavelet transform to the distributional setting.
    \item By picking a suitable \textit{atom} $g \in \mathcal{H}_{\pi}$ we can generate any $f \in \mathcal{C}o_{p}(G)$ through the formula \begin{equation}
        \label{atomic_decomposition_introduction}
        f = \sum_{i \in I}c_{i}(f)\pi(x_i)g,
    \end{equation}
    where $\{x_{i}\}_{i \in I} \subset G$ is a collection of carefully chosen points and $(c_i)_{i \in I}$ are coefficients that depend linearly on $f$. This systematic decomposition is known as an \textit{atomic decomposition}. Intuitively, we decompose each element $f \in \mathcal{C}o_{p}(G)$ into its atomic parts relative to the chosen atom $g \in \mathcal{H}_{\pi}$. The selection of the points $\{x_{i}\}_{i \in I} \subset G$ depends heavily on the structure of $G$, giving the theory a geometric flavor.
\end{itemize}

Two classes of coorbit spaces that have appeared prominently in the literature are the \textit{(homogeneous) Besov spaces} in classical harmonic analysis and the \textit{modulation spaces} in time-frequency analysis. One can obtain a deeper appreciation for these seemingly different spaces by realizing that they are both special cases of the coorbit space machinery. These two examples will be returned to time and time again to illustrate the concepts presented. 
\\\\
\textbf{Existing Literature:} There are sources in the literature that deal with coorbit spaces from a somewhat expository viewpoint. We emphasize three of them as they deserve a special mention:
\begin{itemize}
    \item The Ph.D. thesis \cite{felix_thesis} of Felix Voigtlaender has been very helpful, especially for technical aspects of the survey. Although the first chapters of \cite{felix_thesis} are more advanced than this survey (e.g.\ they deals with quasi-Banach spaces), it nevertheless introduces all the main ideas in a clear manner.
    \item The book \cite{dahlke2015harmonic} is a collection of survey papers written by various authors. Especially Chapter 2 (written by Filippo De Mari and Ernesto De Vito) and Chapter 3 (written by Stephan Dahlke, S\"{o}ren H\"{a}user, Gabriele Steidl, and Gerd Teschke) have been helpful for comprehending the basics of coorbit theory. 
    \item The paper \cite{christensen1996} is mostly an expository account of different aspects of coorbit theory. It is both well-written and useful, although it assumes more background knowledge from the reader than we do. A drawback is that \cite{christensen1996} has, due to its publication date, no modern examples and directions in coorbit theory.
\end{itemize}

As coorbit theory is a popular topic nowadays, there have been several advances of the theory in the last five years. Most of these topics are not discussed outside of their respective research papers. It is our belief that the community would benefit from having these results more easily available. We will go through some of the recent developments in Section~\ref{sec: A Kernel Theorem for Coorbit Spaces} and Chapter \ref{sec: Chapter 4}. In Section~\ref{sec: Where_To_Go_Next?} we give references to many recent works on coorbit theory.
\\\\
\textbf{Unconventional Topics:}
\begin{itemize}
\item \textit{Reproducing Kernel Hilbert Spaces}: This is included in Section \ref{sec: Reproducing_Kernel_Hilbert_Spaces} since the wavelet transform automatically produces reproducing kernel Hilbert spaces, see Proposition \ref{spaces_are_RKHS}. These reproducing kernel Hilbert spaces have received interest recently in \cite{berge2020interpolation, luef2020wiener, romero2020dual, cont_wavelet_transform}. The reproducing kernel approach also illuminates the reproducing formula in Theorem \ref{reproducing_formula}, which is central to the theory. It should be noted that reproducing kernel Hilbert spaces are often implicitly present in works on coorbit theory.
\item \textit{Large Scale Geometry}: We have included certain definitions from large scale geometry in Section \ref{sec: Atomic Decompositions} as this provides a convenient language for discussing discretizations. Large scale geometry has had little intersection with coorbit theory, except for in \cite{renethesis} where it is utilized successfully. Both \cite{renethesis} and the papers \cite{berge2019modulation, eiriklargescale} uses large scale geometry to analyze \textit{decomposition spaces}, which is a family of spaces that are related to coorbit spaces. We hope that large scale geometry can provide a conceptual framework that might bring new ideas to the table.
\end{itemize}
We have chosen to omit \textit{Wiener amalgam spaces} from the survey. This choice is a difficult one; although Wiener amalgam spaces are a useful tool, they are also a conceptional hurdle for some and not always needed in practical applications of coorbit theory. We refer the reader to \cite{feichtinger1980banach} and the survey \cite{heil2003introduction} for more details on Wiener amalgam spaces. 
\\\\
\textbf{Outline: }
\begin{itemize}
    \item \textbf{Chapter \ref{sec: Chapter_2}:} We introduce locally compact groups and unitary representations in Section~\ref{sec: Locally_Compact_Groups} and Section~\ref{sec: Unitary_rep_theory}, respectively. In addition to fixing notation, this allows us to require few prerequisites from the reader. The wavelet transform is a central player in coorbit theory and is introduced in Section~\ref{sec: Unitary_rep_theory}. We go through the orthogonality relations for the wavelet transform in Section~\ref{sec: Orthogonality_of_the_Wavelet_Transform}. In Section~\ref{sec: Reproducing_Kernel_Hilbert_Spaces} we review reproducing kernel Hilbert spaces and show that such spaces naturally arise when considering the wavelet transform. Finally, we derive the reproducing and reconstruction formulas for the wavelet transform in Section~\ref{sec: Inversion_Formula}. 
    
    \item \textbf{Chapter \ref{sec: Construction_of_Smoothness_Spaces}:} We introduce the integrable setting in Section~\ref{sec: Test_Spaces_and_Distributional_Spaces} and extend the wavelet transform to the distributional level in Section~\ref{sec: Reservoirs_and_the_Extended_Wavelet_Transform}. This allows us to define the coorbit spaces in Section~\ref{sec: Coorbit_Spaces} in a rigorous manner. The basic properties of the coorbit spaces are derived in Section~\ref{sec: Basic_properties_of_the_coorbit_spaces} with the help of the \textit{correspondence principle} given in Theorem \ref{corresponance_principle}. In Section~\ref{sec: Weighted_Versions} we discuss weighted coorbit spaces. We show that the coorbit spaces have extraordinary sampling properties in Section~\ref{sec: Atomic Decompositions} through a general procedure called \textit{atomic decompositions}. Terminology borrowed from large scale geometry will be used to make the main result in Theorem~\ref{big_discretization_theorem} more transparent. Finally, we discuss Banach frames and a kernel theorem for coorbit theory in respectively Section~\ref{sec: Banach_Frames} and Section~\ref{sec: A Kernel Theorem for Coorbit Spaces}.
    
    \item \textbf{Chapter \ref{sec: Chapter 4}:} We solidify the results presented in previous chapters by giving non-trivial examples of the theory. This includes shearlet spaces in signal analysis in Section~\ref{sec: Shearlet_Spaces}, Bergman spaces in complex analysis in Section~\ref{sec: Bergman_Spaces_and_the_Blaschke_Group}, and coorbit spaces built on nilpotent Lie groups in Section~\ref{sec: Coorbit_Spaces_on_Nilpotent_Lie_Groups}. We end in Section~\ref{sec: Where_To_Go_Next?} by giving references to recent developments related to embeddings between coorbit spaces and generalizations of coorbit theory. 
\end{itemize}

\textbf{Acknowledgments: }I have received advice and concrete suggestions from many researchers throughout the writing process. I would in particular like to thank Stine Marie Berge, Franz Luef, and Felix Voigtlaender for illuminating discussions and helpful comments. Finally, I would like to express my gratitude to everyone who participated in the seminar course I gave on coorbit theory at the Norwegian University of Science and Technology during the fall of 2020.

\section{Starting Out}
\label{sec: Chapter_2}

We start by giving an overview of preliminary topics, namely locally compact groups, unitary representations, and basic properties of the (generalized) wavelet transform. Most of this material is fairly standard, and is mainly collected from the books \cite{folland2016course, grochenig2001foundations, dahlke2015harmonic, fuhr2005abstract, deitmar2014principles}. We aim for a suitable generality and present concrete examples as we go along. 

\subsection{Prelude on Locally Compact Groups}
\label{sec: Locally_Compact_Groups}

The first order of business is to get acquainted with locally compact groups.

\begin{definition}
A \textit{locally compact group} is a locally compact Hausdorff topological space $G$ that is simultaneously a group such that the multiplication and inversion maps \[(x,y) \longmapsto xy, \qquad x \longmapsto x^{-1}, \quad x,y \in G,\] are continuous.
\end{definition}

For us, the object of main importance on a locally compact group is the left Haar measure: Recall that a \textit{Radon measure} is Borel measure that is finite
on compact sets, inner regular on
open sets, and outer regular on all Borel sets. Do not worry if you are rusty on the measure-theoretic nonsense; we will never use these technical conditions explicitly. The important point is that each locally compact group $G$ can be equipped with a unique (up to a positive constant) \textit{left-invariant} Radon measure $\mu_{L}$, that is, $\mu_{L}$ satisfies $\mu_{L}(xE) = \mu_{L}(E)$ for all $x \in G$ and every Borel set $E \subset G$. We call the measure $\mu_{L}$ the \textit{left Haar measure} of the group $G$. The existence of the left Haar measure implies that any locally compact group is canonically equipped with a measure-theoretic setting. \par 
As the terminology indicates, there is also a \textit{right Haar measure} $\mu_{R}$ on any locally compact group. How much the two measures $\mu_{L}$ and $\mu_{R}$ deviate is captured by the \textit{modular function} $\Delta:G \to (0, \infty)$ defined as follows: For $x \in G$ the measure $\mu_{x}(E) := \mu_{L}(Ex)$ is again a left-invariant Radon measure. Therefore, the uniqueness of the left Haar measure implies the existence of a number $\Delta(x) \in (0, \infty)$ such that \[\mu_{x}(E) = \Delta(x)\mu_{L}(E),\] for every Borel set $E \subset G$. It is straightforward to see that $\mu_{L} = \mu_{R}$ precisely when $\Delta \equiv 1$. Motivated by this observation, groups where $\mu_{L} = \mu_{R}$ are called \textit{unimodular}. When this is the case, we use the abbreviation $\mu := \mu_{L} = \mu_{R}$ and refer to $\mu$ as the \textit{Haar measure} on the group $G$. It is clear that commutative locally compact groups are unimodular. Moreover, locally compact groups that are either compact or discrete are also unimodular, see \cite[Chapter 2.4]{folland2016course}.

\begin{example}
The reader has surely seen plenty of locally compact groups previously. Two elementary ones are $\mathbb{R}^n$ with the usual vector sum and $\mathbb{R}^{*} := \mathbb{R} \setminus \{0\}$ with the usual product. On $\mathbb{R}^n$, the Haar measure is the Lebesgue measure $dx$, while on $\mathbb{R}^{*}$ the Haar measure is $dx/|x|$. To exemplify the last claim, we see for $E = (r,s)$ with $s > r > 0$ and $x > 0$ that \[\mu(xE) = \int_{xr}^{xs} \, \frac{dt}{t} = \log(xs) - \log(xr) = \log\left(\frac{s}{r}\right) = \mu(E).\]
\end{example}

\begin{example}
\label{affine_group_first_example}
There are many locally compact groups of interest that are not unimodular. As an example, we consider the \textit{(full) Affine group} $\textrm{Aff} = \mathbb{R} \times \mathbb{R}^{*}$ with the group multiplication \[(b,a) \cdot (b',a') := (ab' + b,aa'), \qquad (b,a), \, (b',a') \in \textrm{Aff}.\] The group operation models the composition of affine maps, and can equivalently be realized as $2 \times 2$ matrices of the form \[\begin{pmatrix} a & b \\ 0 & 1 \end{pmatrix}, \qquad (b,a) \in \textrm{Aff},\] where the group operation is matrix multiplication. Notice that the group operation is not commutative. Moreover, the affine group is not unimodular: The reader can verify that the left and right Haar measures on $\mathrm{Aff}$ are respectively given by \[\mu_{L}(b,a) = \frac{db \, da}{a^2}, \qquad \mu_{R}(b,a) = \frac{db \, da}{|a|}.\]
\end{example}

\begin{remark}
If you find yourself in the situation where you have a locally compact group $G$ but no obvious candidate for a Haar measure, then do not despair; there are several ways of constructing the Haar measure on many locally compact groups. We refer the interested reader to \cite[Proposition~2.21]{folland2016course} for a concrete example.
\end{remark}

For a locally compact group G, we can form the spaces $L^{p}(G)$ for $1 \leq p < \infty$ consisting of equivalence classes of measurable functions $f:G \to \mathbb{C}$ such that \[\|f\|_{L^{p}(G)} := \left(\int_{G}|f(x)|^{p} \, d\mu_{L}(x) \right)^{\frac{1}{p}} < \infty.\] The case $p = \infty$ also has the obvious extension from the familiar Euclidean case. For locally compact groups that are not unimodular, some authors use the notation $L^{p}(G,\mu_{L})$ for clarity. However, we will always consider the left Haar measure, and thus boldly use the abbreviated notation $L^{p}(G)$. The spaces $L^{p}(G)$ are Banach spaces for all $1 \leq p \leq \infty$. Moreover, when $p = 2$ we even have a Hilbert space structure given by the inner product \[\langle f, g \rangle_{L^{2}(G)} := \int_{G} f(x)\overline{g(x)} \, d\mu_{L}(x).\]\par
We have for each $y \in G$ the \textit{left-translation operator} $L_{y}$ given by $L_{y}f(x) := f(y^{-1}x)$ for $x \in G$. The reason for the inverse is so that we have $L_{y} \circ L_{z} = L_{yz}$ for $y,z \in G$. This detail is important when we study unitary representations in Section \ref{sec: Unitary_rep_theory}. For similar reasons, we define for each $y \in G$ the \textit{right-translation operator} $R_{y}$ by the formula $R_{y}f(x) := f(xy)$ for $x \in G$.

\begin{definition}
For $f,g \in L^{1}(G)$ we can form the \textit{convolution} between $f$ and $g$ given by \[f *_{G} g(x) := \int_{G} f(y)g(y^{-1}x) \, d\mu_{L}(y).\]
\end{definition}
Notice that, in contrast with the usual convolution of functions on $\mathbb{R}^n$, the convolution is generally non-commutative. In fact, the convolution is commutative precisely when the group operation on $G$ is commutative \cite[Theorem 1.6.4]{deitmar2014principles}. Moreover, it follows from \cite[Corollary~20.14]{hewitt2012abstract} that the convolution inequality 
\begin{equation*}
    \|f *_G g \|_{L^{p}(G)} \leq \| f\|_{L^{1}(G)} \| g\|_{L^{p}(G)}
\end{equation*}
is valid for all $1 \leq p \leq \infty$, $g \in L^{p}(G)$, and $f \in L^{1}(G)$. 

\begin{example}
\label{reduced_heisenberg_group_example}
A non-commutative group that will be of central importance for us is the \textit{(full) Heisenberg group} $\mathbb{H}^{n}$. As a set we have $\mathbb{H}^{n} = \mathbb{R}^{n} \times \mathbb{R}^n \times \mathbb{R}$, while the group multiplication is given by \[\big(x,\omega,t \big) \cdot \big(x',\omega',t'\big) := \left(x + x', \omega + \omega', t + t' + \frac{1}{2}(x' \omega - x \omega')\right).\] Although what we have described is strictly speaking one group for each dimension $n$, we collectively refer to these groups as the Heisenberg group for simplicity. In Section \ref{sec: Orthogonality_of_the_Wavelet_Transform} we will use a different realization of the Heisenberg group due to integrability issues. The Heisenberg group is unimodular and the Haar measure on $\mathbb{H}^{n}$ is the usual Lebesgue measure on $\mathbb{R}^{2n+1}$. We refer the reader to \cite{howe1980role} for an excellent exposition on the ubiquity of the Heisenberg group in harmonic analysis.  
\end{example}

\begin{example}
When working with locally compact groups, it is advantageous to have both continuous and discrete examples in mind. Most discrete examples arise from letting $G$ be any countable group with the discrete topology. Let us briefly consider $G = \mathbb{Z}$ to see what the convolution looks like in this case: The Haar measure on $\mathbb{Z}$ is the counting measure. It is common to use the notation $l^{p}(\mathbb{Z}) := L^{p}(\mathbb{Z})$ for all $1 \leq p \leq \infty$. Per convention, we use sequence notation $a = (a_{n})_{n \in \mathbb{Z}}$ with $a_{n} := a(n)$ for functions $a: \mathbb{Z} \to \mathbb{C}$. The convolution between two elements $a,b \in l^{1}(\mathbb{Z})$ is precisely the well-known \textit{Cauchy product} of sequences given by \[(a *_{\mathbb{Z}} b)_{n} := \sum_{m = -\infty}^{\infty}a_m b_{n - m}.\]
\end{example}

\subsection{Unitary Representations and the Wavelet Transform}
\label{sec: Unitary_rep_theory}

We will now consider unitary representations of locally compact groups. This will give rise to the (generalized) wavelet transform that we will examine closely. Ultimately, we use the wavelet transform to construct the coorbit spaces in Chapter \ref{sec: Construction_of_Smoothness_Spaces}. Given a Hilbert space $\mathcal{H}$ we let $\mathcal{U}(\mathcal{H})$ denote the group of all unitary operators from $\mathcal{H}$ to itself.

\begin{definition}
Let $G$ be a locally compact group and let $\mathcal{H}_{\pi}$ be a Hilbert space. A \textit{unitary representation} of $G$ on $\mathcal{H}_{\pi}$ is a group homomorphism $\pi:G \to \mathcal{U}(\mathcal{H}_{\pi})$ such that the transformation
\begin{equation}
\label{continuity_strong}
    G \ni x \longmapsto \pi(x)g \in \mathcal{H}_{\pi}
\end{equation}
is continuous for all $g \in \mathcal{H}_{\pi}$.
\end{definition}

It turns out that the continuity requirement \eqref{continuity_strong} is equivalent to the seemingly weaker requirement that 
\begin{equation}
\label{wavelet_transform_general}
    G \ni x \longmapsto \mathcal{W}_{g}f(x) := \langle f, \pi(x)g \rangle
\end{equation}
is a continuous function on $G$ for all $f,g \in \mathcal{H}_{\pi}$. The function $\mathcal{W}_{g}f$ is called the \textit{(generalized) wavelet transform} of $f$ with respect to $g$. Hence $\mathcal{W}_{g}f:G \to \mathbb{C}$ is a continuous function by assumption whenever we have a unitary representation. Moreover, we have that $\mathcal{W}_{g}f$ is a bounded function on $G$ since 
\[|\mathcal{W}_{g}f(x)| = |\langle f, \pi(x)g \rangle| \leq \|f\|\|\pi(x)g\| = \|f\|\|g\|, \qquad x \in G.\]
We often take the view that $g \in \mathcal{H}_{\pi}$ is fixed and consider the map $\mathcal{W}_{g}:\mathcal{H}_{\pi} \to C_{b}(G)$ sending $f$ to $\mathcal{W}_{g}f$, where $C_{b}(G)$ denotes the set of complex valued continuous functions on $G$ that are bounded. The wavelet transform has a central place in coorbit theory, and much of the theory revolves around understanding subtle properties of this transformation. 

\begin{example}
\label{example_regular_representation}
An example of a unitary representation on any locally compact group $G$ is the \textit{left regular representation} $L: G \to \mathcal{U}(L^{2}(G))$ given by \[L(x)f(y) := L_{x}f(y) = f(x^{-1}y),\] for $x,y \in G$ and $f \in L^{2}(G)$. Note that $L_{xy} = L_{x} \circ L_{y}$ and $L_{x}^{-1} = L_{x^{-1}}$. Hence the fact that $L_{x}$ is unitary follows from the computation \[\|L_{x}f\|_{L^{2}(G)}^2 = \int_{G}|L_{x}f(y)|^2 \, d\mu_{L}(y) = \int_{G}|f(x^{-1}y)|^2 \, d\mu_{L}(y) = \int_{G}|f(y)|^2 \, d\mu_{L}(y) = \|f\|_{L^{2}(G)}^2.\] For the continuity assertion \eqref{continuity_strong}, we refer the reader to \cite[Proposition 2.42]{folland2016course}. 
\end{example}

\begin{definition}
Let $\pi:G \to \mathcal{U}(\mathcal{H}_{\pi})$ be a unitary representation of a locally compact group $G$. 
\begin{itemize}
    \item We say that a closed subspace $\mathcal{M} \subset \mathcal{H}_{\pi}$ is an \textit{invariant subspace} if $\pi(x)g \in \mathcal{M}$ for all $g \in \mathcal{M}$ and $x \in G$. When this happens, the restriction $\pi|_{\mathcal{M}}$ is a unitary representation of $G$ on $\mathcal{M}$ and we call $\pi|_{\mathcal{M}}:G \to \mathcal{U}(\mathcal{M})$ a \textit{subrepresentation} of $\pi$.
    \item If there are no non-trivial (other than $\{0\}$ and $\mathcal{H}_{\pi}$) invariant subspaces of $\mathcal{H}_{\pi}$, then $\pi$ is called \textit{irreducible.} Otherwise, we say that $\pi$ is \textit{reducible}.
\end{itemize}
\end{definition}

For any unitary representation $\pi:G \to \mathcal{U}(\mathcal{H}_\pi)$ we have for $f,g \in \mathcal{H}_{\pi}$ and $x,y \in G$ that
\begin{equation}
\label{wavelet_intertwines_left_reg}
    \mathcal{W}_{g}(\pi(y)f)(x) = \langle \pi(y)f, \pi(x)g \rangle = \langle f, \pi(y^{-1})\pi(x)g \rangle = \mathcal{W}_{g}(f)(y^{-1}x) = L_{y}\left[\mathcal{W}_{g}(f)\right](x).
\end{equation}
The simple calculation \eqref{wavelet_intertwines_left_reg} should not be underestimated; it shows that the wavelet transform gives us a way to relate the representation $\pi$ and the left regular representation $L$ in Example \ref{example_regular_representation}. This notion is formalized in the following definition.

\begin{definition}
Let $G$ be a locally compact group and consider two unitary representations $\pi:G \to \mathcal{U}(\mathcal{H}_{\pi})$ and $\tau:G \to \mathcal{U}(\mathcal{H}_{\tau})$. We say that a bounded linear operator $T:\mathcal{H}_{\pi} \to \mathcal{H}_{\tau}$ is an \textit{intertwiner} between $\pi$ and $\tau$ if $T \circ \pi(x) = \tau(x) \circ T$ for every $x \in G$. If $T$ is additionally a unitary operator, then we refer to $T$ as a \textit{unitary intertwiner}. If a unitary intertwiner exists between $\pi$ and $\tau$, then $\pi$ and $\tau$ are called \textit{equivalent}.
\end{definition}

If we are only considering one unitary representation $\pi:G \to \mathcal{U}(\mathcal{H}_{\pi})$, then a bounded linear operator $T: \mathcal{H}_{\pi} \to \mathcal{H}_{\pi}$ satisfying $T \circ \pi(x) = \pi(x) \circ T$ is simply referred to as a \textit{(unitary) intertwiner} of $\pi$. We leave it to the reader to verify that if $\pi$ is an irreducible unitary representation and $T$ is a unitary intertwiner between $\pi$ and another unitary representation $\tau$, then $\tau$ is also irreducible. \par 
It is tempting, but slightly premature, to reformulate \eqref{wavelet_intertwines_left_reg} in the following way: The wavelet transform $\mathcal{W}_{g}$ is, for any choice of $g \in \mathcal{H}_{\pi}$, an intertwiner between $\pi$ and the left regular representation $L$ given in Example \ref{example_regular_representation}. The problem is that in general the wavelet transform $\mathcal{W}_{g}f$ is not in $L^{2}(G)$ as the following example shows.

\begin{example}
Consider the left regular representation $L:\mathbb{R} \to \mathcal{U}(L^{2}(\mathbb{R}))$ on $G = \mathbb{R}$. Then for $f,g \in L^{2}(\mathbb{R})$ and $x \in \mathbb{R}$ the wavelet transform has the form \[\mathcal{W}_{g}f(x) = \int_{-\infty}^{\infty}f(y)g(y - x) \, dy = f * \check{g}(x),\]
where $\check{g}(x) := g(-x)$. The space $L^{2}(\mathbb{R})$ is not closed under convolution: Let \[f(x) = g(x) = \mathcal{F}\left(e^{-\omega^2}|\omega|^{-\frac{1}{3}}\right)(x),\] where $\mathcal{F}$ denotes the Fourier transform. Then one can check that $f,g \in L^{2}(\mathbb{R})$ and $\mathcal{W}_{g}f \not\in L^{2}(\mathbb{R})$.
\end{example}

We will in Section \ref{sec: Orthogonality_of_the_Wavelet_Transform} work with additional assumptions on the representation $\pi$ and the fixed vector $g \in \mathcal{H}_{\pi}$ so that $\mathcal{W}_{g}f \in L^{2}(G)$ for all $f \in \mathcal{H}_{\pi}$. In that case, a natural question emerges that we will answer in Section \ref{sec: Orthogonality_of_the_Wavelet_Transform}: 

\begin{center}
    \begin{adjustwidth}{40pt}{40pt}
        \textbf{Q:} Is $\mathcal{W}_{g}$ a unitary intertwiner between $\pi$ and some subrepresentation of the left regular representation $L$?
    \end{adjustwidth}
\end{center}

\begin{example}
\label{example_with_the_STFT}
Let us revisit the Heisenberg group $\mathbb{H}^{n}$ in Example \ref{reduced_heisenberg_group_example} and describe its irreducible unitary representations. First of all, we have the family of one-dimensional representations of $\mathbb{H}^{n}$ given by \[\chi_{\alpha,\beta}(x,\omega,t) := e^{2 \pi i(\alpha x + \beta \omega)} \in \mathcal{U}(\mathbb{C}), \qquad \alpha,\beta \in \mathbb{R}^{n}, \, (x,\omega, t) \in \mathbb{H}^{n}.\] The \textit{central characters} $\chi_{\alpha,\beta}$ are obviously irreducible, unitary, and non-equivalent. We refer the reader to \cite[Chapter 9.2]{grochenig2001foundations} for an explanation of why $\chi_{\alpha, \beta}$ are called the central characters of $\mathbb{H}^{n}$. \par
Let $T_{x}$ and $M_{\omega}$ be respectively the \textit{translation operator} and \textit{modulation operator} on $L^{2}(\mathbb{R}^{n})$ given by
\begin{equation}
\label{time_shift_and_frequency_shift}
    T_{x}f(y) := f(y-x), \qquad M_{\omega}f(y) := e^{2 \pi i y \omega}f(y), \quad x,y,\omega \in \mathbb{R}^{n}.
\end{equation}
These operators can be combined to form the \textit{Schr\"{o}dinger representation} $\rho: \mathbb{H}^{n} \to \mathcal{U}(L^{2}(\mathbb{R}^{n}))$ given by 
\begin{equation}
\label{Schrodinger_representation}
    \rho(x,\omega,t)f(y) := e^{2 \pi i t}e^{\pi i x \omega}T_{x}M_{\omega}f(y).
\end{equation}
It can be verified that the Schr\"{o}dinger representation is an irreducible unitary representation of $\mathbb{H}^{n}$, see \cite[Theorem 9.2.1]{grochenig2001foundations}. Moreover, one can generate new non-equivalent irreducible unitary representations by dilating the Schr\"{o}dinger representation \[\rho_{\lambda}(x,\omega,t) := \rho(\lambda x, \omega, \lambda t), \qquad \lambda \in \mathbb{R}\setminus \{0\}.\]
And that's it! The Stone-von Neumann theorem \cite[Theorem 9.3.1]{grochenig2001foundations} states that any irreducible unitary representation of $\mathbb{H}^{n}$ is equivalent to either $\chi_{\alpha,\beta}$ for some $\alpha, \beta \in \mathbb{R}^{n}$ or $\rho_{\lambda}$ for some $\lambda \in \mathbb{R}\setminus \{0\}.$
\end{example}

The following result shows a fundamental relationship between irreducible unitary representations and intertwiners. 

\begin{lemma}[Schur's Lemma]
Let $\pi:G \to \mathcal{U}(\mathcal{H}_{\pi})$ be a unitary representation. Then $\pi$ is irreducible if and only if every intertwiner of $\pi$ is a constant multiple of the identity $\textrm{Id}_{\mathcal{H}_{\pi}}$.
\end{lemma}

We refer the reader to \cite[Theorem 3.5]{folland2016course} for a proof of Schur's Lemma. One of the main uses of Schur's Lemma is showing that certain irreducible representations are impossible. The following result illustrates this.

\begin{corollary}
\label{corollary_abelian_reps}
Let $\pi:G \to \mathcal{U}(\mathcal{H}_{\pi})$ be a unitary representation of a commutative locally compact group $G$. If $\pi$ is irreducible, then $\mathrm{dim}(\mathcal{H}_{\pi}) = 1$.
\end{corollary}
\begin{proof}
Notice that for all $x,y \in G$ we have \[\pi(x)\pi(y) = \pi(xy) = \pi(yx) = \pi(y)\pi(x).\]
Thus $\pi(x) \in \mathcal{U}(\mathcal{H}_{\pi})$ is in fact a unitary intertwiner of $\pi$. Hence Schur's Lemma implies that $\pi(x) = C_{x} \cdot \textrm{Id}_{\mathcal{H}_{\pi}}$ for all $x \in G$, where $C_x$ is a constant dependent on $x$. However, it is now clear that \textit{any} closed subspace of $\mathcal{H}_{\pi}$ is invariant. This can only be the case, under the assumption of irreducibility, when $\mathcal{H}_{\pi}$ does not have any closed subspaces other than $\{0\}$ and $\mathcal{H}_{\pi}$.  
\end{proof}

Let us try to construct an invariant subspace of a unitary representation $\pi:G \to \mathcal{U}(\mathcal{H}_{\pi})$. Fix a non-zero vector $g \in \mathcal{H}_{\pi}$ and form the subspace
\begin{equation*}
\mathcal{M}_{g} := \overline{\textrm{span} \left\{\pi(x)g \, : \, x \in G \right\}} \subset \mathcal{H}_{\pi}.
\end{equation*}
Notice that $\mathcal{M}_{g}$ is a closed subspace of $\mathcal{H}_{\pi}$ that is non-trivial since $g = \pi(e)g \in \mathcal{M}_{g}$, where $e \in G$ is the identity element of $G$. Moreover, $\mathcal{M}_{g}$ is clearly invariant under the action of $\pi$. We call $\mathcal{M}_{g}$ the \textit{cyclic subspace} generated by $g \in \mathcal{H}_{\pi}$. If $\mathcal{M}_{g} = \mathcal{H}_{\pi}$, then the vector $g$ is said to be \textit{cyclic}. If this is not the case, then the representation $\pi$ is reducible as $\mathcal{M}_{g}$ would be a non-trivial invariant subspace. Conversely, assume that every non-zero vector $g \in \mathcal{H}_{\pi}$ is cyclic and let $\mathcal{M} \subset \mathcal{H}_{\pi}$ be a non-trivial invariant subspace. Fix a non-zero $g \in \mathcal{M}$ and notice that $\mathcal{M}_{g} \subset \mathcal{M}$. Since $g$ is cyclic this forces $\mathcal{M} = \mathcal{H}_{\pi}$ so that $\pi$ is irreducible. We summarize this discussion for later reference in the following proposition.

\begin{proposition}
\label{proposition_cyclic_irreducible}
A unitary representation $\pi:G \to \mathcal{U}(\mathcal{H}_{\pi})$ is irreducible precisely when every non-zero vector $g \in \mathcal{H}_{\pi}$ is cyclic.
\end{proposition}

The following result shows that cyclic vectors are of central importance for the wavelet transform.

\begin{lemma}
\label{injection_lemma}
Consider a unitary representation $\pi:G \to \mathcal{U}(\mathcal{H}_{\pi})$ and fix a non-zero vector $g \in \mathcal{H}_{\pi}$. The wavelet transform $\mathcal{W}_{g}:\mathcal{H}_{\pi} \to C_{b}(G)$ is injective if and only if $g$ is a cyclic vector. 
\end{lemma}

\begin{proof}
Assume by contradiction that $g$ is a cyclic vector and $\mathcal{W}_{g}$ is not injective. Pick $f \in \mathcal{H}_{\pi} \setminus \{0\}$ such that $\mathcal{W}_{g}f$ is the zero function on $G$, that is, \[\mathcal{W}_{g}f(x) = \langle f, \pi(x)g \rangle = 0,\] for all $x \in G$. This implies that $f$ is orthogonal to the cyclic subspace $\mathcal{M}_{g}$. In particular, $\mathcal{M}_{g} \neq \mathcal{H}_{\pi}$ and we have a contradiction. Conversely, assume that $g$ is not cyclic so that $\mathcal{M}_{g} \neq \mathcal{H}_{\pi}$. By picking $f \in \mathcal{M}_{g}^{\perp} \setminus \{0\}$ we have that $\langle f, \pi(x)g \rangle = 0$ for all $x \in G$. Hence $\mathcal{W}_{g}:\mathcal{H}_{\pi} \to C_{b}(G)$ is not injective.
\end{proof}

\subsection{Square Integrability and Orthogonality}
\label{sec: Orthogonality_of_the_Wavelet_Transform}

We want to examine the wavelet transform $\mathcal{W}$ given in \eqref{wavelet_transform_general} in more detail. It is instructive to look at a concrete example first to see what we might expect.

\begin{example}
Let us consider the Schr\"{o}dinger representation $\rho$ of the Heisenberg group $\mathbb{H}^{n}$ given in \eqref{Schrodinger_representation}. The wavelet transform corresponding to this representation is given by 
\begin{equation}
\label{generalized_wavelet_Schrodinger}
    \mathcal{W}_{g}f(x,\omega,t) = \langle f, \rho(x,\omega,t)g \rangle = e^{-2\pi i t}e^{-\pi i x \omega} \langle f, T_{x}M_{\omega}g \rangle = e^{-2\pi i t}e^{\pi i x \omega} \langle f, M_{\omega}T_{x}g \rangle,
\end{equation}
for $f,g \in L^{2}(\mathbb{R}^{n})$. We can recognize the term $\langle f, M_{\omega}T_{x}g \rangle$ as the \textit{short-time Fourier transform} (STFT), which is usually denoted by \begin{equation*}
    V_{g}f(x,\omega) := \langle f, M_{\omega}T_{x}g \rangle = \int_{\mathbb{R}^n}f(t)\overline{g(t-x)}e^{-2 \pi i t \omega} \, dt.
\end{equation*}
Hence the wavelet transform for the Schr\"{o}dinger representation is, up to a phase factor, the short-time Fourier transform. The STFT satisfies two important properties:

\begin{description}
    \item[Orthogonality:] For $f_1,f_2,g_1,g_2 \in L^{2}(\mathbb{R}^{n})$ we have the orthogonality relation 
    \begin{equation}
    \label{STFT_orthogonality}
        \langle V_{g_1}f_1, V_{g_2}f_2 \rangle_{L^{2}(\mathbb{R}^{2n})} = \langle f_1, f_2 \rangle_{L^{2}(\mathbb{R}^{n})} \overline{\langle g_1, g_2 \rangle}_{L^{2}(\mathbb{R}^{n})}.
    \end{equation}
    \item[Reconstruction:] Fix $g \in L^{2}(\mathbb{R}^{n})$ with $\|g\|_{L^{2}(\mathbb{R}^{n})} = 1$. Given any $f \in L^{2}(\mathbb{R}^{n})$, we can reconstruct $f$ from $V_{g}f$ through the formula 
    \begin{equation}
    \label{reconstruction_for_the_STFT}
        \langle f, h \rangle_{L^{2}(\mathbb{R}^{n})} = \int_{\mathbb{R}^{2n}}V_{g}f(x,\omega) \overline{V_{g}h(x,\omega)} \, dx \,d\omega,
    \end{equation}
    for any $h \in L^{2}(\mathbb{R}^{n})$.
\end{description}
The proofs can be found in \cite[Theorem 3.2.1]{grochenig2001foundations} and \cite[Corollary 3.2.3]{grochenig2001foundations}, respectively.
\end{example}

We postpone discussing the reconstruction property \eqref{reconstruction_for_the_STFT} to Section \ref{sec: Inversion_Formula}. It turns out that the STFT case is a best case scenario; not all generalized wavelet transforms exhibit such a simple orthogonality relation. From \eqref{STFT_orthogonality} we see that $V_{g}:L^{2}(\mathbb{R}^{n}) \to L^{2}(\mathbb{R}^{2n})$ is an isometry for any normalized $g \in L^{2}(\mathbb{R}^{n})$. Generalizing this observation, we would like to answer the following question in this section:
\begin{center}
    \begin{adjustwidth}{40pt}{40pt}
        \textbf{Q:} Under which conditions on a general unitary representation $\pi:G \to \mathcal{U}(\mathcal{H}_{\pi})$ and a non-zero vector $g \in \mathcal{H}_{\pi}$ can we ensure that the generalized wavelet transform $\mathcal{W}_{g}:\mathcal{H}_{\pi} \to L^{2}(G)$ is an isometry?
    \end{adjustwidth}
\end{center}
Notice that this question is precisely the same as the question we asked in Section \ref{sec: Unitary_rep_theory} regarding whether $\mathcal{W}_{g}$ is a unitary intertwiner between $\pi$ and a subrepresentation of the left regular representation. Given a unitary representation $\pi:G \to \mathcal{U}(\mathcal{H}_{\pi})$ we first of all need that $\mathcal{W}_{g}$ is injective. By Proposition \ref{proposition_cyclic_irreducible} and Lemma \ref{injection_lemma} this will be satisfied for all non-zero vectors $g \in \mathcal{H}_{\pi}$ whenever $\pi$ is irreducible. Henceforth we will require that $\pi$ is irreducible. Secondly, we need a condition on $g$ to ensure that $\mathcal{W}_{g}f \in L^{2}(G)$ for all $f \in \mathcal{H}_{\pi}$. 

\begin{definition}
Let $\pi:G \to \mathcal{U}(\mathcal{H}_{\pi})$ be an irreducible unitary representation. We say that a non-zero vector $g \in \mathcal{H}_{\pi}$ is \textit{square integrable} if $\mathcal{W}_{g}g \in L^{2}(G)$. Explicitly, we require that \[\int_{G} |\langle g, \pi(x)g \rangle |^2 \, d\mu_{L}(x) < \infty.\] The representation $\pi$ is said to be \textit{square integrable} if there exists at least one square integrable vector for $\pi$.
\end{definition}

\begin{remark}
\label{integrability_remark}
Pay attention to the fact that a square integrable representation $\pi$ of a locally compact group $G$ is both unitary and irreducible by definition. These assumptions are implicit whenever we say that a representation $\pi:G \to \mathcal{U}(\mathcal{H}_{\pi})$ is square integrable. A stronger requirement one could impose is for a non-zero vector $g$ to be \textit{integrable} in the sense that $\mathcal{W}_{g}g \in L^{1}(G)$. It follows from the inclusion $L^{1}(G) \cap L^{\infty}(G) \subset L^{2}(G)$ that every integrable vector is square integrable. We will return to this more stringent condition in Chapter \ref{sec: Construction_of_Smoothness_Spaces}.
\end{remark}

\begin{example}
An irreducible unitary representation is not automatically square integrable: Consider the \textit{trivial representation} $\pi:G \to \mathcal{U}(\mathbb{C})$ given by $\pi(x) = \textrm{Id}_{\mathbb{C}}$ for all $x \in G$. Then for $z \in \mathbb{C} \setminus \{0\}$ we have \[\int_{G}|\langle z, \pi(x)z \rangle|^2 \, d\mu_{L}(x) = \int_{G}|\langle z,z \rangle|^2 \, d\mu_{L}(x) = |z|^4 \mu_{L}(G).\] Hence the trivial representation of $G$ is square integrable if and only if $\mu_{L}(G) < \infty.$ This in turn happens if and only if $G$ is compact by \cite[Proposition 1.4.5]{deitmar2014principles}. Since the wavelet transform is continuous, it is clear that any irreducible unitary representation of a compact group is automatically square integrable. In fact, it is not terribly difficult to show that a locally compact group $G$ has a square integrable representation on a finite dimensional vector space if and only if $G$ is compact, see \cite[Proposition 16.4]{robert1983introduction}.
\end{example}

\begin{example}
\label{example_Heisenberg_reduction}
The wavelet transform \eqref{generalized_wavelet_Schrodinger} for the Schr\"{o}dinger representation is \textit{not} square integrable. This is due to the last component $\{0\} \times \{0\} \times \mathbb{R}$ being only present in the phase factors. Notice that $\rho(x,\omega,t) = \textrm{Id}_{L^{2}(\mathbb{R}^n)}$ precisely whenever $(x,\omega,t) = (0,0,n)$ for $n \in \mathbb{Z}$. Hence we can consider the quotient group $\mathbb{H}_{r}^{n} := \mathbb{H}^{n} / \, \textrm{ker}(\rho) \simeq \mathbb{R}^n \times \mathbb{R}^n \times \mathbb{T}$ with the Haar measure $dx \, d\omega \, d\tau$ and the product \[\left(x,\omega,e^{2\pi i \tau}\right) \cdot \left(x',\omega',e^{2\pi i \tau'}\right) := \left(x + x', \omega + \omega', e^{2\pi i (\tau + \tau')}e^{\pi i (x' \omega - x \omega')}\right),\] for $x,x',\omega,\omega' \in \mathbb{R}^n$ and $\tau,\tau' \in \mathbb{R}$. The group $\mathbb{H}_{r}^{n}$ is called the \textit{reduced Heisenberg group}. \par 
The Schr\"{o}dinger representation $\rho:\mathbb{H}^{n} \to \mathcal{U}(L^{2}(\mathbb{R}^n))$ descends to an irreducible unitary representation $\rho_{r}:\mathbb{H}_{r}^{n} \to \mathcal{U}(L^{2}(\mathbb{R}^n))$ given by \[\rho_{r}\left(x,\omega,e^{2\pi i \tau}\right)f(y) = e^{2\pi i \tau}e^{\pi i x \omega}T_{x}M_{\omega}f(y), \qquad \left(x,\omega,e^{2\pi i \tau}\right) \in \mathbb{H}_{r}^{n},\] where $T_{x}$ and $M_{\omega}$ are given in \eqref{time_shift_and_frequency_shift}. Although sloppy, it is common to refer to $\rho_r$ as the \textit{Schr\"{o}dinger representation} as well. In contrast with $\rho$, the representation $\rho_{r}$ is square integrable: For any non-zero $g \in L^{2}(\mathbb{R}^n)$ we have \begin{equation}
\label{Reduced_Heisenberg_computation}
\|\mathcal{W}_{g}g\|_{L^{2}(\mathbb{H}_{r}^{n})}^2 = \int_{0}^{1} \int_{\mathbb{R}^n} \int_{\mathbb{R}^n} |V_{g}g(x,\omega)|^2 \, dx \, d\omega \, d\tau = \|V_{g}g\|_{L^{2}(\mathbb{R}^{2n})}^2 = \|g\|_{L^{2}(\mathbb{R}^n)}^4,    
\end{equation}
where we used the orthogonality relation \eqref{STFT_orthogonality} of the STFT. Hence the map $\mathcal{W}_{g}$ is an isometry from $L^{2}(\mathbb{R}^n)$ to $L^{2}(\mathbb{H}_{r}^{n})$ when $\|g\|_{L^{2}(\mathbb{R}^n)} = 1$.
\end{example}

At first glance, the condition that $g \in \mathcal{H}_{\pi}$ is square integrable seems slightly weaker than the requirement desired, namely that $\mathcal{W}_{g}f \in L^{2}(G)$ for all $f \in \mathcal{H}_{\pi}$. However, it turns out that they are in fact equivalent.

\begin{proposition}
\label{L2-equality-definitions}
Let $\pi:G \to \mathcal{U}(\mathcal{H}_{\pi})$ be a square integrable representation with a square integrable vector $g \in \mathcal{H}_{\pi}$. Then $\mathcal{W}_{g}f \in L^{2}(G)$ for all $f \in \mathcal{H}_{\pi}$.
\end{proposition}

\begin{proof}
Consider the subspace $\mathcal{H}_{g} \subset \mathcal{H}_{\pi}$ consisting of those $f \in \mathcal{H}_{\pi}$ such that $\mathcal{W}_{g}f \in L^{2}(G)$. Then $\mathcal{H}_{g}$ is a non-trivial subspace since $g \in \mathcal{H}_{g}$. The fact that $\mathcal{H}_{g}$ is closed in $\mathcal{H}_{\pi}$ is rather tricky, and we refer the reader to \cite[Lemma 6.3]{wong2002wavelet} for the argument. Notice that $\mathcal{H}_{g}$ is an invariant subspace since \eqref{wavelet_intertwines_left_reg} shows that \[\mathcal{W}_{g} \pi(x)f = L_{x}\mathcal{W}_{g}f, \qquad f \in \mathcal{H}_{g}, \, x \in G.\] By irreducibility, we have $\mathcal{H}_{g} = \mathcal{H}_{\pi}$ and the result follows.
\end{proof}

\begin{remark}
There are several ways of characterizing square integrable representations that we will not emphasize. One of the more elegant formulations \cite[Theorem 2]{duflo1976regular} states that an irreducible unitary representation is square integrable precisely when it is equivalent to a subrepresentation of the left regular representation. In the literature, e.g.\ \cite{fuhr2005abstract}, such representations are sometimes referred to as \textit{discrete series representations}.
\end{remark}

The next result gives a complete answer to how the orthogonality relation \eqref{STFT_orthogonality} generalizes to arbitrary square integrable representations. 

\begin{theorem}[Duflo-Moore Theorem]
\label{Duflo_Moore_theorem}
Let $\pi:G \to \mathcal{U}(\mathcal{H}_{\pi})$ be a square integrable representation. There exists a unique self-adjoint, positive, densely defined operator $C_{\pi}: \mathcal{D}(C_{\pi}) \subset \mathcal{H}_{\pi} \to \mathcal{H}_{\pi}$ with a densely defined inverse such that:
\begin{itemize}
    \item A non-zero element $g \in \mathcal{H}_{\pi}$ is square integrable if and only if $g \in \mathcal{D}(C_{\pi})$.
    \item For $g_1, g_2 \in \mathcal{D}(C_{\pi})$ and $f_1,f_2 \in \mathcal{H}_{\pi}$ we have the orthogonality relation 
    \begin{equation}
    \label{wavelet_transform_orthogonality}
        \langle \mathcal{W}_{g_1}f_1, \mathcal{W}_{g_2}f_2 \rangle_{L^{2}(G)} = \langle f_1, f_2 \rangle_{\mathcal{H}_{\pi}} \overline{\langle C_{\pi}g_1, C_{\pi}g_2 \rangle}_{\mathcal{H}_{\pi}}.
    \end{equation}
    \item The operator $C_{\pi}$ is injective and satisfies the invariance relation
    \begin{equation}
    \label{invariance_relation}
        \pi(x)C_{\pi} = \sqrt{\Delta(x)} C_{\pi} \pi(x),
    \end{equation}
    for all $x \in G$ where $\Delta$ denotes the modular function on $G$.
\end{itemize}
\end{theorem}

For readers interested in the details of this remarkable result, we recommend reading the appendix in \cite[Chapter 2.4]{grossmann1985transforms} as well as the original paper \cite{duflo1976regular}. We will refer to the operator $C_{\pi}$ in Theorem \ref{Duflo_Moore_theorem} as the \textit{Duflo-Moore operator} corresponding to the square integrable representation $\pi:G \to \mathcal{U}(\mathcal{H}_{\pi})$. For our purposes, we record the following consequence: The map $\mathcal{W}_{g}:\mathcal{H}_{\pi} \to L^{2}(G)$ is an isometry if and only if $g \in \mathcal{H}_{\pi}$ is in the domain of the Duflo-Moore operator $C_{\pi}$ and satisfies the \textit{admissibility condition} \[\|C_{\pi}g\|_{\mathcal{H}_{\pi}} = 1.\] An element $g \in \mathcal{H}_{\pi}$ that satisfies these conditions is said to be \textit{admissible}. Notice that any square integrable vector can be normalized to become admissible.

\begin{corollary}
\label{corollary_unimodular}
Let $\pi:G \to \mathcal{U}(\mathcal{H}_{\pi})$ be a square integrable representation of a unimodular group $G$. Then the Duflo-Moore operator $C_{\pi}$ is defined on the whole $\mathcal{H}_{\pi}$ and satisfies $C_{\pi} = c_{\pi} \cdot \mathrm{Id}_{\mathcal{H}_{\pi}}$ for some $c_{\pi} > 0$. In particular, every non-zero vector $g \in \mathcal{H}_{\pi}$ is square integrable. 
\end{corollary}

\begin{proof}
By looking at the invariance relation \eqref{invariance_relation} when $\Delta(x) = 1$ for all $x \in G$, we see that $C_{\pi}$ is a (densely defined) intertwiner of the representation $\pi$. This is only possible when $C_{\pi} = c_{\pi} \cdot \textrm{Id}_{\mathcal{H}_{\pi}}$ for some constant $c_{\pi} \in \mathbb{C}$ due to a generalization of Schur's Lemma, see \cite[Proposition 12.2.2]{deitmar2014principles}. The constant $c_{\pi}$ necessarily has to be positive since $C_{\pi}$ is a positive operator.
\end{proof}

\begin{remark}
We would like to point out that a converse statement to Corollary \ref{corollary_unimodular} is also valid: If $\pi:G \to \mathcal{U}(\mathcal{H}_{\pi})$ is a square integrable representation such that the Duflo-Moore operator $C_{\pi}$ is defined on the whole of $\mathcal{H}_{\pi}$, then the group $G$ is necessarily unimodular. To see this, one uses the invariance relation \eqref{invariance_relation} together with the general fact that the modular function $\Delta$ is either identically one or unbounded. For the last property, it suffices that $\Delta:G \to (0, \infty)$ is a group homomorphism by \cite[Proposition 2.24]{folland2016course}.
\end{remark}

\begin{example}
\label{Example_duflo_morse_Heisenberg}
Let us quickly verify that the Schr\"{o}dinger representation $\rho_{r}$ does indeed fit within this framework. We have previously mentioned that the reduced Heisenberg group $\mathbb{H}_{r}^{n}$ is unimodular. Hence Corollary \ref{corollary_unimodular} implies that the Duflo-Moore operator $C_{\rho_{r}}$ corresponding to $\rho_{r}$ is simply a constant times the identity. We can gauge from \eqref{Reduced_Heisenberg_computation} that $C_{\rho_r} = \textrm{Id}_{L^{2}(\mathbb{R}^n)}$. Hence a function $g \in L^{2}(\mathbb{R}^{n})$ is admissible for the Schr\"{o}dinger representation precisely when $\|g\|_{L^{2}(\mathbb{R}^{n})} = 1$.
\end{example}

\begin{example}
\label{compact_duflo_moore_example}
Let $\pi:G \to \mathcal{U}(\mathcal{H}_{\pi})$ be an irreducible unitary representation of a compact group $G$. From Peter-Weyl theory, see e.g.\ \cite[Theorem 7.3.2]{deitmar2014principles}, it follows that $\mathcal{H}_{\pi}$ has to be finite dimensional. Moreover, any non-zero vector $g \in \mathcal{H}_{\pi}$ is square integrable since $\mathcal{W}_{g}g$ is a continuous function on the compact space $G$. Thus the Duflo-Moore operator satisfies $C_{\pi} = c_{\pi} \cdot \textrm{Id}_{\mathcal{H}_{\pi}}$ for some $c_{\pi} > 0$. So, what is the constant $c_{\pi}$? It follows from \cite[Example 12.2.7]{deitmar2014principles} that we have the elegant formula \[c_{\pi} = \frac{1}{\sqrt{\textrm{dim}(\mathcal{H}_{\pi})}}.\] 
\end{example}

\begin{example}
Let us demonstrate how Theorem \ref{Duflo_Moore_theorem} can simplify concrete settings: Consider two normalized vectors $x, y \in \mathbb{R}^{n}$ and a rotation $R \in SO(n)$. The quantity $|\langle y, Rx \rangle|^{2}$ measures the square deviation from $Rx$ and $y$ being orthogonal. What is the average of such orthogonality deviations when the normalized vectors $x, y \in \mathbb{R}^{n}$ are fixed and $R \in SO(n)$ is allowed to vary? Unwinding the question, we are asking for the value \[\int_{SO(n)}|\langle y, Rx \rangle|^2 \, d\mu(R), \qquad x,y \in \mathbb{R}^{n}, \quad \|x\| = \|y\| = 1.\] When $n = 2$ the answer should be $1/2$ based on geometric considerations. This can be verified by brute force since any $R \in SO(2)$ can be written as $R = R_{\theta}$ for $\theta \in [0, 2\pi)$ with 
\[R_{\theta} := \begin{pmatrix} \cos(\theta) & -\sin(\theta) \\ \sin(\theta) & \cos(\theta)\end{pmatrix}.\] \par
Is there a more satisfactory approach that works for all $n \geq 2$? Look closely, there is nothing up my sleeve: Consider the obvious representation $\pi:SO(n) \to \mathcal{U}(\mathbb{R}^{n})$ given by $\pi(R)x := R \cdot x$ for $R \in SO(n)$ and $x \in \mathbb{R}^{n}$. Then $\pi$ is easily seen to be square integrable. We can by Theorem \ref{Duflo_Moore_theorem} and Example \ref{compact_duflo_moore_example} write 
\begin{equation}
\label{orthogonal_group_formula}
    \int_{SO(n)}|\langle y, Rx \rangle|^2 \, d\mu(R) = \int_{SO(n)}|\mathcal{W}_{x}y(R)|^2 \, d\mu(R) = \langle y, y\rangle \langle C_{\pi}x, C_{\pi}x \rangle = \frac{1}{n}.
\end{equation}
In words, the formula \eqref{orthogonal_group_formula} expresses the fact that in higher dimensions, two random normalized vectors are more likely to be orthogonal to each other; there are simply more ways to be orthogonal in higher dimensions.
\end{example}

We would like to end this section with an example of a square integrable representation of a non-unimodular group. Although somewhat lengthy, we encourage the fatigued reader to soldier on through the next example as most of the theory we have developed is present in some way. 

\begin{example}
\label{wavelet_example}
In this example we examine a unitary representation of the affine group $\textrm{Aff}$ given in Example \ref{affine_group_first_example}. We have a family of \textit{dilation operators} $D_{a}$ on $L^{2}(\mathbb{R})$ for $a \in \mathbb{R}^{*}$ given by
\begin{equation}
\label{usual_dilation_operator}
D_{a}f(x) := \frac{1}{\sqrt{|a|}}f\left(\frac{x}{a}\right), \qquad f \in L^{2}(\mathbb{R}).
\end{equation}
Together with the translation operator $T_{b}$ in \eqref{time_shift_and_frequency_shift} we obtain a unitary representation of the affine group $\pi:\textrm{Aff} \to \mathcal{U}(L^{2}(\mathbb{R}))$ given by 
\begin{equation}
\label{classical_wavelet_representation}
    \pi(b,a)f(x) := T_{b}D_{a}f(x) = \frac{1}{\sqrt{|a|}}f\left(\frac{x - b}{a}\right), \qquad (b,a) \in \mathrm{Aff}.
\end{equation} 
It is common to refer to $\pi$ as the \textit{wavelet representation}. To see that a unitary representation is irreducible, it can often be a good strategy to jump straight to checking when it is square integrable. For the wavelet representation, a formal computation using the Fourier transform shows that 
\begin{equation}
\label{wavelet_square_int_equation}
    \int_{\textrm{Aff}}|\langle f, \pi(b,a)g \rangle |^2 \, \frac{db \, da}{a^2} = \left(\int_{\mathbb{R}}|\mathcal{F}(f)(b)|^2 \, db\right) \left(\int_{\mathbb{R}^{*}} \frac{|\mathcal{F}(g)(a)|^2}{|a|} \, da\right),
\end{equation}
for any $f,g \in L^{2}(\mathbb{R})$. We refer the reader to \cite[Example 2.48]{dahlke2015harmonic} for details of the computation above. The right-hand side of \eqref{wavelet_square_int_equation} is always non-zero as long as we choose $f,g$ to be non-zero elements in $L^{2}(\mathbb{R})$. Hence $g$ is a cyclic vector for all non-zero $g \in L^{2}(\mathbb{R})$. This implies that the wavelet representation $\pi$ is irreducible by Proposition \ref{proposition_cyclic_irreducible}. \par
Which non-zero vectors $g \in L^{2}(\mathbb{R})$ are square integrable? From \eqref{wavelet_square_int_equation}, we see that we need $g$ to satisfy the condition 
\begin{equation}
\label{Calderon_equation}
    \int_{\mathbb{R}^{*}} \frac{|\mathcal{F}(g)(a)|^2}{|a|} \, da < \infty.
\end{equation}
The condition \eqref{Calderon_equation} is sometimes called the \textit{Calder\'{o}n condition} or the \textit{wavelet condition}. It is clear from \eqref{wavelet_square_int_equation} and the uniqueness statement of Theorem \ref{Duflo_Moore_theorem} that the Duflo-Moore operator $C_{\pi}$ is the Fourier multiplier given by \[C_{\pi}g =  \mathcal{F}^{-1}\left(\frac{1}{\sqrt{|a|}}\mathcal{F}(g)(a)\right), \qquad g \in \mathcal{D}(C_{\pi}).\] 
We know that $g \in \mathcal{D}\left(C_{\pi}\right)$ is admissible if and only if $\|C_{\pi}g\|_{L^{2}(\mathbb{R})} = 1$. Hence $g \in L^{2}(\mathbb{R})$ is admissible for the wavelet representation if and only if 
\begin{equation}
\label{admissible_wavelet}
    \int_{\mathbb{R}^{*}} \frac{|\mathcal{F}(g)(a)|^2}{|a|} \, da = 1.
\end{equation}
Elements in $L^{2}(\mathbb{R})$ that satisfy \eqref{admissible_wavelet} are sometimes called \textit{admissible wavelets} in the literature.

The wavelet transform for the wavelet representation is given explicitly by 
\begin{equation}
\label{classical_wavelet_transform}
\mathcal{W}_{g}f(b,a) = \langle f, T_{b}D_{a}g \rangle = \frac{1}{\sqrt{|a|}}\int_{\mathbb{R}}f(x)\overline{g\left(\frac{x - b}{a}\right)} \, dx,
\end{equation}
where $(b,a) \in \textrm{Aff}$ and $f,g \in L^{2}(\mathbb{R})$. This is precisely the \textit{continuous wavelet transform} in wavelet analysis, see e.g.\ \cite[Chapter 2]{daubechies1992ten}. In fact, this example is the motivation for the terminology \textit{(generalized) wavelet transform}. If $g \in L^{2}(\mathbb{R})$ is an admissible wavelet and $f_1,f_2 \in L^{2}(\mathbb{R})$ are arbitrary, then Theorem \ref{Duflo_Moore_theorem} implies that we have the orthogonality relation \[\int_{\textrm{Aff}} \mathcal{W}_{g}f_{1}(b,a) \overline{\mathcal{W}_{g}f_2(b,a)} \, \frac{db \, da}{a^2} = \int_{\mathbb{R}}f_1(x)\overline{f_{2}(x)} \, dx.\]
\end{example}

\subsection{Reproducing Kernel Hilbert Spaces}
\label{sec: Reproducing_Kernel_Hilbert_Spaces}

In this section we define reproducing kernel Hilbert spaces and show that they naturally occur in the setting of generalized wavelet transforms. We believe that reproducing kernel Hilbert spaces can illuminate the theory and make results such as Theorem \ref{reproducing_formula} in Section \ref{sec: Inversion_Formula} more transparent. Although the theory of reproducing kernel Hilbert spaces is often implicit in works on coorbit theory, it is seldom written out in detail.

\begin{definition}
Let $X$ be a set and let $\mathcal{H}$ be a Hilbert space consisting of functions $f: X \to \mathbb{C}$. We say that $\mathcal{H}$ is a \textit{reproducing kernel Hilbert space} if the evaluation functionals $\{E_{x}\}_{x \in X}$ are bounded, where \[E_{x}(f) := f(x), \qquad f \in \mathcal{H}.\] If the evaluation functionals $\{E_{x}\}_{x \in X}$ are uniformly bounded, then we refer to $\mathcal{H}$ as a \textit{uniform reproducing kernel Hilbert space}.
\end{definition}

Given a reproducing kernel Hilbert space $\mathcal{H}$, we have by the Riesz representation theorem that for each $x \in X$ there is a unique element $k_{x} \in \mathcal{H}$ such that \[f(x) = \langle f, k_x \rangle, \qquad f \in \mathcal{H}.\] We refer to $k_x$ as the \textit{reproducing kernel for the point} $x \in X$. Since $k_{x}$ is again a function on $X$, we can evaluate $k_{x}(y)$ for $y \in X$ and obtain $k_{x}(y) = \langle k_{x}, k_{y} \rangle = \overline{k_{y}(x)}$. The function $K:X \times X \to \mathbb{C}$ given by \[K(x,y) = \langle k_y, k_x \rangle\] is called the \textit{reproducing kernel} for $\mathcal{H}$.

\begin{example}
Consider the \textit{Paley-Wiener space} $PW_{A}$ for a fixed $A > 0$ consisting of functions $f \in L^{2}(\mathbb{R})$ such that $\textrm{supp}(\mathcal{F}(f)) \subset \left[-A, A\right]$, where $\mathcal{F}$ denotes the Fourier transform. This space plays a major role in sampling theory and classical harmonic analysis. The elements in $PW_{A}$ are actually smooth functions since their Fourier transforms have compact support. Moreover, the space $PW_{A}$ is a Hilbert spaces under the inner-product \[\langle f, g \rangle_{PW_{A}} := \left<\mathcal{F}(f), \mathcal{F}(g)\right>_{L^{2}\left[-A, A\right]}.\] To see that the evaluation functionals $\{E_{x}\}_{x \in \mathbb{R}}$ are bounded, we compute for $f \in PW_{A}$ that 
\begin{align*}
    |E_{x}(f)| = |f(x)| & = \left| \mathcal{F}^{-1}\left(\mathcal{F}(f)\right)(x)\right| \\ & = \left|\int_{-A}^{A}\mathcal{F}(f)(\omega)e^{2 \pi i x \omega} \, d\omega\right| \\ & \leq \left(\int_{-A}^{A}\left|\mathcal{F}(f)(\omega)\right|^2 \, d\omega \right)^{\frac{1}{2}} \left(\int_{-A}^{A} \, d\omega\right)^{\frac{1}{2}} \\ & = \sqrt{2A} \cdot \|f\|_{PW_{A}}.
\end{align*}
Since $A > 0$ is fixed, we conclude that $PW_{A}$ is a uniform reproducing kernel Hilbert space. To find the reproducing kernel $K_{A}: \mathbb{R} \times \mathbb{R} \to \mathbb{C}$, notice that \[f(x) = \langle f, k_{x} \rangle_{PW_{A}} = \int_{-A}^{A}\mathcal{F}(f)(\omega) \overline{\mathcal{F}(k_{x})(\omega)} \, d\omega.\] In view of the Fourier inversion $f = \mathcal{F}^{-1}\left(\mathcal{F}(f)\right)$, it follows that $\mathcal{F}(k_{x})(\omega) = e^{-2 \pi i x \omega}$. Hence \[K_{A}(x,y) = \overline{k_{x}(y)} = \overline{\mathcal{F}^{-1}(e^{-2\pi i x \cdot})(y)} = \begin{cases} \frac{1}{\pi} \frac{\sin(2\pi A(x - y))}{x - y}, \, & \textrm{if} \, \, x \neq y \\ 2A, \, & \textrm{if} \, \, x = y\end{cases}.\]
\end{example}

A useful feature of reproducing kernel Hilbert spaces is that convergence in norm implies pointwise convergence. To see this, let $f_n,f \in \mathcal{H}$ and assume $\|f_n- f\| \to 0$. Then 
\begin{equation}
\label{norm_implies_pointwise}
    |f_{n}(x) - f(x)| = |\langle f_n - f, k_x\rangle| \leq \|f_n - f\| \|k_x\| = \|f_n - f\| \|E_x\| \to 0.
\end{equation}
If $\mathcal{H}$ in addition is a uniform reproducing kernel Hilbert space, then \eqref{norm_implies_pointwise} shows that convergence in norm implies uniform convergence. The reader can consult \cite{paulsen2016introduction} for more examples and properties of general reproducing kernel Hilbert spaces. \par
We now return to the setting of square integrable representations $\pi:G \to \mathcal{U}(\mathcal{H}_{\pi})$ to illustrate how they naturally give rise to reproducing kernel Hilbert spaces. Pick an admissible vector $g \in \mathcal{H}_{\pi}$ so that $\mathcal{W}_{g}:\mathcal{H}_{\pi} \to L^{2}(G)$ is an isometry. We will consider the image space \[\mathcal{W}_{g}(\mathcal{H}_{\pi}) \subset L^{2}(G).\]
Notice that, since $\mathcal{W}_{g}$ is an isometry, we have \begin{equation}
\label{eq:isometry_property}
    \mathcal{W}_{g}^{*} \circ \mathcal{W}_{g} = \textrm{Id}_{\mathcal{H}_{\pi}} \, \, \, \textrm{and} \, \, \,  \mathcal{W}_{g} \circ \mathcal{W}_{g}^{*}\Big|_{\mathcal{W}_{g}(\mathcal{H}_{\pi})} = \textrm{Id}_{\mathcal{W}_{g}(\mathcal{H}_{\pi})}.
\end{equation}

\begin{proposition}
\label{spaces_are_RKHS}
Let $\pi:G \to \mathcal{U}(\mathcal{H}_{\pi})$ be a square integrable representation with an admissible vector $g \in \mathcal{H}_{\pi}$. The space $\mathcal{W}_{g}(\mathcal{H}_{\pi})$ is a uniform reproducing kernel Hilbert space with reproducing kernel
\[K_{g}(x,y) = \mathcal{W}_{g}g(y^{-1}x), \qquad x,y \in G.\] 
\end{proposition}

\begin{proof}
The admissibility of $g \in \mathcal{H}_{\pi}$ ensures that $\mathcal{W}_{g}(\mathcal{H}_{\pi})$ is a closed subspace of $L^{2}(G)$. Thus $\mathcal{W}_{g}(\mathcal{H}_{\pi})$ is a Hilbert space with the norm \[\|\mathcal{W}_{g}f\|_{\mathcal{W}_{g}(\mathcal{H}_{\pi})} := \|\mathcal{W}_{g}f\|_{L^{2}(G)} = \|f\|_{\mathcal{H}_{\pi}}, \qquad f \in \mathcal{H}_{\pi}.\] For $F \in \mathcal{W}_{g}(\mathcal{H}_{\pi})$ and $x \in G$ we can thus write \[F(x) = \mathcal{W}_{g}\left(\mathcal{W}_{g}^{*}F\right)(x) = \left\langle \mathcal{W}_{g}^{*}F, \pi(x)g \right\rangle = \left\langle F, \mathcal{W}_{g}\left(\pi(x)g\right) \right\rangle.\] Since $\mathcal{W}_{g}\left(\pi(x)g\right) \in \mathcal{W}_{g}(\mathcal{H}_{\pi})$ we have that $\mathcal{W}_{g}(\mathcal{H}_{\pi})$ is a reproducing kernel Hilbert space. The reproducing kernel $K_{g}:G \times G \to \mathbb{C}$ is given by \[K_{g}(x,y) =  \left\langle \mathcal{W}_{g}\left(\pi(y)g\right), \mathcal{W}_{g}\left(\pi(x)g\right) \right\rangle = \langle \pi(y)g, \pi(x)g \rangle = \mathcal{W}_{g}(\pi(y)g)(x) = \mathcal{W}_{g}g(y^{-1}x).\] If $E_{x}$ is the evaluation functional for the point $x \in G$ then \[\|E_{x}\|_{\mathcal{W}_{g}(\mathcal{H}_{\pi})^{*}} = \|k_x\|_{\mathcal{W}_{g}(\mathcal{H}_{\pi})} = \|\mathcal{W}_{g}\left(\pi(x)g\right)\|_{\mathcal{W}_{g}(\mathcal{H}_{\pi})} = \|\pi(x)g\|_{\mathcal{H}_{\pi}} = \|g\|_{\mathcal{H}_{\pi}}.\]
It follows that $\mathcal{W}_{g}(\mathcal{H}_{\pi})$ is a uniform reproducing kernel Hilbert space since we have fixed $g$.
\end{proof}

For a locally compact group $G$, we say that an element $S \in L^{2}(G)$ is \textit{self-adjoint convolution idempotent} if $\overline{S(x)} = S(x^{-1})$ for all $x \in G$ and $S *_{G} S = S$. It will follow from Theorem \ref{reproducing_formula} that the element $\mathcal{W}_{g}g$ is self-adjoint convolution idempotent whenever $g \in \mathcal{H}_{\pi}$ is admissible. A converse to this statement can be found in \cite[Proposition 2.38]{fuhr2005abstract}. In \cite[Theorem 2.45]{fuhr2005abstract} the following generalization of a classical result of Wilczok \cite{wilczok2000new} is derived.

\begin{proposition}
Let $G$ be a locally compact group that is connected and non-compact. Consider a square integrable representation $\pi:G \to \mathcal{U}(\mathcal{H}_{\pi})$ and fix an admissible vector $g \in \mathcal{H}_{\pi}$. If $F \in \mathcal{W}_{g}(\mathcal{H}_{\pi})$ is supported on a set of finite Haar measure, then $F \equiv 0$. 
\end{proposition}

\begin{remark}
The reader can consult \cite[Chapter 2.5]{fuhr2005abstract} for more interesting results regarding self-adjoint convolution idempotents. We refer the reader to \cite{berge2020interpolation, cont_wavelet_transform} for further properties of the spaces $\mathcal{W}_{g}(\mathcal{H}_{\pi})$.
\end{remark}

\subsection{The Reproducing and Reconstruction Formulas}
\label{sec: Inversion_Formula}

We end this chapter by providing two important results that tie up loose ends. Firstly, we prove the \textit{reproducing formula} in Theorem \ref{reproducing_formula}. This result has a simple interpretation in the language of reproducing kernel Hilbert spaces. Secondly, we generalize the reconstruction formula for the STFT in \eqref{reconstruction_for_the_STFT} to square integrable representations in Corollary \ref{reconstruction_corollary}. Both of these results have short and elegant proofs that build on the theory developed so far.

\begin{theorem}[Reproducing Formula]
\label{reproducing_formula}
Let $\pi:G \to \mathcal{U}(\mathcal{H}_{\pi})$ be a square integrable representation and fix an admissible vector $g \in \mathcal{H}_{\pi}$. Then $\mathcal{W}_{g} \circ \mathcal{W}_{g}^{*}$ is the projection from $L^{2}(G)$ to $\mathcal{W}_{g}(\mathcal{H}_{\pi})$ and has the explicit form 
\begin{equation*}
    \mathcal{W}_{g} \circ \mathcal{W}_{g}^{*}(F) = F *_{G} \mathcal{W}_{g}g, \qquad F \in L^{2}(G).
\end{equation*}
In particular, for $F \in \mathcal{W}_{g}(\mathcal{H}_{\pi})$ we have 
\begin{equation}
\label{special_case_reproducing_formula}
F = F *_{G} \mathcal{W}_{g}g.
\end{equation}
\end{theorem}

\begin{proof}
The map $\mathcal{W}_{g}:\mathcal{H}_{\pi} \to L^{2}(G)$ is an isometry since $g \in \mathcal{H}_{\pi}$ is admissible. Hence $\mathcal{W}_{g} \circ \mathcal{W}_{g}^{*}$ is the projection from $L^{2}(G)$ to $\mathcal{W}_{g}(\mathcal{H}_{\pi})$. For $x \in G$ and $F \in L^{2}(G)$, an initial computation using \eqref{wavelet_intertwines_left_reg} shows that \[\mathcal{W}_{g}\left(\mathcal{W}_{g}^{*}F\right)(x) = \langle \mathcal{W}_{g}^{*}(F), \pi(x)g \rangle = \langle F, \mathcal{W}_{g}(\pi(x)g)\rangle = \langle F, L_{x}\mathcal{W}_{g}g\rangle.\] Since $\overline{\mathcal{W}_{g}g(x)} = \mathcal{W}_{g}g(x^{-1})$ we end up with \[\langle F, L_{x}\mathcal{W}_{g}g\rangle = \int_{G}F(y)\overline{\mathcal{W}_{g}g(x^{-1}y)} \, d\mu_{L}(y) = \int_{G}F(y)\mathcal{W}_{g}g(y^{-1}x) \, d\mu_{L}(y) =  (F *_{G} \mathcal{W}_{g}g)(x). \qedhere\]
\end{proof}

\begin{remark}
The special case \eqref{special_case_reproducing_formula} motivates the name \textit{reproducing formula}, as we can reproduce the values of $F \in \mathcal{W}_{g}(\mathcal{H}_{\pi})$ by convolving $F$ with $\mathcal{W}_{g}g \in \mathcal{W}_{g}(\mathcal{H}_{\pi})$. Notice that $\mathcal{W}_{g}g$ is precisely the reproducing kernel $k_{e}$ for the identity element $e \in G$. Hence \eqref{special_case_reproducing_formula} shows that the reproducing kernel $k_{e}$ is a (right) identity for $\mathcal{W}_{g}(\mathcal{H}_{\pi})$ with respect to the convolution product. The fact that the wavelet transform $\mathcal{W}_{g}$ for any admissible $g \in \mathcal{H}_{\pi}$ is an isomorphism \[\mathcal{W}_{g}:\mathcal{H}_{\pi} \xrightarrow{\sim} \mathcal{W}_{g}(\mathcal{H}_{\pi}) = \left\{F \in L^{2}(G) \, : \, F = F *_{G} \mathcal{W}_{g}g\right\}\] is a special case of the \textit{correspondence principle} in Theorem \ref{corresponance_principle}.
\end{remark}

We now take a brief detour to \textit{weak integrals} so that uninitiated readers will be less squeamish when encountering expressions on the form \eqref{eq: weak_integral_adjoint}. Let $\Phi:G \to \mathcal{H}$ be a continuous function from a locally compact group $G$ to a Hilbert space $\mathcal{H}$. We need to make sense of 
\begin{equation}
\label{weak_integral_element}
    \int_{G} \Phi(x) \, d\mu_{L}(x)
\end{equation}
as an element in $\mathcal{H}$. This can be done under a mild additional requirement. Specifically, we require that the linear functional on $\mathcal{H}$ given by
\[f \longmapsto \int_{G} \left\langle \Phi(x), f \right\rangle \, d\mu_{L}(x)\] is well-defined and bounded. Under this assumption, the Riesz representation theorem implies the existence of an element in $\mathcal{H}$ denoted by \eqref{weak_integral_element}
such that \[\left\langle \int_{G} \Phi(x) \, d\mu_{L}(x), f \right\rangle = \int_{G} \langle \Phi(x), f \rangle \, d\mu_{L}(x),\] for every $f \in \mathcal{H}$. We refer to the element \eqref{weak_integral_element} as the \textit{weak integral} of the function $\Phi$. 

\begin{proposition}
\label{proposition_adjoint_formula}
Let $\pi:G \to \mathcal{U}(\mathcal{H}_{\pi})$ be a square integrable representation and fix an admissible vector $g \in \mathcal{H}_{\pi}$. Then for $F \in L^{2}(G)$ we can represent $\mathcal{W}_{g}^{*}(F)$ as the weak integral \begin{equation}
\label{eq: weak_integral_adjoint}
    \mathcal{W}_{g}^{*}(F) = \int_{G} F(x)\pi(x)g \, d\mu_{L}(x).
\end{equation}
\end{proposition}

\begin{proof}
Notice that $\Phi^{F}:G \to \mathcal{H}_{\pi}$ given by $\Phi^{F}(x) \coloneqq F(x)\pi(x)g$ for $F \in L^{2}(G)$ satisfies the required properties for a weak integral due to the assumed continuity of $\pi$ and the estimate
\begin{align*}
\left|\int_{G}\langle F(x) \pi(x)g, f \rangle \, d\mu_{L}(x) \right| 
& \leq
\int_{G}|\langle F(x) \pi(x)g, f \rangle| \, d\mu_{L}(x)
\\ & = 
\int_{G}|F(x)|\cdot|\mathcal{W}_{g}f(x)| \, d\mu_{L}(x)
\\ & \leq
\|F\|_{L^{2}(G)}\|\mathcal{W}_{g}f\|_{L^{2}(G)}
\\ & = 
\|F\|_{L^{2}(G)}\|f\|_{\mathcal{H}_{\pi}},
\end{align*}
for $f \in \mathcal{H}_{\pi}$. The claim hence follows from the computation
\[\langle \mathcal{W}_{g}^{*}(F), f \rangle_{\mathcal{H}_{\pi}} = \langle F, \mathcal{W}_{g}f \rangle_{\mathcal{W}_{g}(\mathcal{H}_{\pi})} = \int_{G}F(x) \overline{\mathcal{W}_{g}f(x)} \, d\mu_{L}(x) = \int_{G}\langle F(x) \pi(x)g, f \rangle \, d\mu_{L}(x). \qedhere\]
\end{proof}

By combining Proposition \ref{proposition_adjoint_formula} with \eqref{eq:isometry_property} we obtain the following generalization of the reconstruction formula for the STFT given in \eqref{reconstruction_for_the_STFT}.

\begin{corollary}[Reconstruction Formula]
\label{reconstruction_corollary}
Let $\pi:G \to \mathcal{U}(\mathcal{H}_{\pi})$ be a square integrable representation and fix an admissible vector $g \in \mathcal{H}_{\pi}$. We can represent any $f \in \mathcal{H}_{\pi}$ as the weak integral 
\begin{equation*}
    f = \mathcal{W}_{g}^{*}\left(\mathcal{W}_{g}f\right) = \int_{G}\mathcal{W}_{g}f(x) \pi(x)g \, d\mu_{L}(x).
\end{equation*}
Hence we have for any $h \in \mathcal{H}_{\pi}$ that 
\[\langle f, h \rangle = \int_{G}\mathcal{W}_{g}f(x)\overline{\mathcal{W}_{g}h(x)} \, d\mu_{L}(x).\]
\end{corollary}

\begin{example}
For the wavelet representation given in Example \ref{wavelet_example}, the reconstruction formula in Corollary \ref{reconstruction_corollary} takes the form \[f = \int_{\textrm{Aff}}\mathcal{W}_{g}f(b,a)T_{b}D_{a}g \, \frac{db \, da}{a^2},\]
for $f \in L^{2}(\mathbb{R})$ arbitrary and $g \in L^{2}(\mathbb{R})$ an admissible wavelet. 
\end{example}

\section{In the Midst of Coorbit Spaces}
\label{sec: Construction_of_Smoothness_Spaces}
In this chapter we will define the coorbit spaces and derive their basic properties. The coorbit spaces consist of elements $\eta$ such that the wavelet transform $\mathcal{W}_{g}\eta$ has suitable decay as a function on the group $G$. However, the elements $\eta$ will not be picked from $\mathcal{H}_{\pi}$, but rather from a larger distributional space. The aim of the first two sections in this chapter is to make this notion precise. Once this is ready, we will define coorbit spaces without weights in Section \ref{sec: Coorbit_Spaces}. The weighted versions will be introduced in Section \ref{sec: Weighted_Versions} so that we can initially introduce coorbit spaces with minimal technicalities. Although this is an uncommon approach in the literature, we believe that what this approach lacks in efficiency is made up for by increased clarity. In Section \ref{sec: Atomic Decompositions} we show that the coorbit spaces can be discretized in a way that reflects the geometry of the underlying group. Finally, we discuss Banach frames and kernel theorems for coorbit spaces respectively in Section~\ref{sec: Banach_Frames} and Section~\ref{sec: A Kernel Theorem for Coorbit Spaces}.
\\\\
\textbf{Restriction to $\sigma$-compact groups:} For some results in this chapter, we will need that the locally compact group $G$ is \textit{$\sigma$-compact}, that is, there exists a sequence of compact sets $(K_n)_{n \in \mathbb{N}}$ with $K_n \subset G$ such that $\cup_{n \in \mathbb{N}}K_n = G$. Rather than explicitly requiring this at individual points in the exposition, we henceforth restrict our attention to $\sigma$-compact groups. Whenever we refer to a representation $\pi:G \to \mathcal{U}(\mathcal{H}_{\pi})$, it is from now on implicitly assumed that $G$ is a $\sigma$-compact locally compact group. We remark that $\sigma$-compactness for locally compact groups is a mild condition: Not only is any second countable or connected locally compact group (e.g. any Lie group) $\sigma$-compact, but by \cite[Proposition 2.4]{folland2016course} we can always find a subgroup of a locally compact group that is open, closed, and $\sigma$-compact.

\subsection{Integrable Representations and Test Vectors}
\label{sec: Test_Spaces_and_Distributional_Spaces}

In this section we will go from the square integrable setting to the integrable setting. An irreducible unitary representation $\pi:G \to \mathcal{U}(\mathcal{H}_{\pi})$ is said to be \textit{integrable} if there exists an \textit{integrable vector}, that is, if there is a non-zero vector $g \in \mathcal{H}_{\pi}$ such that $\mathcal{W}_{g}g \in L^{1}(G)$. Notice that $\pi$ is then automatically square integrable since $\mathcal{W}_{g}g \in L^{1}(G) \cap L^{\infty}(G) \subset L^{2}(G)$. We use the notation \[\mathcal{A} := \left\{g \in \mathcal{H}_{\pi} \, : \, \mathcal{W}_{g}g \in L^{1}(G) \right\}.\] The set $\mathcal{A}$ is sometimes called the \textit{analyzing vectors} in the literature \cite{feichtinger1988unified}. Notice that $\mathcal{A}$ contains all the integrable vectors as well as the zero vector. \par 
From now on, we will require that $\mathcal{A}$ is non-trivial, that is, we require that the representation $\pi$ is integrable. Given an integrable vector $g \in \mathcal{A} \setminus \{0\}$, we can define the corresponding space of \textit{test vectors} \[\mathcal{H}_{g}^{1} := \left\{f \in \mathcal{H}_{\pi} \, : \, \mathcal{W}_{g}f \in L^{1}(G)\right\}.\]
The terminology \textquote{test vectors} is not standard, although it has been used in \cite{speck19, felix_thesis}. We will explain in Section \ref{sec: Reservoirs_and_the_Extended_Wavelet_Transform} why this terminology is suitable. \par 
The space $\mathcal{H}_{g}^{1}$ can be equipped with the norm 
\begin{equation*}
    \|f\|_{\mathcal{H}_{g}^{1}} := \|\mathcal{W}_{g}f\|_{L^{1}(G)}, \qquad f \in \mathcal{H}_{g}^{1}.
\end{equation*}
To see that this is a norm and not just a seminorm we assume that $\|\mathcal{W}_{g}f\|_{L^{1}(G)} = 0$ for some $f \in \mathcal{H}_{g}^{1}$. Then $\mathcal{W}_{g}f$ is zero almost everywhere as a function on $G$. This implies that $\mathcal{W}_{g}f$ represents the zero equivalence class in $L^{2}(G)$. The injectivity of $\mathcal{W}_{g}:\mathcal{H}_{\pi} \to L^{2}(G)$ ensured by Proposition \ref{proposition_cyclic_irreducible} and Lemma \ref{injection_lemma} gives that $f = 0$ as an element in $\mathcal{H}_{\pi},$ and hence also as an element in $\mathcal{H}_{g}^{1}$. 

\begin{proposition}
\label{inclusion_proposition}
Let $\pi:G \to \mathcal{U}(\mathcal{H}_{\pi})$ be an integrable representation and fix an integrable vector $g \in \mathcal{A} \setminus \{0\}$.
The restriction $\pi|_{\mathcal{H}_{g}^{1}}$ acts by isometries on the set of test vectors $\mathcal{H}_{g}^{1}$. Furthermore, the space of test vectors $\mathcal{H}_{g}^{1}$ is dense in $\mathcal{H}_{\pi}$.
\end{proposition}

\begin{proof}
 It is clear that $\mathcal{H}_{g}^{1}$ is a linear subspace of $\mathcal{H}_{\pi}$. Moreover, for $f \in \mathcal{H}_{g}^{1}$ we have for $x \in G$ that  \[\|\pi(x)f\|_{\mathcal{H}_{g}^{1}} = \|\mathcal{W}_{g}\pi(x)f\|_{L^{1}(G)} = \|L_{x}\mathcal{W}_{g}f\|_{L^{1}(G)} = \|\mathcal{W}_{g}f\|_{L^{1}(G)} = \|f\|_{\mathcal{H}_{g}^{1}}.\] Hence the closure of $\mathcal{H}_{g}^{1}$ in the norm on $\mathcal{H}_{\pi}$ is a non-trivial closed subspace of $\mathcal{H}_{\pi}$ where $\pi$ acts by isometries. The irreducibility of $\pi$ implies that $\mathcal{H}_{g}^{1}$ is a dense subspace of $\mathcal{H}_{\pi}$.
\end{proof}

\begin{remark}
It is tempting, in light of Proposition \ref{L2-equality-definitions}, to attempt to show that $\mathcal{H}_{g}^{1}$ is closed in $\mathcal{H}_{\pi}$. Then Proposition \ref{inclusion_proposition} would show that $\mathcal{H}_{g}^{1} = \mathcal{H}_{\pi}$. However, this is generally false and we will give a concrete counterexample in Example \ref{Feichtinger_algebra_example}. In fact, it will be clear from Section \ref{sec: Coorbit_Spaces} that coorbit theory is not very interesting whenever $\mathcal{H}_{g}^{1} = \mathcal{H}_{\pi}$.
\end{remark}

\begin{proposition}
\label{prop:completeness_H1}
Let $\pi:G \to \mathcal{U}(\mathcal{H}_{\pi})$ be an integrable representation. Then for any integrable vector $g \in \mathcal{A} \setminus \{0\}$ the test vectors $\mathcal{H}_{g}^1$ form a Banach space that is continuously embedded into $\mathcal{H}_{\pi}$.
\end{proposition}
\begin{proof}
We begin by showing that the space $\mathcal{H}_{g}^{1}$ is continuously embedded into $\mathcal{H}_{\pi}$. For an element $f \in \mathcal{H}_{g}^{1}$, we have by the orthogonality relations in \eqref{wavelet_transform_orthogonality} that
\begin{align*}
    \|C_{\pi}g\|_{\mathcal{H}_{\pi}}^{2}\|f\|_{\mathcal{H}_{\pi}}^2 = \|\mathcal{W}_{g}f\|_{L^{2}(G)}^2 & = \int_{G}|\langle f, \pi(x)g \rangle|  |\mathcal{W}_{g}f(x)| \, d\mu_{L}(x)  \\ & \leq \int_{G}\|f\|_{\mathcal{H}_{\pi}} \|\pi(x)g\|_{\mathcal{H}_{\pi}} |\mathcal{W}_{g}f(x)| \, d\mu_{L}(x) \\ & = \|f\|_{\mathcal{H}_{\pi}} \|g\|_{\mathcal{H}_{\pi}} \|\mathcal{W}_{g}f\|_{L^{1}(G)}. 
\end{align*}
Since $C_{\pi}g \neq 0$ due to the injectivity of $C_{\pi}$, we can rearrange and obtain  
\begin{equation}
\label{norm_inequality}
    \|f\|_{\mathcal{H}_{\pi}} \leq \frac{\|g\|_{\mathcal{H}_{\pi}}}{\|C_{\pi}g\|_{\mathcal{H}_{\pi}}^{2}} \|\mathcal{W}_{g}f\|_{L^{1}(G)} = \frac{\|g\|_{\mathcal{H}_{\pi}}}{\|C_{\pi}g\|_{\mathcal{H}_{\pi}}^{2}} \|f\|_{\mathcal{H}_{g}^{1}}.
\end{equation}
\par 
Let us now show that $\mathcal{H}_{g}^{1}$ is complete. Assume that $\{f_{n}\}_{n \in \mathbb{N}}$ is a Cauchy-sequence in $\mathcal{H}_{g}^{1}$. By definition, this means that for every $\epsilon > 0$ there exists $N \in \mathbb{N}$ such that for $n,m \geq N$ we have \[ \|\mathcal{W}_{g}f_n - \mathcal{W}_{g}f_m\|_{L^{1}(G)} = \|f_n - f_m\|_{\mathcal{H}_{g}^1} < \epsilon.\] Now, by completeness of $L^{1}(G)$, the sequence $\mathcal{W}_{g}f_n$ converges to an element $F \in L^{1}(G)$. Moreover, we see from \eqref{norm_inequality} that there exists an element $f \in \mathcal{H}_{\pi}$ such that $f_n$ converges to $f$ in the norm on $\mathcal{H}_{\pi}$. Hence by the continuity of $\mathcal{W}_{g}$ as a transformation from $\mathcal{H}_{\pi}$ to $L^{2}(G)$, the sequence $\mathcal{W}_{g}f_n$ converges to $\mathcal{W}_{g}f$ in the $L^{2}(G)$-norm. However, since $\mathcal{W}_{g}(\mathcal{H}_{\pi})$ is a reproducing kernel Hilbert space, we know that the convergence $\mathcal{W}_{g}f_n \to \mathcal{W}_{g}f$ is also valid pointwise. This forces $F = \mathcal{W}_{g}f$. Hence $f \in \mathcal{H}_{g}^{1}$ and $f_n \to f$ in the $\mathcal{H}_{g}^{1}$-norm, showing that $\mathcal{H}_{g}^{1}$ is complete.
\end{proof}

The main goal of this section is to show in Theorem \ref{Theorem_about_independence_of_window} that the set of test vectors $\mathcal{H}_{g}^{1}$ does not depend on the choice of integrable vector $g \in \mathcal{A} \setminus \{0\}$. To do this, we first need two preliminary results given in Lemma \ref{lemma_about_domain_of_duflo_more} and Lemma \ref{lemma_about_existence_of_non_orthogonal_element} regarding the Duflo-Moore operator $C_{\pi}$ and integrable vectors. These technicalities are somewhat neglected in the original sources \cite{feichtinger1989banach1, feichtinger1989banach2, feichtinger1988unified} on coorbit theory. To our knowledge, this was first put on rigorous footing in \cite[Lemma 2.4.5]{felix_thesis}. 

\begin{lemma}
\label{lemma_about_domain_of_duflo_more}
Let $\pi:G \to \mathcal{U}(\mathcal{H}_{\pi})$ be an integrable representation. Then \[C_{\pi}(\mathcal{A}) \subset \mathcal{D}(C_{\pi}),\]
where $\mathcal{D}(C_{\pi})$ denotes the domain of the Duflo-Moore operator. 
\end{lemma}

\begin{proof}
Due to the self-adjointness of $C_{\pi}$, it suffices to show that $C_{\pi}(g) \in \mathcal{D}(C_{\pi}^{*})$ for all $g \in \mathcal{A}$. To show this, we prove that the linear functional on $\mathcal{D}(C_{\pi})$ given by \[f \longmapsto \langle C_{\pi}f, C_{\pi}g \rangle_{\mathcal{H}_{\pi}} =  \langle f, C_{\pi}^{*}C_{\pi}g \rangle_{\mathcal{H}_{\pi}}\] is bounded. For $g = 0$ the boundedness clearly holds. For $g \neq 0$ the claim follows from the orthogonality relation \eqref{wavelet_transform_orthogonality} since 
\begin{align*}
    |\langle C_{\pi}f, C_{\pi}g \rangle_{\mathcal{H}_{\pi}}| & = \|g\|_{\mathcal{H}_{\pi}}^{-2} |\langle \mathcal{W}_{g}g, \mathcal{W}_{f}g \rangle_{L^{2}(G)}|
    \\ & \leq \|g\|_{\mathcal{H}_{\pi}}^{-2} \|\mathcal{W}_{g}g\|_{L^{1}(G)} \|\mathcal{W}_{f}g\|_{L^{\infty}(G)}
    \\ & \leq \left(\|g\|_{\mathcal{H}_{\pi}}^{-1}\|g\|_{\mathcal{H}_{g}^{1}}\right) \cdot \|f\|_{\mathcal{H}_{\pi}}. \qedhere
\end{align*}
\end{proof}

\begin{lemma}
\label{lemma_about_existence_of_non_orthogonal_element}
Let $\pi:G \to \mathcal{U}(\mathcal{H}_{\pi})$ be an integrable representation and fix two integrable vectors $g_1,g_2 \in \mathcal{A} \setminus \{0\}$. Then there exists an integrable vector $g \in \mathcal{A} \setminus \{0\}$ such that \[\langle C_{\pi}g, C_{\pi}g_i \rangle \neq 0, \qquad i = 1,2.\]
\end{lemma}

\begin{proof}
If $\langle C_{\pi}g_1, C_{\pi}g_{2} \rangle \neq 0$, then we can simply take $g = g_1$. The injectivity of the Duflo-Moore operator $C_{\pi}$ ensures that $\langle C_{\pi}g_1, C_{\pi}g_1 \rangle \neq 0$. Hence we are left with the case $\langle C_{\pi}g_1, C_{\pi}g_{2} \rangle = 0$. \par 
We point out that Lemma \ref{lemma_about_domain_of_duflo_more} allows us to consider $C_{\pi}(C_{\pi}(g_{2}))$. Notice that neither $C_{\pi}g_2$ nor $C_{\pi}(C_{\pi}(g_{2}))$ can be zero due to the injectivity of $C_{\pi}$. Since the collection $\{\pi(x)g_1\}_{x \in G}$ is dense in $\mathcal{H}_{\pi}$, there exists some fixed $x_0 \in G$ such that \[0 \neq \langle \pi(x_0)g_1, C_{\pi}(C_{\pi}(g_{2})) \rangle = \langle C_{\pi}(\pi(x_0)g_1),C_{\pi}g_2 \rangle,\]
where we used that $C_{\pi}$ is self-adjoint. The desired element we need will be of the form \[g := g_1 + \epsilon \cdot \pi(x_0)g_1\] for some $\epsilon > 0$ that is yet to be determined. First of all, we need to check that $g \in \mathcal{A} \setminus \{0\}$ for every $\epsilon > 0$. This follows from the calculation
\begin{align*}
    \mathcal{W}_{g}g & = \mathcal{W}_{g_1}g_1 + \epsilon \cdot \left( \mathcal{W}_{g_1}\pi(x_0)g_1 + \mathcal{W}_{\pi(x_0)g_1}g_1 \right) + \epsilon^2 \cdot \mathcal{W}_{\pi(x_0)g_1}\pi(x_0)g_1 \\ & = \mathcal{W}_{g_1}g_1 + \epsilon \cdot \left(L_{x_0}\mathcal{W}_{g_1}g_1 + R_{x_0}\mathcal{W}_{g_1}g_1 \right) + \epsilon^2 \cdot L_{x_0}R_{x_0}\mathcal{W}_{g_1}g_1,
\end{align*}
together with the fact that $L^{1}(G)$ is both left-invariant and right-invariant. To see that $g$ satisfies the required properties, we first have that
\begin{equation*}
    \langle C_{\pi}g, C_{\pi}g_2 \rangle = \epsilon \cdot \langle C_{\pi}(\pi(x_0)g_1), C_{\pi}g_2 \rangle \neq 0.
\end{equation*}
Secondly, by choosing $\epsilon$ sufficiently small we also have that
\begin{equation*}
    \langle C_{\pi}g, C_{\pi}g_1 \rangle = \|C_{\pi}g_1\|^2 + \epsilon \cdot \langle C_{\pi}(\pi(x_0)g_1), C_{\pi}g_1 \rangle \neq 0 . \qedhere
\end{equation*}
\end{proof}

We can now state the main result of this section regarding the independence of the test vectors $\mathcal{H}_{g}^{1}$ of the chosen integrable vector $g \in \mathcal{A} \setminus \{0\}$.

\begin{theorem}
\label{Theorem_about_independence_of_window}
Let $\pi:G \to \mathcal{U}(\mathcal{H}_{\pi})$ be an integrable representation. Given two integrable vectors $g_1, g_2 \in \mathcal{A} \setminus \{0\}$ the spaces $\mathcal{H}_{g_1}^1$ and $\mathcal{H}_{g_2}^1$ coincide with equivalent norms.
\end{theorem}

\begin{proof}
Assume first that the two integrable vectors $g_1, g_2 \in \mathcal{A} \setminus \{0\}$ satisfy $\langle C_{\pi}g_1, C_{\pi}g_2 \rangle \neq 0$. We pick $f \in \mathcal{H}_{g_1}^{1}$ and want to show that $f \in \mathcal{H}_{g_2}^{1}$, that is, we need to check that $\mathcal{W}_{g_2}f \in L^{1}(G)$. A short calculation reveals that
\begin{align*}
    \left(\mathcal{W}_{g_1}f *_{G} \mathcal{W}_{g_2}g_2\right)(x) & = \int_{G} \langle f, \pi(y)g_1 \rangle \langle g_2, \pi(y^{-1}x) g_2 \rangle \, d\mu_{L}(y) \\ & = \int_{G} \langle f, \pi(y)g_1 \rangle \overline{\langle \pi(x) g_2, \pi(y) g_2 \rangle} \, d\mu_{L}(y) \\ & = \left\langle \mathcal{W}_{g_1}f, \mathcal{W}_{g_2}(\pi(x)g_2) \right\rangle_{L^{2}(G)} \\ & =   \overline{\langle C_{\pi}g_1, C_{\pi}g_2\rangle} \mathcal{W}_{g_2}f(x).
\end{align*}
Since $\langle C_{\pi}g_1, C_{\pi}g_2 \rangle \neq 0$, we can rearrange and integrate so that \[\|\mathcal{W}_{g_2}f\|_{L^{1}(G)} = \frac{\|\mathcal{W}_{g_1}f *_{G} \mathcal{W}_{g_2}g_2\|_{L^{1}(G)}}{|\langle C_{\pi}g_1, C_{\pi}g_2\rangle|} \leq \frac{\|\mathcal{W}_{g_1}f\|_{L^{1}(G)}\|\mathcal{W}_{g_2}g_2\|_{L^{1}(G)}}{|\langle C_{\pi}g_1, C_{\pi}g_2\rangle|}.\] \par 
Let us now tackle the case where $g_1, g_2 \in \mathcal{A} \setminus \{0\}$ satisfy $\langle C_{\pi}g_1, C_{\pi}g_2 \rangle = 0$. Again, we assume that $f \in \mathcal{H}_{g_1}^{1}$ and we want to show that $f \in \mathcal{H}_{g_2}^{1}$. We can by Lemma \ref{lemma_about_existence_of_non_orthogonal_element} pick an integrable vector $g \in \mathcal{A} \setminus \{0\}$ such that $\langle C_{\pi}g, C_{\pi}g_i \rangle \neq 0$ for $i = 1,2$. Performing similar calculations as previously, we obtain
\begin{align*}
    \left(\mathcal{W}_{g_1}f *_{G} \mathcal{W}_{g}g\right) *_{G} \mathcal{W}_{g_2}g_2 & = \overline{\langle C_{\pi}g_1, C_{\pi}g \rangle} \mathcal{W}_{g}f *_{G} \mathcal{W}_{g_2} g_2 \\ & = \overline{\langle C_{\pi}g_1, C_{\pi}g \rangle} \overline{\langle C_{\pi}g, C_{\pi}g_2 \rangle} \mathcal{W}_{g_2}f. 
\end{align*}
We have conceptually used $g$ as a stepping stone between $g_1$ and $g_2$. After a rearrangement, we can integrate and obtain
\begin{align*}
    \|\mathcal{W}_{g_2}f\|_{L^{1}(G)} & = \frac{\|\mathcal{W}_{g_1}f *_{G} \mathcal{W}_{g}g *_{G} \mathcal{W}_{g_2}g_2\|_{L^{1}(G)}}{|\langle C_{\pi}g_1, C_{\pi}g \rangle||\langle C_{\pi}g, C_{\pi}g_2 \rangle|} \\ & \leq  \frac{\|\mathcal{W}_{g_1}f\|_{L^{1}(G)}\|\mathcal{W}_{g}g\|_{L^{1}(G)} \|\mathcal{W}_{g_2}g_2\|_{L^{1}(G)}}{|\langle C_{\pi}g_1, C_{\pi}g \rangle||\langle C_{\pi}g, C_{\pi}g_2 \rangle|}.
\end{align*}
It clear from the arguments above that the norms on $\mathcal{H}_{g_1}^1$ and $\mathcal{H}_{g_2}^1$ are equivalent. 
\end{proof}

Due to the independence of $\mathcal{H}_{g}^{1}$ of the integrable vector $g \in \mathcal{A} \setminus \{0\}$, we will use the notation \[\mathcal{H}^{1} := \mathcal{H}_{g}^{1}.\] It follows from Theorem \ref{Theorem_about_independence_of_window} that $\mathcal{A} \subset \mathcal{H}^{1}$ since $g \in \mathcal{A}$ is in $\mathcal{H}_{g}^{1}$ by definition. For unimodular groups, the following result shows that we do not need to keep track of both $\mathcal{H}^{1}$ and $\mathcal{A}$.

\begin{proposition}
\label{equality_of_spaces}
We have the equality $\mathcal{A} = \mathcal{H}^{1}$ when $\pi:G \to \mathcal{U}(\mathcal{H}_{\pi})$ is an integrable representation of a unimodular group $G$.
\end{proposition}

\begin{proof}
We fix $f \in \mathcal{H}^{1}$ and want to show that $f \in \mathcal{A}$. The orthogonality relation in \eqref{wavelet_transform_orthogonality} for $x \in G$ shows that \[\langle C_{\pi}g, C_{\pi}g \rangle_{\mathcal{H}_{\pi}} \langle f, \pi(x)f \rangle_{\mathcal{H}_{\pi}} = \left\langle \mathcal{W}_{g}f, \mathcal{W}_{g}(\pi(x)f)\right\rangle_{L^{2}(G)}.\] Taking the absolute value and using the intertwining property \eqref{wavelet_intertwines_left_reg}, we have \[\|C_{\pi}g\|_{\mathcal{H}_{\pi}}^{2}|\langle f, \pi(x)f \rangle_{\mathcal{H}_{\pi}}| \leq \int_{G}|\mathcal{W}_{g}f(y)| |\mathcal{W}_{g}f(x^{-1}y)| \, d\mu(y).\] Notice that $\|C_{\pi}g\|^2 \neq 0$ since $C_{\pi}$ is injective. Hence we can divide by $\|C_{\pi}g\|^2$ and integrate with respect to $x$, use Fubini's theorem, and use the right-invariance of the measure $\mu$ to obtain 
\begin{align*}
    \|\mathcal{W}_{f}f\|_{L^{1}(G)} & \leq \frac{1}{\|C_{\pi}g\|^2}\int_{G}\int_{G}|\mathcal{W}_{g}f(y)| |\mathcal{W}_{g}f(x^{-1}y)| \, d\mu(y) \, d\mu(x)  \\ & = \frac{1}{\|C_{\pi}g\|^2}\int_{G}|\mathcal{W}_{g}f(y)|\left(\int_{G} |\mathcal{W}_{g}f(x^{-1}y)| \, d\mu(x) \right)\, d\mu(y) \\ & = \frac{1}{\|C_{\pi}g\|^2}\int_{G} |\mathcal{W}_{g}f(x^{-1})| \, d\mu(x) \int_{G}\left|\mathcal{W}_{g}f(y)\right|\, d\mu(y).
\end{align*}
Since $G$ is unimodular, we can use the substitution $x \mapsto x^{-1}$ to obtain \[\|\mathcal{W}_{f}f\|_{L^{1}(G)} \leq \frac{1}{\|C_{\pi}g\|^2}\int_{G} \left|\mathcal{W}_{g}f(x)\right| \, d\mu(x) \int_{G}\left|\mathcal{W}_{g}f(y)\right|\, d\mu(y) \leq \frac{1}{\|C_{\pi}g\|^2} \|\mathcal{W}_{g}f\|_{L^{1}(G)}^{2} < \infty.\]
Thus $f \in \mathcal{A}$ and the claim follows. 
\end{proof}

\begin{example}
\label{Feichtinger_algebra_example}
Let us consider the Schr\"{o}dinger representation $\rho_{r}$ of the reduced Heisenberg group $\mathbb{H}_{r}^{n}$. It follows from \eqref{generalized_wavelet_Schrodinger} that for any $g \in L^{2}(\mathbb{R}^n)$ we have $\mathcal{W}_{g}g \in L^{1}(\mathbb{H}_{r}^{n})$ precisely whenever $V_{g}g \in L^{1}(\mathbb{R}^{2n})$, where $V$ denotes the STFT and $\mathcal{W}$ denotes the wavelet transform corresponding to $\rho_r$. Motivated by this observation, we will work with the STFT instead of the wavelet transform. \par 
It is straightforward to check that $V_{g}g \in \mathcal{S}(\mathbb{R}^{2n}) \subset L^{1}(\mathbb{R}^{2n})$ whenever $g \in \mathcal{S}(\mathbb{R}^{n})$ is a smooth and rapidly decaying function, for details see \cite[Theorem 11.2.5]{grochenig2001foundations}. Hence $\rho_r$ is an integrable representation. We can by Theorem \ref{Theorem_about_independence_of_window} and Proposition \ref{equality_of_spaces} unambiguously define the \textit{Feichtinger algebra} 
\[M^{1}(\mathbb{R}^n) := \mathcal{H}^{1} = \mathcal{A} = \left\{f \in L^{2}(\mathbb{R}^{n}) \, : \, V_{f}f \in L^{1}(\mathbb{R}^{2n}) \right\}.\] \par 
We obtain from Proposition \ref{prop:completeness_H1} that $M^{1}(\mathbb{R}^n)$ is a Banach space. The Feichtinger algebra $M^{1}(\mathbb{R}^n)$ was first introduced in \cite{feichtinger1981new} and gained more widespread attention after its appearance in \cite{grochenig2001foundations}. We refer the reader to \cite{jakobsen2018no} for a detailed and modern exposition on the Feichtinger algebra. In particular, functions in $M^{1}(\mathbb{R})$ are automatically continuous by \cite[Corollary 4.2]{jakobsen2018no}. Since there are plenty of non-continuous\footnote{More precisely, there is a dense subset $D \subset L^{2}(\mathbb{R}^{n})$ of equivalence classes of functions that does not have a continuous representative.} elements in $L^{2}(\mathbb{R}^n)$, this gives an example where $\mathcal{H}^{1} \neq \mathcal{H}_{\pi}$.
\end{example}

\subsection{Reservoirs and the Extended Wavelet Transform}
\label{sec: Reservoirs_and_the_Extended_Wavelet_Transform}

Let $\pi:G \to \mathcal{U}(\mathcal{H}_{\pi})$ be an integrable representation and fix an integrable vector $g \in \mathcal{A} \setminus \{0\}$. In light of the previous section, we might prematurely define the coorbit space $\mathcal{C}o_{p}(G)$ for $1 \leq p \leq \infty$ to be all $f \in \mathcal{H}_{\pi}$ such that $\mathcal{W}_{g}f \in L^{p}(G)$. However, this naive definition suffers from the following problem: We will obtain $\mathcal{C}o_{p}(G) = \mathcal{H}_{\pi}$ for every $p \geq 2$. Only having interesting coorbit spaces in the range $1 \leq p \leq 2$ shatters any dream of good duality results; see Proposition \ref{duality_statement} for what we are missing out on. The problem is that the space $\mathcal{H}_{\pi}$ is to small to accommodate a full range $1 \leq p \leq \infty$ of interesting spaces. In this section, we will fix this problem by introducing a larger reference space $\mathcal{R}$ and ensuring that everything works the way it should. After this is done, we can confidently define the coorbit spaces properly in Section \ref{sec: Coorbit_Spaces}.

\begin{definition}
    Let $\pi: G \to \mathcal{U}(\mathcal{H}_{\pi})$ be an integrable representation. The space of bounded anti-linear functionals on $\mathcal{H}^{1}$ is denoted by $\mathcal{R}$ and called the \textit{reservoir space}.
\end{definition}

\begin{remark}
Implicitly, we have chosen an integrable vector $g \in \mathcal{A} \setminus \{0\}$ and are considering $\mathcal{H}_{g}^{1}$ and the space $\mathcal{R}_{g}$ of bounded anti-linear functionals on $\mathcal{H}_{g}^{1}$. However, due to Theorem \ref{Theorem_about_independence_of_window} we omit $g$ from the notation as it is of minor importance. The reservoir space $\mathcal{R}$ will seldom consist of functions in any reasonable sense. If we want to understand when two elements $\phi,\psi \in \mathcal{R}$ are equal, we need to test them on all the elements in $\mathcal{H}^{1}$. This is the motivation for calling $\mathcal{H}^{1}$ the space of \textit{test vectors}.
\end{remark}

\begin{lemma}
\label{lemma_inclusions}
There are natural continuous embeddings
\[\mathcal{H}^{1} \xhookrightarrow{} \mathcal{H}_{\pi} \xhookrightarrow{} \mathcal{R}.\]
\end{lemma}

\begin{proof}
If $\phi \in \mathcal{R}$ and $g \in \mathcal{H}^{1}$, we denote the dual pairing $\phi(g)$ by $\langle \phi, g \rangle$. We can embed $\mathcal{H}_{\pi}$ into $\mathcal{R}$ by letting $f \in \mathcal{H}_{\pi}$ act on $g \in \mathcal{H}^{1}$ by \[f(g) := \langle f, g \rangle_{\mathcal{H}_{\pi}}.\] To see that the inclusion $\mathcal{H}_{\pi} \xhookrightarrow{} \mathcal{R}$ is continuous we compute for $f \in \mathcal{H}_{\pi}$ that \[\|f\|_{\mathcal{R}} = \sup_{g \in \mathcal{H}^{1} \setminus \{0\}}\frac{|\langle f,g \rangle|}{\|g\|_{\mathcal{H}^{1}}} \leq \left(\sup_{g \in \mathcal{H}^{1} \setminus \{0\}} \frac{\|g\|_{\mathcal{H}_{\pi}}}{\|g\|_{\mathcal{H}^{1}}} \right)\|f\|_{\mathcal{H}_{\pi}}.\]
The claim follows the continuity of the inclusion $\mathcal{H}^{1} \xhookrightarrow{} \mathcal{H}_{\pi}$ in Proposition \ref{prop:completeness_H1}.
\end{proof}

Given an integrable representation $\pi:G \to \mathcal{U}(\mathcal{H}_{\pi})$ we can let $\pi$ act on the reservoir space $\mathcal{R}$ through duality. More precisely, for $x \in G$ and $\phi \in \mathcal{R}$ we define $\pi(x)\phi$ to be the element in $\mathcal{R}$ that acts on $g \in \mathcal{H}^{1}$ by \[(\pi(x)\phi)(g) = \langle \pi(x)\phi,g \rangle := \langle \phi, \pi(x^{-1})g \rangle.\] This gives an isometric action on $\mathcal{R}$ since \[\|\pi(x)\phi\|_{\mathcal{R}} = \sup_{g \in \mathcal{H}^{1} \setminus \{0\}} \frac{|\langle \pi(x)\phi, g \rangle|}{\|g\|_{\mathcal{H}^{1}}} = \sup_{g \in \mathcal{H}^{1} \setminus \{0\}} \frac{|\langle \phi, \pi(x^{-1})g \rangle|}{\|\pi(x^{-1})g\|_{\mathcal{H}^{1}}} = \|\phi\|_{\mathcal{R}},\]
where we used that $\pi$ acts by isometries on $\mathcal{H}^{1}$, see Proposition \ref{inclusion_proposition}. We can now extend the wavelet transform to a duality pairing between $\mathcal{H}^{1}$ and $\mathcal{R}$ as follows: 

\begin{definition}
Let $\pi:G \to \mathcal{U}(\mathcal{H}_{\pi})$ be an integrable representation. For $\phi \in \mathcal{R}$ and $g \in \mathcal{H}^1$ we define the \textit{(extended) wavelet transform} to be the function on $G$ given by \[\mathcal{W}_{g}\phi(x) := \langle \phi, \pi(x)g \rangle = \phi\left(\pi(x)g\right) = (\pi(x^{-1})\phi)(g), \qquad x \in G.\]
\end{definition}

Notice that the definition of the extended wavelet transform is well-defined since $\mathcal{H}^{1}$ is invariant under $\pi$. Some authors, e.g.\ \cite{dahlke2015harmonic}, change the notation for the extended wavelet transform to emphasize its domain, while other authors \cite{felix_thesis} do not change the notation. We have opted for the latter and will strive to make it clear what the wavelet transform acts on. 

\begin{proposition}
\label{proposition_basic_props_of_extended_wavelet_transform}
Let $\pi:G \to \mathcal{U}(\mathcal{H}_{\pi})$ be an integrable representation and fix $g \in \mathcal{H}^{1}$. Then $\mathcal{W}_{g}(\mathcal{R}) \subset C_{b}(G)$ and we have the intertwining property
\begin{equation}
\label{extended_intertwining_property}
    \mathcal{W}_{g}(\pi(x)\phi) = L_{x}\left[\mathcal{W}_{g}\phi\right],
\end{equation}
for $x \in G$ and $\phi \in \mathcal{R}$.
\end{proposition}

\begin{proof}
The map $\Gamma_{g}: x \mapsto \pi(x)g$ is clearly a continuous map $\Gamma_{g}: G \to \mathcal{H}^{1}$ by Proposition \ref{inclusion_proposition}. Hence $\mathcal{W}_{g}\phi = \phi \circ \Gamma_{g}$ is continuous since it can be described as the composition of two continuous maps. The boundedness of $\mathcal{W}_{g}\phi$ follows from the straightforward computation
\begin{equation*}
    |\mathcal{W}_{g}\phi(x)| = |\phi(\pi(x)g)| \leq \|\phi\|_{\mathcal{R}} \|\pi(x)g\|_{\mathcal{H}^{1}} = \|\phi\|_{\mathcal{R}} \|g\|_{\mathcal{H}^{1}}, \qquad x \in G.
\end{equation*}
Finally, the intertwining property is verified by the computation 
\begin{equation*}
    \left(\mathcal{W}_{g}(\pi(x)\phi)\right)(y) = \langle \pi(x)\phi, \pi(y)g \rangle = \langle \phi, \pi(x^{-1}y)g \rangle = L_{x}\mathcal{W}_{g}\phi(y), \qquad x,y \in G. \qedhere
\end{equation*}
\end{proof}

\begin{remark}
Although the (extended) wavelet transform $\mathcal{W}_{g}$ is well-defined for all $g \in \mathcal{H}^{1}$, we will for the most part work with the setting where $g \in \mathcal{A} \subset \mathcal{H}^1$ for convenience. Hence we will primarily state results for $\mathcal{W}_{g}$ when $g \in \mathcal{A}$, even though they are sometimes valid for $g \in \mathcal{H}^{1}$ as well.
\end{remark}

\begin{example}
We defined in Example \ref{Feichtinger_algebra_example} the Feichtinger algebra $M^{1}(\mathbb{R}^n)$ as the set of test vectors corresponding to the STFT. The reservoir space $\mathcal{R}$ in this setting will be denoted by $M^{\infty}(\mathbb{R}^{n})$. \par 
Let us do a concrete calculation in the case $n = 1$: The \textit{Dirac Comb distribution} $\delta_{\mathbb{Z}}$ is defined formally as acting on functions $f:\mathbb{R} \to \mathbb{C}$ by 
\begin{equation}
\label{eq:dirac_comb}
    \delta_{\mathbb{Z}}(f) := \sum_{n = -\infty}^{\infty}f(n).
\end{equation}
The expression \eqref{eq:dirac_comb} is obviously not always well defined. It follows from \cite[Corollary 5.9]{jakobsen2018no} that $\delta_{\mathbb{Z}} \in M^{\infty}(\mathbb{R})$. For $g(t) := e^{-t^2} \in \mathcal{S}(\mathbb{R}) \subset M^{1}(\mathbb{R})$ and $(x,\omega) \in \mathbb{R}^{2}$ we have the explicit computation
\begin{align*}
    V_{g}\delta_{\mathbb{Z}}(x,\omega) = \delta_{\mathbb{Z}}\left(M_{-\omega}T_{-x}g\right) = \delta_{\mathbb{Z}}\left(e^{-2 \pi i \omega t}e^{-(t - x)^{2}}\right) = \sum_{n = -\infty}^{\infty}e^{-2 \pi i \omega n}e^{-(n - x)^{2}}.
\end{align*}
An interesting observation is that \[V_{g}\delta_{\mathbb{Z}}(0,\omega) = \vartheta(z, \tau),\] where $\tau = i/\pi$, $z = -\omega$, and $\vartheta$ is the \textit{Jacobi theta function} omnipresent in complex analysis.
\end{example}

\begin{lemma}
\label{lemma_convergence_in_H1_and_convolution_relation}
Let $\pi:G \to \mathcal{U}(\mathcal{H}_{\pi})$ be an integrable representation and fix $g \in \mathcal{A} \setminus \{0\}$. Then linear combinations of elements of the form $\pi(x)g$ for $x \in G$ constitute a dense subspace of $\mathcal{H}^{1}$ with respect to the norm on $\mathcal{H}^{1}$. Moreover, if $g$ is admissible then we have the reproducing formula
\begin{equation*}
    \mathcal{W}_{g}\phi = \mathcal{W}_{g}\phi *_{G} \mathcal{W}_{g}g,
\end{equation*}
for any $\phi \in \mathcal{R}$.
\end{lemma}

\begin{remark}
Originally the density statement in Lemma \ref{lemma_convergence_in_H1_and_convolution_relation} was proved by showing a minimality statement regarding the space $\mathcal{H}^{1}$. More precisely, it was shown in \cite[Corollary 4.8]{feichtinger1988unified} that $\mathcal{H}^{1}$ is the minimal $\pi$-invariant Banach space inside $\mathcal{H}_{\pi}$ where $\pi$ acts isometrically and such that $\mathcal{A} \cap \mathcal{H}^{1} \neq \{0\}.$ A different proof of the density statement in Lemma \ref{lemma_convergence_in_H1_and_convolution_relation} was given in \cite[Lemma 2.4.7]{felix_thesis} using Bochner integration. The reader can also find a proof of the convolution statement in \cite[Lemma 2.4.8]{felix_thesis}, again using Bochner integration. We have opted to not present a proof of Lemma \ref{lemma_convergence_in_H1_and_convolution_relation} as it is mostly a technical tool.
\end{remark}

\begin{corollary}
\label{injection_on_L_infinity_corollary}
Let $\pi:G \to \mathcal{U}(\mathcal{H}_{\pi})$ be an integrable representation and fix an integrable vector $g \in \mathcal{A} \setminus \{0\}$. Then $\mathcal{W}_{g}:\mathcal{R} \to L^{\infty}(G)$ is injective.
\end{corollary}

\begin{proof}
Assume that $\mathcal{W}_{g}\phi(x) = \phi(\pi(x)g) = 0$ for every $x \in G$. Then Lemma \ref{lemma_convergence_in_H1_and_convolution_relation} shows that $\phi = 0$ since the span of the elements $\pi(x)g$ for $x \in G$ is a dense subspace of $\mathcal{H}^{1}$.
\end{proof}

Notice that for an integrable vector $g \in \mathcal{A} \setminus \{0\}$ we have by definition that $\mathcal{W}_g: \mathcal{H}^{1} \to L^{1}(G)$. Hence we can consider the adjoint map $\mathcal{W}_{g}^{*}: L^{\infty}(G) \to \mathcal{R}$ defined by the relation
\[\langle \mathcal{W}_{g}^{*}(F), f \rangle_{\mathcal{R}, \mathcal{H}^{1}} = \langle F, \mathcal{W}_{g}f \rangle_{L^{\infty}(G), L^{1}(G)} = \int_{G}F(x)\overline{\mathcal{W}_{g}f(x)} \, d\mu_{L}(x) = \int_{G}F(x)\langle \pi(x)g, f \rangle \, d\mu_{L}(x),\]
for $F \in L^{\infty}(G)$ and $f \in \mathcal{H}^{1}$. The adjoint map $\mathcal{W}_{g}^{*}: L^{\infty}(G) \to \mathcal{R}$ can hence be written weakly as \[\mathcal{W}_{g}^{*}(F) = \int_{G}F(x)\pi(x)g \, d\mu_{L}(x), \qquad F \in L^{\infty}(G).\] 

\begin{proposition}
\label{results_about_adjoint}
Let $\pi:G \to \mathcal{U}(\mathcal{H}_{\pi})$ be an integrable representation and fix an integrable vector $g \in \mathcal{A} \setminus \{0\}$. The adjoint map $\mathcal{W}_{g}^{*}: L^{\infty}(G) \to \mathcal{R}$ satisfies
\begin{equation*}
    \mathcal{W}_{g}\left(\mathcal{W}_{g}^{*}(F)\right) = F *_{G} \mathcal{W}_{g}g, \qquad \mathcal{W}_{g}^{*}(\mathcal{W}_{g}\phi) = \phi,
\end{equation*}
for $F \in L^{\infty}(G)$ and $\phi \in \mathcal{R}$.
\end{proposition}

\begin{proof}
For $x \in G$ a straightforward computation shows that
\begin{equation}
\label{convolution_in_L_infinity}
    \mathcal{W}_{g}\left(\mathcal{W}_{g}^{*}(F)\right)(x) = \langle \mathcal{W}_{g}^{*}(F), \pi(x)g \rangle_{\mathcal{R}, \mathcal{H}^{1}} = \langle F, \mathcal{W}_{g}(\pi(x)g) \rangle_{L^{\infty}(G), L^{1}(G)} = (F *_{G} \mathcal{W}_{g}g)(x).
\end{equation}
Finally, we need to show that the map $\mathcal{W}_{g}^{*} \circ \mathcal{W}_{g}:\mathcal{R} \to \mathcal{R}$ is in fact the identity map. For $\phi \in \mathcal{R}$ we have from \eqref{convolution_in_L_infinity} and Lemma \ref{lemma_convergence_in_H1_and_convolution_relation} that \[\mathcal{W}_{g}(\mathcal{W}_{g}^{*}(\mathcal{W}_{g}\phi)) = \mathcal{W}_{g}\phi *_{G} \mathcal{W}_{g}g = \mathcal{W}_{g}\phi.\] The injectivity of $\mathcal{W}_{g}:\mathcal{R} \to L^{\infty}(G)$ ensured by Corollary \ref{injection_on_L_infinity_corollary} shows that $\mathcal{W}_{g}^{*}(\mathcal{W}_{g}\phi) = \phi$.
\end{proof}

The following result reveals a deep connection between the extended wavelet transform and convolutions on the group $G$. 

\begin{theorem}
\label{resovoir_vs_L_infinity}
Let $\pi:G \to \mathcal{U}(\mathcal{H}_{\pi})$ be an integrable representation and fix an integrable vector $g \in \mathcal{A} \setminus \{0\}$. A function $F \in L^{\infty}(G)$ satisfies the convolution relation $F = F *_{G} \mathcal{W}_{g}g$ precisely when it can be written uniquely as $F = \mathcal{W}_{g}\phi$ for some $\phi \in \mathcal{R}$.
\end{theorem}

\begin{proof}
If $F \in L^{\infty}(G)$ is such that $F = F *_{G} \mathcal{W}_{g}g$, then Proposition \ref{results_about_adjoint} shows that $F = \mathcal{W}_{g}(\phi)$ where $\phi := \mathcal{W}_{g}^{*}(F)$. Moreover, the description $F = \mathcal{W}_{g}\phi$ is necessarily unique due to the injectivity of $\mathcal{W}_{g}:\mathcal{R} \to L^{\infty}(G)$. Conversely, assume that $F \in L^{\infty}(G)$ satisfies $F = \mathcal{W}_{g}\phi$ for some $\phi \in \mathcal{R}$. Then we have from Proposition \ref{results_about_adjoint} that \[\mathcal{W}_{g}^{*}(F) = \mathcal{W}_{g}^{*}\left(\mathcal{W}_{g}\phi\right) = \phi.\]
Thus $\mathcal{W}_{g}(\mathcal{W}_{g}^{*}(F)) = \mathcal{W}_{g}\phi = F$. The claim follows from a final application of Proposition \ref{results_about_adjoint}.
\end{proof}

\begin{remark}
We mentioned in Example \ref{Feichtinger_algebra_example} that the space $\mathcal{S}(\mathbb{R}^{n})$ is included in the Feichtinger algebra $M^{1}(\mathbb{R}^{n})$. Hence we have by Lemma \ref{lemma_inclusions} the inclusions \[\mathcal{S}(\mathbb{R}^{n}) \subset M^{1}(\mathbb{R}^{n}) \subset L^{2}(\mathbb{R}^{n}) \subset M^{\infty}(\mathbb{R}^{n}) \subset \mathcal{S}'(\mathbb{R}^{n}),\]
where the set of \textit{tempered distributions} $\mathcal{S}'(\mathbb{R}^{n})$ is the dual space of $\mathcal{S}(\mathbb{R}^{n})$. We can view the pair $\left(M^{1}(\mathbb{R}^{n}), M^{\infty}(\mathbb{R}^{n})\right)$ as a refinement of the pair $\left(\mathcal{S}(\mathbb{R}^{n}), \mathcal{S}'(\mathbb{R}^{n})\right)$. A time-frequency analysis enthusiast might even use the word \textquote{improvement} since the Feichtinger algebra $M^{1}(\mathbb{R}^{n})$ is, in contrast with $\mathcal{S}(\mathbb{R}^{n})$, a Banach space. 
\end{remark}

\subsection{Coorbit Spaces and the Correspondence Principle}
\label{sec: Coorbit_Spaces}

Now that all the pieces are in place we will define the coorbit spaces. These are the main objects of study for this survey, and we spend a decent amount of time deriving their basic properties.

\begin{definition}
Let $\pi:G \to \mathcal{U}(\mathcal{H}_{\pi})$ be an integrable representation and fix an integrable vector $g \in \mathcal{A} \setminus \{0\}$. The \textit{coorbit space} $\mathcal{C}o_{p}(G)$ consists of all elements in the reservoir space $\phi \in \mathcal{R}$ such that $\mathcal{W}_{g}\phi$ decays fast enough to be in $L^{p}(G)$. Precisely, we define for each $1 \leq p \leq \infty$ the space \[\mathcal{C}o_{p}(G) := \mathcal{C}o_{p}^{\pi}(G) := \left\{\phi \in \mathcal{R} \, : \, \mathcal{W}_{g}\phi \in L^{p}(G)\right\},\] with the norm \[\|\phi\|_{\mathcal{C}o_{p}(G)} := \|\mathcal{W}_{g}\phi\|_{L^{p}(G)}.\] 
\end{definition}

We will only use the full notation $\mathcal{C}o_{p}^{\pi}(G)$ in Section~\ref{sec: A Kernel Theorem for Coorbit Spaces} when we are dealing with multiple representations. The observant reader will have noticed that we did not mention the integrable vector $g \in \mathcal{A} \setminus \{0\}$ in the notation $\mathcal{C}o_{p}(G)$. This is because, as probably suspected, the coorbit spaces $\mathcal{C}o_{p}(G)$ do not depend on the choice of integrable vector, see \cite[Section 5.2]{feichtinger1988unified} for details. 

\begin{example}
Let $G$ be a compact group and let $\pi:G \to \mathcal{U}(\mathcal{H}_{\pi})$ be an irreducible representation. Then $\pi$ is automatically integrable since any $g \in \mathcal{H}_{\pi}$ satisfies 
\begin{equation*}
    \int_{G}|\mathcal{W}_{g}g(x)| \, d\mu_{L}(x) \leq \|\mathcal{W}_{g}g\|_{L^{\infty}(G)}\cdot\mu_{L}(G) < \infty.
\end{equation*}
Here we used that the Haar measure of a compact group is finite, see \cite[Proposition 1.4.5]{deitmar2014principles}. Moreover, it is clear that every $g \in \mathcal{H}_{\pi}$ satisfies $\mathcal{W}_{g}g \in L^{p}(G)$ for all $1 \leq p \leq \infty$. Thus all the coorbit spaces coincide, that is, $\mathcal{C}o_{p}(G) = \mathcal{H}_{\pi}$ for all $1 \leq p \leq \infty$. Moreover, we mentioned in Example \ref{compact_duflo_moore_example} that the space $\mathcal{H}_{\pi}$ is necessarily finite-dimensional whenever $G$ is compact. Hence coorbit spaces are rather dull when considering compact groups. \par Coorbit spaces associated with a commutative group $G$ are even more boring: In this case Corollary \ref{corollary_abelian_reps} ensures that $\mathcal{H}_{\pi}$ is one-dimensional. From this, it is easy to check that an integrable representation $\pi: G \to \mathcal{U}(\mathcal{H}_{\pi})$ can only exist whenever $G$ is compact. Henceforth we will only be interested in coorbit spaces corresponding to locally compact groups that are both non-compact and non-commutative.
\end{example}

\begin{remark}
Before we proceed, it is instructive to consider how the definition of the coorbit spaces can be generalized. 
\begin{itemize}
    \item One could allow $p$ to take values in $(0,1)$ as well. This would make the spaces $\mathcal{C}o_{p}(G)$ for $p \in (0,1)$ quasi-normed spaces instead of normed spaces. We will not consider this extension, and refer the reader to \cite{felix_thesis} for basic results in this direction.
    \item We can consider weighted coorbit spaces $\mathcal{C}o_{p,w}(G)$ where $w:G \to (0, \infty)$ is a weight function. To do this, one must first incorporate weights into the definition of analyzing vectors $\mathcal{A}_{w}$ and test vectors $\mathcal{H}_{w}^{1}$. We will briefly go through this extension in Section \ref{sec: Weighted_Versions}. The weighted extension offer mostly technical challenges rather than conceptual ones. As such, we feel content with supplying the proofs only in the unweighted setting. We will however provide the reader the proper references whenever we leave out details.
    \item One could go a step further and consider the coorbit space $\mathcal{C}o(Y)$, where $Y$ is a \textit{solid} and \textit{translation invariant} Banach space of functions on $G$. We omit the precise definitions here and refer the reader to the original papers \cite{feichtinger1988unified, feichtinger1989banach1, feichtinger1989banach2} as well as Voigtlaender's Ph.D. thesis \cite{felix_thesis} for more on the theory in this level of generality. Most concrete applications of coorbit theory use weighted $L^{p}$-spaces, or mixed-norm $L^{p,q}$ spaces as in the following example.
\end{itemize} 
\end{remark}

\begin{example}
\label{example_modulation_spaces_definition}
Let us again consider the STFT. In this case, we have the notation $\mathcal{H}^{1} = M^{1}(\mathbb{R}^{n})$ and $\mathcal{R} = M^{\infty}(\mathbb{R}^{n})$. The coorbit spaces in this setting are called the \textit{modulation spaces}. More explicitly, for a non-zero $g \in M^{1}(\mathbb{R}^{n})$ and $1 \leq p \leq \infty$ the space $M^{p}(\mathbb{R}^{n})$ consists of elements $f \in M^{\infty}(\mathbb{R}^n)$ such that \begin{equation*}
    \|f\|_{M^{p}(\mathbb{R}^{n})} := \left(\int_{\mathbb{R}^{2n}}|V_{g}f(x,\omega)|^{p} \, dx \, d\omega\right)^{\frac{1}{p}} < \infty.
\end{equation*}
It will be clear from Proposition \ref{concrete_descriptions} that $M^{2}(\mathbb{R}^n) = L^{2}(\mathbb{R}^{n})$. \par
We can generalize the modulation spaces slightly by using mixed-norm $L^{p,q}$ spaces. More precisely, we define the \textit{mixed-norm modulation spaces} $M^{p,q}(\mathbb{R}^{n})$ for $1 \leq p,q \leq \infty$ as the elements $f \in M^{\infty}(\mathbb{R}^n)$ such that \[\|f\|_{M^{p,q}(\mathbb{R}^{n})} := \left(\int_{\mathbb{R}^n}\left(\int_{\mathbb{R}^{n}}|V_{g}f(x,\omega)|^{p} \, dx \right)^{\frac{q}{p}} \, d\omega \right)^{\frac{1}{q}} < \infty.\]
Notice that $M^{p,p}(\mathbb{R}^{n}) = M^{p}(\mathbb{R}^n)$. This extension allows us to consider different levels of integrability in time and frequency.
We remark that the space $M^{\infty,1}(\mathbb{R}^{n})$ has appeared in the theory of pseudodifferential operators under the name \textit{Sj\"{o}strand's class}. We refer the reader to \cite{grochenig2006} for more on Sj\"{o}strand's class in the context of time-frequency analysis. More general information regarding the mixed modulation spaces $M^{p,q}(\mathbb{R}^{n})$ can be found in \cite[Chapter 11 and 12]{grochenig2001foundations}.
\end{example}

Most of the basic properties of coorbit spaces will be derived in Section \ref{sec: Basic_properties_of_the_coorbit_spaces}. Before this, we will establish a powerful result known as the \textit{correspondence principle}. In essence, the correspondence principle states that one can identify the abstract coorbit space $\mathcal{C}o_{p}(G)$ with the space \[\mathcal{M}_{p}(G) := \{F \in L^{p}(G) \,: \, F = F *_{G} \mathcal{W}_{g}g \}.\] Notice that $\mathcal{M}_{p}(G)$ is more concrete that $\mathcal{C}o_{p}(G)$, in the sense that it consists of functions on the group $G$ in question. The fact that the wavelet transform $\mathcal{W}_{g}$ for $g \in \mathcal{A} \setminus \{0\}$ provides the isomorphism between $\mathcal{C}o_{p}(G)$ and $\mathcal{M}_{p}(G)$ makes the result even more conceptually pleasing. 

\begin{theorem}[Correspondence Principle]
\label{corresponance_principle}
Let $\pi:G \to \mathcal{U}(\mathcal{H}_{\pi})$ be an integrable representation and fix an integrable vector $g \in \mathcal{A} \setminus \{0\}$. Then for every $1 \leq p \leq \infty$ the wavelet transform $\mathcal{W}_{g}$ is an isomorphism 
\begin{equation*}
    \mathcal{W}_{g}:\mathcal{C}o_{p}(G) \xrightarrow{\sim} \mathcal{M}_{p}(G).
\end{equation*}
\end{theorem}

\begin{proof}
It follows immediately from Theorem \ref{resovoir_vs_L_infinity} that $\mathcal{W}_{g}(\mathcal{C}o_{p}(G)) \subset \mathcal{M}_{p}(G)$. Hence it only remains to show that any $F \in \mathcal{M}_{p}(G)$ is in fact of the form $F = \mathcal{W}_{g}f$ for some $f \in \mathcal{C}o_{p}(G)$. Notice that $\mathcal{W}_{g}g \in L^{q}(G)$ for all $1 \leq q \leq \infty$ since \[\mathcal{W}_{g}g \in L^{1}(G) \cap L^{\infty}(G) \subset L^{q}(G).\]
We choose $q$ such that $p^{-1} + q^{-1} = 1$, with the obvious caveats for $p = 1, \infty$. Then
\[F = F * \mathcal{W}_{g}g \in L^{\infty}(G).\]
Hence the machinery in Theorem \ref{resovoir_vs_L_infinity} implies that $F = \mathcal{W}_{g}f$ for some $f \in \mathcal{R}$. We have that $f \in \mathcal{C}o_{p}(G)$ by definition since $F \in L^{p}(G)$.
\end{proof}

\subsection{Basic Properties of Coorbit Spaces}
\label{sec: Basic_properties_of_the_coorbit_spaces}

In this section we derive the basic properties of coorbit spaces. The reader should pay special attention to how the correspondence principle we proved in Theorem \ref{corresponance_principle} is utilized in several of the proofs in this section.

\begin{theorem}
\label{completeness_of_coorbit_spaces}
Let $\pi:G \to \mathcal{U}(\mathcal{H}_{\pi})$ be an integrable representation. Then the coorbit spaces $\mathcal{C}o_{p}(G)$ are $\pi$-invariant Banach spaces on which $\pi$ acts by isometries.
\end{theorem}

\begin{proof}
We fix an integrable vector $g \in \mathcal{A} \setminus \{0\}$. Let us first show that $\|\cdot\|_{\mathcal{C}o_{p}(G)}$ is in fact a norm. The only non-trivial point is the positive-definiteness. Assume that $\|\mathcal{W}_{g}f\|_{L^{p}(G)} = 0$ for some $f \in \mathcal{C}o_{p}(G)$. Then $\mathcal{W}_{g}f$ is zero almost everywhere as a function on $G$. Since $\mathcal{W}_{g}f$ is a continuous function on $G$ by Proposition \ref{proposition_basic_props_of_extended_wavelet_transform}, we have that $\mathcal{W}_{g}f$ is identically zero. Since $\mathcal{W}_{g}:\mathcal{R} \to L^{\infty}(G)$ is injective, we conclude that $f = 0$. \par
To show completeness, we assume that $\{f_n\}_{n \in \mathbb{N}}$ is a Cauchy sequence in $\mathcal{C}o_{p}(G)$. Then $\{\mathcal{W}_{g}f_{n}\}_{n \in \mathbb{N}}$ is a Cauchy sequence in $L^{p}(G)$. By completeness of $L^{p}(G)$, there exists $F \in L^{p}(G)$ such that $\mathcal{W}_{g}f_n \to F$ in $L^{p}(G)$. It follows that \[F * \mathcal{W}_{g}g = \left(\lim_{n \to \infty}\mathcal{W}_{g}f_{n}\right) *_{G} \mathcal{W}_{g}g = \lim_{n \to \infty}\left(\mathcal{W}_{g}f_{n} *_{G} \mathcal{W}_{g}g \right) =  \lim_{n \to \infty}\mathcal{W}_{g}f_{n} = F.\]
We can now use the correspondence principle in Theorem \ref{corresponance_principle} to conclude that $F = \mathcal{W}_{g}f$ for some $f \in \mathcal{C}o_{p}(G)$.
Hence the coorbit spaces $\mathcal{C}o_{p}(G)$ are complete since \[\|f_n - f\|_{\mathcal{C}o_{p}(G)} = \|\mathcal{W}_{g}f_n - \mathcal{W}_{g}f\|_{L^{p}(G)} \to 0.\]
Finally, if $f \in \mathcal{C}o_{p}(G)$ and $x \in G$ then we use \eqref{extended_intertwining_property} to obtain \[\|\pi(x)f\|_{\mathcal{C}o_{p}(G)} = \|\mathcal{W}_{g}(\pi(x)f)\|_{L^{p}(G)} = \|L_{x}\mathcal{W}_{g}f\|_{L^{p}(G)} = \|\mathcal{W}_{g}f\|_{L^{p}(G)} = \|f\|_{\mathcal{C}o_{p}(G)}. \qedhere\]
\end{proof}

The following proposition shows that the spaces $\mathcal{H}^{1}$, $\mathcal{H}_{\pi}$, and $\mathcal{R}$ all have descriptions in terms of coorbit spaces. 

\begin{proposition}
\label{concrete_descriptions}
Let $\pi:G \to \mathcal{U}(\mathcal{H}_{\pi})$ be an integrable representation. We have the descriptions \[\mathcal{C}o_{1}(G) = \mathcal{H}^{1}, \quad \mathcal{C}o_{2}(G) = \mathcal{H}_{\pi}, \quad \mathcal{C}o_{\infty}(G) = \mathcal{R}.\]
\end{proposition}

\begin{proof}
As usual, we fix an integrable vector $g \in \mathcal{A} \setminus \{0\}$. The statement $\mathcal{C}o_{\infty}(G) = \mathcal{R}$ is clear from the definition of $\mathcal{C}o_{\infty}(G)$ since every $\phi \in \mathcal{R}$ satisfies $\mathcal{W}_{g}\phi \in L^{\infty}(G)$ by Proposition \ref{proposition_basic_props_of_extended_wavelet_transform}. We have that $\mathcal{H}^{1} \subset \mathcal{C}o_{1}(G)$ and $\mathcal{H}_{\pi} \subset \mathcal{C}o_{2}(G)$ through the inclusions in Lemma \ref{lemma_inclusions}. Conversely, assume that $f \in \mathcal{C}o_{2}(G)$. Then $\mathcal{W}_{g}f \in L^{2}(G)$ and satisfies by the correspondence principle in Theorem \ref{corresponance_principle} the convolution relation \[\mathcal{W}_{g}f = \mathcal{W}_{g}f *_{G} \mathcal{W}_{g}g.\]
However, in Theorem \ref{reproducing_formula} we showed that $F \mapsto F *_{G} \mathcal{W}_{g}g$ is the projection from $L^{2}(G)$ to the space $\mathcal{W}_{g}(\mathcal{H}_{\pi})$. Hence we conclude that $\mathcal{W}_{g}f = \mathcal{W}_{g}h$ for some $h \in \mathcal{H}_{\pi}$. Since $\mathcal{W}_{g}:\mathcal{R} \to L^{\infty}(G)$ is injective we have that $f = h$ as elements in $\mathcal{R}$. Moreover,  the injectivity of the inclusion $\mathcal{H}_{\pi} \xhookrightarrow{} \mathcal{R}$ forces $f \in \mathcal{H}_{\pi}$, and thus the claim $\mathcal{C}o_{2}(G) = \mathcal{H}_{\pi}$ follows. Since $L^{1}(G) \cap L^{\infty}(G) \subset L^{2}(G)$, we can repeat the same argument for $f \in \mathcal{C}o_{1}(G)$ and find that $f \in \mathcal{H}_{\pi}$. As $\mathcal{H}^{1}$ is by definition the set of elements $f \in \mathcal{H}_{\pi}$ such that $\mathcal{W}_{g}f \in L^{1}(G)$, we have that $\mathcal{C}o_{1}(G) = \mathcal{H}^{1}$.
\end{proof}

\begin{remark}
The proof of Proposition \ref{concrete_descriptions} shows that $\mathcal{C}o_{p}(G) \subset \mathcal{H}_{\pi}$ for all $p \in [1,2]$ since then $L^{p}(G) \cap L^{\infty}(G) \subset L^{2}(G)$.
\end{remark}

The following result shows that the coorbit spaces $\mathcal{C}o_{p}(G)$ inherit their duality properties from the $L^{p}(G)$-spaces. For a proof of this result, we refer the reader to \cite[Theorem 4.9]{feichtinger1989banach1}.

\begin{proposition}
\label{duality_statement}
Let $\pi:G \to \mathcal{U}(\mathcal{H}_{\pi})$ be an integrable representation. The coorbit spaces $\mathcal{C}o_{p}(G)$ for $1 \leq p < \infty$ satisfy the duality \[\mathcal{C}o_{p}(G)' = \mathcal{C}o_{q}(G), \qquad \frac{1}{p} + \frac{1}{q} = 1.\] In particular, the coorbit spaces $\mathcal{C}o_{p}(G)$ are reflexive Banach spaces for $1 < p < \infty$.
\end{proposition}

\begin{example}
\label{affine_coorbit_spaces_example}
Let us again consider the affine group $\textrm{Aff}$ together with the wavelet representation $\pi:\textrm{Aff} \to \mathcal{U}(L^{2}(\mathbb{R}))$ given by 
\begin{equation*}
    \pi(b,a)f(x) := T_{b}D_{a}f(x) = \frac{1}{\sqrt{|a|}}f\left(\frac{x - b}{a}\right).
\end{equation*} 
We showed in Example \ref{wavelet_example} that $\pi$ is a square integrable representation. A straightforward computation shows that $\pi$ is in fact integrable by considering a non-zero function $g \in \mathcal{S}(\mathbb{R})$ such that $\mathcal{F}(g)$ is supported on $[r,s]$ for $r,s \in (0, \infty)$. Hence for any integrable vector $g \in \mathcal{A} \setminus \{0\}$ we obtain for each $1 \leq p < \infty$ the \textit{affine coorbit space} $\mathcal{C}o_{p}(\textrm{Aff})$ defined by \[\mathcal{C}o_{p}(\textrm{Aff}) := \left\{f \in \mathcal{R} \, : \, \|f\|_{\mathcal{C}o_{p}(\textrm{Aff})} := \left(\int_{\textrm{Aff}}|\mathcal{W}_{g}f(b,a)|^{p} \, \frac{db \, da}{a^2}\right)^{\frac{1}{p}} < \infty \right\}.\]
As usual, the case $p = \infty$ is defined with the supremum. We immediately get from Theorem \ref{completeness_of_coorbit_spaces} that $\mathcal{C}o_{p}(\textrm{Aff})$ is a Banach space for each $1 \leq p \leq \infty$ on which the wavelet representation $\pi$ acts by isometries.
\end{example}

\subsection{Extension to the Weighted Setting}
\label{sec: Weighted_Versions}

In this section, we will discuss how coorbit spaces can be generalized to include weights. This is usually done right from the beginning in the literature, see e.g.\ \cite{feichtinger1988unified, feichtinger1989banach1, feichtinger1989banach2, dahlke2015harmonic, felix_thesis}. However, we have opted to introduce this separately so that coorbit spaces could initially be introduced with minimal technicalities. As weights do not introduce anything conceptually new, this section mostly consists of technicalities that invoke feelings of déjà vu. \par

\begin{definition}
Let $G$ be a locally compact group. Given any continuous function $w:G \to (0, \infty)$ we can form the \textit{weighted $L^{p}$-space} $L_{w}^{p}(G)$ for $1 \leq p \leq \infty$ consisting of all equivalence classes of measurable function $f:G \to \mathbb{C}$ such that \[\|f\|_{L_{w}^{p}(G)} := \|f \cdot w\|_{L^{p}(G)} < \infty.\]
We say that a continuous function $w:G \to (0, \infty)$ is a \textit{weight function} if it is \textit{sub-multiplicative}, that is, $w$ satisfies $w(xy) \leq w(x)w(y)$ for all $x,y \in G$.
\end{definition}

\begin{remark}
The reader should be aware that the conditions that goes into the term \textit{weight function} (or simply \textit{weight}) differs quite a bit from author to author: In \cite{felix_thesis} a sub-multiplicative weight is not assumed to be continuous, only measurable. It turns out that a not necessarily continuous sub-multiplicative weight is automatically bounded on compact sets by \cite[Theorem 2.2.22]{felix_thesis}. A weight $w$ in \cite[Chapter 3]{dahlke2015harmonic} is assumed to be \textit{symmetric}, meaning that $w(x) = w(x^{-1})$ for all $x \in G$. The symmetry assumption automatically gives that $w \geq 1$. If $w$ is a not necessarily symmetric weight function on $G$ such that $w \geq 1$, then $L_{w}^{p}(G) \xhookrightarrow{} L^{p}(G)$ is a continuous embedding since \[\|f\|_{L^{p}(G)} = \left(\int_{G}|f(x)|^{p} \, d\mu_{L}(x)\right)^{\frac{1}{p}} \leq\left(\int_{G}|f(x)|^{p}w(x)^{p} \, d\mu_{L}(x)\right)^{\frac{1}{p}} = \|f\|_{L_{w}^{p}(G)},\]
for all $f \in L_{w}^{p}(G)$.
\end{remark}

\begin{example}
Consider the function $w$ on $G = (0, \infty)$ given by \[w(x) := e^{|\log(x)|} = \begin{cases} x \, & \textrm{ if } x \geq 1 \\ \frac{1}{x}, \, & \textrm{ if } x < 1\end{cases}.\] It is straightforward to verify that $w$ is a symmetric weight function. The condition for a measurable function $f:G \to \mathbb{C}$ to be in $L_{w}^{1}(G)$ takes the form \[\int_{0}^{1}\frac{|f(x)|}{x^2} \, dx + \int_{1}^{\infty}|f(x)| \, dx < \infty.\]
\end{example}

\begin{definition}
Let $\pi:G \to \mathcal{U}(\mathcal{H}_{\pi})$ be an irreducible unitary representation of the locally compact group $G$ and fix a weight function $w:G \to (0, \infty)$. The representation $\pi$ is called \textit{$w$-integrable} if there exists a non-zero element $g \in \mathcal{H}_{\pi}$ such that $\mathcal{W}_{g}g \in L_{w}^{1}(G)$. We use the notation \[\mathcal{A}_{w} := \left\{g \in \mathcal{H}_{\pi} \, : \, \mathcal{W}_{g}g \in L_{w}^{1}(G) \right\}.\]
Similarly as before, we fix $g \in \mathcal{A}_{w} \setminus \{0\}$ and define the space of \textit{$w$-test vectors} $\mathcal{H}_{w,g}^{1}$ as the elements $f \in \mathcal{H}_{\pi}$ such that $\mathcal{W}_{g}f \in L_{w}^{1}(G)$.
\end{definition}

The proof of the following result illustrates the usefulness of the sub-multiplicative condition. 

\begin{lemma}
\label{act_continuously_lemma}
Let $\pi:G \to \mathcal{U}(\mathcal{H}_{\pi})$ be a $w$-integrable representation and fix $g \in \mathcal{A}_{w} \setminus \{0\}$. Then $\pi$ acts continuously and invariantly on $\mathcal{H}_{w,g}^{1}$. 
\end{lemma}

\begin{proof}
We fix $f \in \mathcal{H}_{w,g}^{1}$ and compute for $x\in G$ that 
\begin{align*}
    \|\pi(x)f\|_{\mathcal{H}_{w,g}^{1}} & = \int_{G} |\mathcal{W}_{g}(\pi(x)f)(y)|w(y) \, d\mu_{L}(y) \\ & = \int_{G} |\mathcal{W}_{g}(f)(x^{-1}y)|w(y) \, d\mu_{L}(y) \\ & = \int_{G} |\mathcal{W}_{g}(f)(y)|w(xy) \, d\mu_{L}(y).
\end{align*}
By using the sub-multiplicative condition we end up with \[\|\pi(x)f\|_{\mathcal{H}_{w,g}^{1}} \leq w(x)\int_{G} |\mathcal{W}_{g}(f)(y)|w(y) \, d\mu_{L}(y) = w(x)\cdot\|f\|_{\mathcal{H}_{w,g}^{1}}. \qedhere\]
\end{proof}

We can now use Lemma \ref{act_continuously_lemma} to see that the space $\mathcal{H}_{w,g}^{1}$ is dense in $\mathcal{H}_{\pi}$ for all $g \in \mathcal{A}_{w} \setminus \{0\}$. It is straightforward to check that the space $L_{w}^{1}(G)$ is invariant under both the left-translation operator and the right-translation operator. This fact is sufficient for Lemma \ref{lemma_about_existence_of_non_orthogonal_element} to go through in the weighted setting. Finally, only minor changes are needed in Proposition \ref{prop:completeness_H1}, Theorem \ref{Theorem_about_independence_of_window}, and Proposition \ref{equality_of_spaces} to obtain the weighted statements. Hence $\mathcal{H}_{w,g}^{1}$ does not depend on the choice of $g \in \mathcal{A}_{w} \setminus \{0\}$ and we simply write \[\mathcal{H}_{w}^{1} := \mathcal{H}_{w,g}^{1}.\] 

\begin{example}
\label{weighted_Feichtinger_space}
A class of commonly used symmetric weight functions on $\mathbb{R}^{2n}$ is given by \[v_{s}(x,\omega) := (1 + |x|^2 + |\omega|^2)^{\frac{s}{2}}, \quad (x,\omega) \in \mathbb{R}^{2n}, \, s \geq 0.\] The family $v_s$ is sometimes referred to as the \textit{polynomial weights}. For the STFT we can use the polynomial weights to define the \textit{weighted Feichtinger algebra} $M_{s}^{1}(\mathbb{R}^{n}) := \mathcal{H}_{v_s}^{1}$. The inequality \[v_{s} \leq v_{t}, \qquad 0 \leq s \leq t\] implies the inclusion $M_{t}^{1}(\mathbb{R}_{n}) \subset M_{s}^{1}(\mathbb{R}^{n})$. In particular, we have $M_{s}^{1}(\mathbb{R}^{n}) \subset M^{1}(\mathbb{R}^{n})$ for all $s \geq 0.$ It is straightforward to check that $M_{s}^{1}(\mathbb{R}^{n})$ still contains the rapidly decaying and smooth functions $\mathcal{S}(\mathbb{R}^{n})$ for all $s \geq 0$. Is there anything more than $\mathcal{S}(\mathbb{R}^{n})$ contained in all of the weighted Feichtinger algebras $M_{s}^{1}(\mathbb{R}^{n})$ for $s \geq 0$? By \cite[Proposition 11.3.1]{grochenig2001foundations} the answer is negative and we can write \[\mathcal{S}(\mathbb{R}^{n}) = \bigcap_{s \geq 0} M_{s}^{1}(\mathbb{R}^{n}).\]
\end{example}

\begin{definition}
\label{definition_weighted_reservoir}
Let $\pi:G \to \mathcal{U}(\mathcal{H}_{\pi})$ be a $w$-integrable representation. We define the \textit{$w$-reservoir space} $\mathcal{R}_{1/w}$ as the space of bounded anti-linear functionals on $\mathcal{H}_{w}^{1}$. 
\end{definition}

The duality between $\mathcal{H}_{w}^{1}$ and $\mathcal{R}_{1/w}$ is again denoted by $\phi(g) = \langle \phi, g \rangle$ for $g \in \mathcal{H}_{w}^{1}$ and $\phi \in \mathcal{R}_{1/w}$. Lemma \ref{lemma_inclusions} goes through directly with the new notational changes and we have the inclusions 
\[\mathcal{H}_{w}^{1} \xhookrightarrow{} \mathcal{H}_{\pi} \xhookrightarrow{} \mathcal{R}_{1/w}.\]
The action of $\pi$ on $\mathcal{R}_{1/w}$ is defined in the same way as in Section \ref{sec: Reservoirs_and_the_Extended_Wavelet_Transform}. We can again define the \textit{(extended) wavelet transform} by the formula \[\mathcal{W}_{g}\phi(x) :=  \langle \phi, \pi(x)g \rangle, \qquad g \in \mathcal{H}_{w}^{1}, \, \phi \in \mathcal{R}_{1/w}.\]

\begin{remark}
The reader should be aware that $1/w$ is not in general a weight function, even when $w:G \to (0,\infty)$ is a symmetric weight function. However, the failure of $1/w$ to be sub-multiplicative can be remedied: If $w$ is symmetric, then we can write for $x,y \in G$ that
\[w(x) = w(xyy^{-1}) \leq w(xy)w(y^{-1}) = w(xy)w(y).\] Hence \[\frac{1}{w(xy)} \leq \frac{1}{w(x)}w(y).\]
This relation suffices in most settings.  
\end{remark}

The proof of Lemma \ref{lemma_convergence_in_H1_and_convolution_relation} in \cite[Lemma 2.4.7 and Lemma 2.4.8]{felix_thesis} is stated in the weighted case. Finally, Corollary \ref{injection_on_L_infinity_corollary}, Proposition \ref{results_about_adjoint}, and Theorem \ref{resovoir_vs_L_infinity} are almost verbatim the same as previously. The only thing worth remarking is that the space of bounded anti-linear functionals on $L_{w}^{1}(G)$ is $L_{1/w}^{\infty}(G)$. This motivates the notation $\mathcal{R}_{1/w}$.

\begin{example}
For the STFT we use the notation $M_{1/s}^{\infty}(\mathbb{R}^{n}) := \mathcal{R}_{1/v_s}$, where $v_{s}$ for $s \geq 0$ are the polynomial weights introduced in Example \ref{weighted_Feichtinger_space}. The discussion in Example \ref{weighted_Feichtinger_space} regarding $\mathcal{S}(\mathbb{R}^{n})$ has the dual version  \[\mathcal{S}'(\mathbb{R}^{n}) = \bigcup_{s \geq 0} M_{1/s}^{\infty}(\mathbb{R}^{n}).\] Hence the pair $\left(\mathcal{S}(\mathbb{R}^{n}),\mathcal{S}'(\mathbb{R}^{n}) \right)$ works as limiting cases for respectively the weighted Feichtinger algebras $M_{s}^{1}(\mathbb{R}^{n})$ and the weighted reservoir spaces $M_{1/s}^{\infty}(\mathbb{R}^{n})$ for $s \geq 0$. 
\end{example}

\begin{definition}
Let $\pi:G \to \mathcal{U}(\mathcal{H}_{\pi})$ be a $w$-integrable representation and fix a $w$-integrable vector $g \in \mathcal{A}_{w} \setminus \{0\}$. The \textit{(weighted) coorbit space} $\mathcal{C}o_{p,w}(G)$ for $1 \leq p \leq \infty$ is given by the straightforward extension 
\begin{equation*}
    \mathcal{C}o_{p,w}(G) := \{\phi \in \mathcal{R}_{1/w} \, : \, \mathcal{W}_{g}\phi \in L_{w}^{p}(G)\},
\end{equation*}
with the norm \[\|\phi\|_{\mathcal{C}o_{p,w}(G)} := \|\mathcal{W}_{g}\phi\|_{L_{w}^{p}(G)}.\] 
\end{definition}

As previously, the coorbit spaces $\mathcal{C}o_{p,w}(G)$ do not depend on the choice of $w$-integrable vector $g \in \mathcal{A}_{w} \setminus \{0\}$, see \cite[Section 5.2]{feichtinger1988unified} for details. It is clear that \[\mathcal{C}o_{\infty,1/w}(G) = \mathcal{R}_{1/w}.\]

\begin{example}
We define the \textit{weighted modulation spaces} $M_{s}^{p,q}(\mathbb{R}^{n})$ for $1 \leq p,q \leq \infty$ and $s \geq 0$ to be all elements $f \in M_{1/s}^{\infty}(\mathbb{R}^{n})$ such that \[\|f\|_{M_{s}^{p,q}(\mathbb{R}^{n})} := \left(\int_{\mathbb{R}^n}\left(\int_{\mathbb{R}^{n}}|V_{g}f(x,\omega)|^{p} (1 + |x|^2 + |\omega|^2)^{\frac{ps}{2}} \, dx \right)^{\frac{q}{p}} \, d\omega \right)^{\frac{1}{q}} < \infty,\]
where $g \in \mathcal{A}_{v_s} \setminus \{0\}$ is fixed. Since the reduced Heisenberg group $\mathbb{H}_{r}^{n}$ is unimodular, it follows from the weighted version of Proposition \ref{equality_of_spaces} that \[\mathcal{A}_{v_s} = \mathcal{H}_{v_s}^{1} = M_{s}^{1}(\mathbb{R}^{n}).\] It is common in practice to choose $g \in \mathcal{S}(\mathbb{R}^{n}) \setminus \{0\}$, which is valid since $\mathcal{S}(\mathbb{R}^{n}) \subset M_{s}^{1}(\mathbb{R}^{n})$ for all $s \geq 0$. Moreover, one can choose the reservoir to be $\mathcal{S}'(\mathbb{R}^{n})$ instead of $M_{1/s}^{\infty}(\mathbb{R}^{n})$ without changing the weighted modulation spaces. It is possible to use different weights to obtain other weighted modulation spaces. We refer the reader to \cite[Section 11.4]{grochenig2001foundations} to see how one can define weighted modulation spaces where the weights have exponential growth.
\end{example}

The completeness of $L_{w}^{p}(G)$ for a weight function $w:G \to (0, \infty)$ allows us to extend the first statement in Theorem \ref{completeness_of_coorbit_spaces} to the weighted setting. The second statement in Theorem \ref{completeness_of_coorbit_spaces} has to be altered to say that $\pi$ acts continuously on the weighted coorbit spaces $\mathcal{C}o_{p,w}(G)$; this uses the same argument we gave in the proof of Lemma \ref{act_continuously_lemma}. 
The statement in Proposition \ref{concrete_descriptions} is valid in the weighted setting with the nessesary changes. More precisely, for $g \in \mathcal{A}_{w} \setminus \{0\}$ we let $\mathcal{H}_{\pi,w}$ denote the elements $f \in \mathcal{H}_{\pi}$ such that $\mathcal{W}_{g}f \in L_{w}^{2}(G)$. Then we can adapt the proof of Proposition~\ref{concrete_descriptions} to see that \[\mathcal{C}o_{1,w}(G) = \mathcal{H}_{w}^{1}, \qquad \mathcal{C}o_{2,w}(G) = \mathcal{H}_{\pi,w}, \qquad \mathcal{C}o_{\infty,1/w}(G) = \mathcal{R}_{1/w}.\] Finally, the duality statement in Proposition \ref{duality_statement} is still valid in the weighted setting with the nessesary changes, see \cite[Theorem 4.9]{feichtinger1989banach1} for details.
Before moving on, we summarize the most important results regarding the weighted coorbit spaces in one theorem so that we have precise statements we can reference later in the survey.

\begin{theorem}
\label{summary_theorem}
Let $\pi:G \to \mathcal{U}(\mathcal{H}_{\pi})$ be a $w$-integrable representation where $w:G \to (0, \infty)$ is a weight function. Fix a $w$-integrable vector $g \in \mathcal{A}_{w} \setminus \{0\}$. Then the coorbit spaces $\mathcal{C}o_{p,w}(G)$ for $1 \leq p \leq \infty$ satisfy the following properties:
\begin{enumerate}[label=\alph*)]
    \item The coorbit spaces $\mathcal{C}o_{p,w}(G)$ are Banach spaces on which the representation $\pi$ acts invariantly and continuously.
    \item An element $F \in L_{w}^{p}(G)$ satisfies the convolution relation $F = F *_{G} \mathcal{W}_{g}g$ if and only if $F = \mathcal{W}_{g}f$ for some $f \in \mathcal{C}o_{p,w}(G)$.
    \item We have the identifications \[\mathcal{C}o_{1,w}(G) = \mathcal{H}_{w}^{1}, \qquad \mathcal{C}o_{2,w}(G) = \mathcal{H}_{\pi,w}, \qquad \mathcal{C}o_{\infty,1/w}(G) = \mathcal{R}_{1/w}.\]
    \item The coorbit spaces $\mathcal{C}o_{p,w}(G)$ for $1 \leq p < \infty $ satisfy the duality relation \[\mathcal{C}o_{p,w}(G)' = \mathcal{C}o_{q,1/w}(G), \qquad \frac{1}{p} + \frac{1}{q} = 1,\]
    where \[\mathcal{C}o_{p,1/w}(G) := \{\phi \in \mathcal{R}_{1/w} \, : \, \mathcal{W}_{g}\phi \in L_{1/w}^{p}(G)\}.\]
\end{enumerate}
\end{theorem}

\begin{example}
Consider the function $w_{s}: \textrm{Aff} \to (0, \infty)$ for $s \geq 0$ on the affine group $\textrm{Aff}$ given by $w_{s}(b,a) := |a|^{-s}$. The computation \[w_{s}(b,a)w_{s}(d,c) = |a|^{-s}|c|^{-s} = |ac|^{-s} = w_{s}((b,a)\cdot (d,c)), \qquad (b,a) \, (d,c) \in \textrm{Aff},\] shows that $w_{s}$ is multiplicative, and hence clearly a weight function. The argument in Example \ref{affine_coorbit_spaces_example} can be extended to show that the wavelet representation $\pi:\textrm{Aff} \to \mathcal{U}(L^{2}(\mathbb{R}))$ is $w_{s}$-integrable for any $s \geq 0$. Thus we can consider the \textit{weighted affine coorbit spaces} $\mathcal{C}o_{p,w_{s}}(\textrm{Aff})$. It turns out that \[\mathcal{C}o_{p,w_{s}}(\textrm{Aff}) = \dot{\mathcal{B}}_{p}^{s - \frac{1}{2} + \frac{1}{p}}(\mathbb{R}),\]
where $\dot{\mathcal{B}}_{p}^{s}(\mathbb{R})$ denotes the \textit{homogeneous Besov space}
in classical harmonic analysis with smoothness parameter $s \in \mathbb{R}$ and integrability parameter $1 \leq p \leq \infty$. We refer the reader to \cite{feichtinger1988unified} for details of this fascinating connection. 
\end{example}

\subsection{Atomic Decompositions}
\label{sec: Atomic Decompositions}

We have so far introduced the coorbit spaces and derived their basic properties. The message that should be drawn from Theorem \ref{summary_theorem} is that coorbit spaces form a well-behaved class of Banach spaces. Nevertheless, the reader might find herself wondering what the fuzz is all about. Constructing function spaces is commonplace in modern mathematics, so it is maybe unclear why coorbit spaces offer something special. The goal of this section is to convince the reader that the coorbit spaces are deeply connected with the geometry of the underlying locally compact group. Moreover, this connection is inherently practical as it furnishes us with a natural way to discretize elements in coorbit spaces as we mentioned in \eqref{atomic_decomposition_introduction}. This makes coorbit spaces novel because they form a bridge between geometry, representation theory, and approximation theory. \par
Let us start by precisely stating the continuous reconstruction formula for coorbit spaces. Fix a weight function $w:G \to (0, \infty)$ and a $w$-integrable representation $\pi: G \to \mathcal{U}(\mathcal{H}_{\pi})$. Then for $f \in \mathcal{C}o_{p,w}(G)$ we can use the weighted version of Proposition \ref{results_about_adjoint} to write 
\begin{equation}
\label{coorbit_reconstruction_formula}
    f = \mathcal{W}_{g}^{*}\left(\mathcal{W}_{g}f\right) = \int_{G}\mathcal{W}_{g}f(x) \pi(x)g \, d\mu_{L}(x),
\end{equation}
for $g \in \mathcal{A}_{w} \setminus \{0\}$. We refer to \eqref{coorbit_reconstruction_formula} as the \textit{continuous reconstruction formula} for $\mathcal{C}o_{p,w}(G)$. \par
What does a discretization of \eqref{coorbit_reconstruction_formula} look like? Replacing the integral with summation, we hope to express $f \in \mathcal{C}o_{p,w}(G)$ as the discrete superposition 
\begin{equation}
\label{discrete_series_expansion}
    f = \sum_{i \in I}c_{i}(f)\pi(x_i)g,
\end{equation}
where $(c_{i}(f))_{i \in I}$ are coefficients that depend on $f$ and $\{x_{i}\}_{i \in I} \subset G$ is a chosen countable collection of points. We note that \eqref{discrete_series_expansion} should be interpreted as convergence in the norm on $\mathcal{C}o_{p,w}(G)$ for $1 \leq p < \infty$. When $p = \infty$ we interpret \eqref{discrete_series_expansion} as convergence in the $\mathrm{weak}^*$-topology. In the literature, expansions on the form \eqref{discrete_series_expansion} are sometimes called \textit{atomic decompositions} as the element $g$ is considered an \textit{atom} from which all other relevant functions are constructed. Three natural questions emerge:
\begin{itemize}
    \item How can we chose the collection $\{x_{i}\}_{i \in I} \subset G$ such that \eqref{discrete_series_expansion} converges appropriately?
    \item How does the size of $f \in \mathcal{C}o_{p,w}(G)$ affect the size of $(c_{i}(f))_{i \in I}$ in a suitable norm?
    \item Is it possible to choose the coefficients $(c_{i}(f))_{i \in I}$ to depend linearly on $f$?
\end{itemize}
Before we answer the questions above in Theorem \ref{big_discretization_theorem} we will borrow some terminology from large scale geometry. This will provide a conceptual language for discussing discretizations.
\begin{definition}
Let $X$ be a non-empty set. We will refer to a collection of non-empty subsets $\mathcal{Q} = (Q_{i})_{i \in I}$ as an \textit{admissible covering} for $X$ if $X = \cup_{i \in I}Q_{i}$ and 
\begin{equation}
\label{neighbours}
    \sup_{i \in I}\left|\left\{j \in I \, \Big| \, Q_{i} \cap Q_{j} \ne \emptyset \right\}\right| < \infty.
\end{equation} 
\end{definition}
Intuitively, the condition \eqref{neighbours} states that each $Q_{i} \in \mathcal{Q}$ can not have to many neighbors. Given an admissible covering $\mathcal{Q} = (Q_{i})_{i \in I}$ for a non-empty set $X$, we call a sequence $Q_{i_1}, \dots, Q_{i_k} \in \mathcal{Q}$ with $x \in Q_{i_1}$ and $y \in Q_{i_k}$ a $\mathcal{Q}$-\textit{chain} from $x$ to $y$ of \textit{length} $k$ whenever $Q_{i_l} \cap Q_{i_{l+1}} \ne \emptyset$ for every $1 \leq l \leq k-1$. The notation $\mathcal{Q}(k,x,y)$ will be used to denote all $\mathcal{Q}$-chains of length $k$ from $x$ to $y$. An admissible covering $\mathcal{Q}$ on a set $X$ will be called a \textit{concatenation} if for every pair of points $x,y \in X$ there exists a positive number $k \in \mathbb{N}$ such that $\mathcal{Q}(k,x,y) \ne \emptyset$. The idea, originating from \cite{Hans_Grobner}, is to consider a metric $d_{\mathcal{Q}}$ that incorporates closeness relative to the covering $\mathcal{Q}$. This idea has more recently been further investigated in \cite{eiriklargescale, renethesis}. Formally, we have the following definition.

\begin{definition}
Consider a concatenation $\mathcal{Q} = (Q_{i})_{i \in I}$ for a non-empty set $X$.
Define the metric $d_{\mathcal{Q}}$ on $X$ by the rule $d_{\mathcal{Q}}(x,x) = 0$ for all $x \in X$ and \[d_{\mathcal{Q}}(x,y) = \inf \left\{k:\mathcal{Q}(k,x,y) \ne \emptyset \right\}, \qquad x,y \in X, \, \, x \ne y.\]
\end{definition}

It is straightforward to check that $d_Q$ is indeed a metric on $X$. Notice that $d_{Q}(x,y) < \infty$ for all $x,y \in X$ precisely because we assume that $\mathcal{Q}$ is a concatenation. We will refer to $(X,d_{\mathcal{Q}})$ as the \textit{associated metric space} to the concatenation $\mathcal{Q}$. A subset $N \subset X$ is called a \textit{net} if there exists a fixed constant $C > 0$ such that for every $x \in X$ there is $y \in N$ such that $d_{\mathcal{Q}}(x,y) < C$. 

\begin{definition}
Let $(X,d_{X})$ and $(Y, d_{Y})$ be two metric spaces. We say that a map $f:X \to Y$ is a \textit{quasi-isometry} if $f(X)$ is a net in $(Y,d_{Y})$ and there exist fixed constants $C, L > 0$ such that \[\frac{1}{L}d_{X}(x,y) - C \leq d_{Y}(f(x), f(y)) \leq L d_{X}(x,y) + C,\]
for every $x,y \in X$.
\end{definition}

\begin{remark}
Notice that a quasi-isometry $f:X \to Y$ is a generalization of an isometry where the map $f$ does not need to be injective nor surjective. This is a suitable notion for comparing metric spaces of different cardinalities. As an example, the inclusion $i:\mathbb{Z} \hookrightarrow \mathbb{R}$ is a quasi-isometry when considering the standard metrics. 
\end{remark}

Let us now focus on the setting we are interested in. Given a locally compact group $G$ we fix a compact set $Q$ with non-empty interior that contains the identity element $e \in G$. Then the collection \[\mathcal{Q}_{\textrm{cont}} := (x \cdot Q)_{x \in G}\] is a cover for $G$ that is typically not admissible. However, it is always possible to find a subfamily $N = \{x_{i}\}_{i \in I} \subset G$ such that $\mathcal{Q} := (x_i \cdot Q)_{i \in I}$ is admissible by \cite[Theorem 4.1 (A)]{Hans_2}. This way of obtaining $N$ is non-constructive and one usually relies on an understanding of the geometry of $G$ in practical situations to construct $N$. We refer to $\mathcal{Q}$ as the \textit{uniform covering} corresponding to $G$ with reference set $Q$. When $G$ is path-connected the covering $\mathcal{Q}$ is actually a concatenation, see \cite[Lemma 3.1]{eiriklargescale}. Hence we obtain an associated metric $d_{\mathcal{Q}}$ on $G$. Maybe surprisingly, the resulting metric space $(G,d_{\mathcal{Q}})$ does not depend (up to quasi-isometry) on the choice of $N$ by \cite[Theorem 4.1 (B)]{Hans_2}. In light of this, we refer to the metric $d_{\mathcal{Q}}$ as the \textit{uniform metric} and the space $(G,d_{\mathcal{Q}})$ as the \textit{uniform metric space} corresponding to a path-connected locally compact group $G$. Although the metric $d_{\mathcal{Q}}$ is left-invariant, it is almost never compatible with the underlying topology of $G$. 

\begin{example}
Consider the group $G = \mathbb{R}$ with the reference set $Q = [-1,1]$. Then \[\mathcal{Q}_{\textrm{cont}} = (x + Q)_{x \in \mathbb{R}} = ([x - 1, x + 1])_{x \in \mathbb{R}}.\] The subfamily $N = \mathbb{Z}$ makes \[\mathcal{Q} := (n + Q)_{n \in \mathbb{Z}} = ([n - 1, n + 1])_{n \in \mathbb{Z}}\] into a concatenation. Due to the left invariance of the metric $d_{\mathcal{Q}}$, it is completely determined by \[d_{\mathcal{Q}}(0, x) = \lceil x \rceil, \qquad x > 0,\]
where $\lceil x \rceil$ denotes the ceiling function of $x \in \mathbb{R}$.
\end{example}

\begin{remark}
The points $\{x_{i}\}_{i \in I}$ such that $\mathcal{Q} = (x_i \cdot Q)_{i \in I}$ is the uniform covering of $G$ are only \textit{candidates} for points where the atomic discretization \eqref{discrete_series_expansion} is valid. As an extreme example, consider when $G$ is compact and we pick the reference set $Q = G$. Then $\mathcal{Q} = \{e \cdot Q\} = \{Q\}$ is the uniform covering. However, one does not generally have a discretization \[f = c(f) \cdot \pi(e)g = c(f) \cdot g,\] for all $f \in \mathcal{C}o_{p}(G)$ since $\mathcal{C}o_{p}(G)$ is not necessarily one-dimensional. The problem here is that the reference set $Q$ is to large.
\end{remark}

The following theorem is the main result regarding atomic decompositions. 

\begin{theorem}[Atomic Decomposition Theorem]
\label{big_discretization_theorem}
Let $\pi: G \to \mathcal{U}(\mathcal{H}_{\pi})$ be a $w$-integrable representation, where $w:G \to (0,\infty)$ is a weight function. For well-behaved $g \in \mathcal{A}_{w} \setminus \{0\}$ we have for any sufficiently small reference set $Q$ and any $1 \leq p \leq \infty$ the following properties:
\begin{itemize}
    \item For $f \in \mathcal{C}o_{p,w}(G)$ we have the discrete reconstruction formula 
    \begin{equation*}
        f = \sum_{i \in I}c_{i}(f)\pi(x_i)g,
    \end{equation*}
    where $\mathcal{Q} = (x_i \cdot Q)_{i \in I}$ is the uniform covering corresponding to the reference set $Q$. The sequence $(c_{i}(f))_{i \in I}$ depends linearly on $f$. Moreover, there exists a constant $C_{A} > 0$ not depending on $f$ such that
    \[\|(c_{i}(f))_{i \in I}\|_{l_{w}^{p}(I)} \leq C_{A} \|f\|_{\mathcal{C}o_{p,w}(G)}.\]
    \item Given a sequence $(c_i)_{i \in I} \in l_{w}^{p}(I)$ we can construct \[f = \sum_{i \in I}c_{i}\pi(x_i)g\] such that \[\|f\|_{\mathcal{C}o_{p,w}(G)} \leq C_{R} \|(c_i)_{i \in I}\|_{l_{w}^{p}(I)},\]
    where $C_{R} > 0$ is a constant not depending on $(c_i)_{i \in I}$.
\end{itemize}
\end{theorem}

\begin{remark}
There are a few details regarding Theorem \ref{big_discretization_theorem} that should be clarified:
\begin{itemize}
    \item For a discrete index set $I$ and a function $w:I \to (0, \infty)$, the space $l_{w}^{p}(I)$ for $1 \leq p < \infty$ denotes the sequences $(a_i)_{i \in I}$ such that \begin{equation}
    \label{norm_on_sequence_spaces}
    \left(\sum_{i \in I}|a_{i}|^{p} w(i)^p\right)^{\frac{1}{p}} < \infty.
    \end{equation}
    The case $p = \infty$ is given by replacing summation with supremum. It is straightforward to check that $l_{w}^{p}(I)$ are Banach spaces with the norm \eqref{norm_on_sequence_spaces}. In the setting of Theorem \ref{big_discretization_theorem} the function $w:I \to (0,\infty)$ is obtained by $w(i) := w(x_{i})$, where $w(x_{i})$ is the weight function $w:G \to (0,\infty)$ evaluated at the point $x_{i} \in G$. We use the same notation for the weight function $w:G \to (0, \infty)$ and the induced map $w:I \to (0, \infty)$ on the index set $I$.
    \item The requirement that $g \in \mathcal{A}_{w} \setminus \{0\}$ should be well-behaved is a technical condition. A sufficient criterion in general is that $\mathcal{W}_{g}g$ belongs to certain \textit{Wiener amalgam spaces} \cite[Theorem 3.15]{dahlke2015harmonic}. We refer the reader to \cite{feichtinger1980banach} and the survey \cite{heil2003introduction} for more details on Wiener amalgam spaces.
    \item The idea for the proof of Theorem \ref{big_discretization_theorem} is to approximate the convolution operator \[F \longmapsto F *_{G} \mathcal{W}_{g}g\] with special operators involving the wavelet transform. As these ideas are further elaborated on in \cite[Proof of Theorem 24.2.4]{frames_and_riesz_bases2016}, we will not go more into this. The full proof of Theorem \ref{big_discretization_theorem} can be found in \cite[Theorem 6.1]{feichtinger1989banach1} and in \cite[Theorem 3.15]{dahlke2015harmonic}.
\end{itemize}
\end{remark}

Let us use the language of large scale geometry to make Theorem \ref{big_discretization_theorem} more conceptual: Pick a sufficiently small reference set $Q \subset G$ and the associated family $N = (x_{i})_{i \in I}$ corresponding to the uniform covering $\mathcal{Q} = (x_{i} \cdot Q)_{i \in I}$. Define the trivial map $j:N \to l_{w}^{p}(I)$ given by $j(x_{i}) = \delta_{i}$. Fix a well-behaved element $g \in \mathcal{A}_{w} \setminus \{0\}$ and define $h:G \to \mathcal{C}o_{p,w}(G)$ by $h(x) = \pi(x)g$. Finally, we have a reconstruction map $\mathcal{R}: l_{w}^{p}(I) \to \mathcal{C}o_{p,w}(G)$ given by \[\mathcal{R}((c_{i})_{i \in I}) = \sum_{i \in I}c_{i}\pi(x_i)g.\] Together, these maps form the commutative diagram 
\begin{equation}\label{diagram}
    \begin{tikzcd}
        G \arrow[r, "h"]
        & \mathcal{C}o_{p,w}(G) \\
        N \arrow[r, "j"] \arrow[u, hook]
        & l_{w}^{p}(I) \arrow[u, "\mathcal{R}"]
    \end{tikzcd}
\end{equation}
When $G$ is equipped with the uniform metric $d_{\mathcal{Q}}$ the inclusion $N \xhookrightarrow{} G$ in \eqref{diagram} is a quasi-isometry. The gist of Theorem \ref{big_discretization_theorem} is that this quasi-isometry carries over to the reconstruction map $\mathcal{R}$ in \eqref{diagram}, where it manifests itself as a norm-equivalence.

\subsection{One Banach Frame to Discretize Them All}
\label{sec: Banach_Frames}

Looking back at Theorem \ref{big_discretization_theorem}, we see that it characterizes the elements in $\mathcal{C}o_{p,w}(G)$ in terms of discrete expansions. However, if we are given $f \in \mathcal{R}_{1/w}$, then it might not be obvious to check whether $f \in \mathcal{C}o_{p,w}(G)$ with a set of discrete conditions. This leads us to the following question:

\begin{center}
    \begin{adjustwidth}{40pt}{40pt}
        \textbf{Q:} Given the elements $\pi(x_i)g$ for $i \in I$ in Theorem \ref{big_discretization_theorem}, is it possible to determine if $f \in \mathcal{C}o_{p,w}(G)$ based on the interaction between $f$ and $\pi(x_i)g$ for all $i \in I$?
    \end{adjustwidth}
\end{center}
We show in this section that the answer to the question is affirmative. Before stating the result, we briefly discuss Banach frames to put the result into context.

\begin{definition}
Let $B$ be a separable Banach space. Consider a countable subset $\mathcal{E} = \{g_{i}\}_{i \in I}$ of continuous anti-linear functionals on $B$ together with an associated sequence space $B_{\mathcal{E}}$ on the index set $I$. We say that the pair $(\mathcal{E}, B_{\mathcal{E}})$ is a \textit{Banach frame} for $B$ if the following two properties are satisfied:
\begin{itemize}
    \item The \textit{coefficient operator} $\mathcal{C}_{\mathcal{E}}:B \to B_{\mathcal{E}}$ defined by $\mathcal{C}_{\mathcal{E}}(f) = (\langle g_{i}, f \rangle)_{i \in I}$ for $f \in B$ satisfies the norm-equivalence 
    \begin{equation}\label{eq:Banach_frame_norm_equivalence}
        \|f\|_{B} \asymp \|\mathcal{C}_{\mathcal{E}}(f)\|_{B_{\mathcal{E}}}.
    \end{equation}
    \item There exists a bounded linear map $R_{\mathcal{E}}:B_{\mathcal{E}} \to B$ called the \textit{reconstruction operator} that is a left inverse for $\mathcal{C}_{\mathcal{E}}$.
\end{itemize}
\end{definition}

Explicitly, a reconstruction operator $R_{\mathcal{E}}:B_{\mathcal{E}} \to B$ for the Banach frame $(\mathcal{E}, B_{\mathcal{E}})$ satisfies \[R_{\mathcal{E}}\left((\langle g_{i}, f \rangle)_{i \in I}\right) = f, \qquad f \in B.\]
The notion of a Banach frame was first considered in \cite{gro91}. In \cite[Proposition 2.4]{frames_casazza} it was shown that there exists a Banach frame for any separable Banach space. However, the mere existence of a Banach frame is not necessarily useful as it might be difficult to both understand and compute.

\begin{example}
The most well-studied example of a Banach frame is in the case where $\mathcal{H} = B$ is a separable Hilbert space. Then, by identifying $\mathcal{H}$ with its anti-dual space, we can consider the sequence $\mathcal{E} = \{g_{i}\}_{i \in I} \subset \mathcal{H}$. Moreover, in this case there is a natural sequence space available, namely $l^{2}(I)$. Hence the norm equivalence in \eqref{eq:Banach_frame_norm_equivalence} requires that there exists $A,B > 0$ such that \[A \, \|f\|_{\mathcal{H}}^{2} \leq \sum_{i \in I}|\langle f, g_{i}\rangle|^{2} \leq B \,\|f\|_{\mathcal{H}}^{2}.\]
It turns out the existence of the reconstruction operator $R_{\mathcal{E}}$ is automatically satisfied in this case by \cite[Theorem 3.2.3]{frames_and_riesz_bases2016} and is given by \[R_{\mathcal{E}}((c_{i})_{i \in I}) = \sum_{i \in I}c_{i}g_{i}, \qquad (c_{i})_{i \in I} \in l^{2}(I).\]
In fact, the reconstruction operator $\mathcal{R}_{\mathcal{E}}$ is in this case simply the Hilbert space adjoint of the coefficient operator! In light of these simplifications, it makes sense to simply refer to the collection $\mathcal{E}$ as a \textit{frame} for the Hilbert space $\mathcal{H}$. Frame theory has a prominent place in modern applied harmonic analysis, and we refer the reader to \cite{frames_and_riesz_bases2016} for more on this fascinating topic.
\end{example}

The following result answers the question posed in the introduction of this section, and we refer the reader to the original paper \cite{gro91} for a proof.

\begin{theorem}
\label{banach_frame_theorem}
Consider a $w$-integrable representation $\pi: G \to \mathcal{U}(\mathcal{H}_{\pi})$ where $w:G \to (0,\infty)$ is a weight function. Choose a well-behaved $g \in \mathcal{A}_{w} \setminus \{0\}$ and any sufficiently small reference set $Q$. Then the pair
\[\mathcal{E} = \{\pi(x_{i})g\}_{i \in I}, \qquad B_{\mathcal{E}} = l_{w}^{p}(I)\] is a Banach frame for the coorbit space $\mathcal{C}o_{p,w}(G)$, where $1 \leq p \leq \infty$ and $\mathcal{Q} = (x_i \cdot Q)_{i \in I}$ is the uniform covering corresponding to the reference set $Q$. 
\end{theorem}

\begin{remark}\hfill
\begin{itemize}
    \item Notice that, under the assumptions in Theorem \ref{banach_frame_theorem}, the elements in $\mathcal{E} = \{\pi(x_{i})g\}_{i \in I}$ belong to $\mathcal{H}_{w}^{1}$. Hence it makes sense for $f \in \mathcal{C}o_{p,w}(G) \subset \mathcal{R}_{1/w}$ to consider the duality pairing \[\mathcal{W}_{g}f(x_{i}) = \langle f, \pi(x_{i})g  \rangle_{\mathcal{R}_{1/w}, \mathcal{H}_{w}^{1}}.\] As such, the coefficient operator $\mathcal{C}_{\mathcal{E}}$ in this case is simply given by sampling on the points $\{x_{i}\}_{i \in I}  \subset G$, that is, 
    \begin{equation*}
        \mathcal{C}_{\mathcal{E}}(f) = \left(\mathcal{W}_{g}f(x_{i})\right)_{i \in I}.
    \end{equation*}
    \item The reader should be aware that although the collection $\mathcal{E} = \{\pi(x_{i})g\}_{i \in I}$ is fixed for each $1 \leq p \leq \infty$, the sequence space $B_{\mathcal{E}} = l_{w}^{p}(I)$ does indeed depend on $p$. Since we have defined a Banach frame as the pair $(\mathcal{E}, B_{\mathcal{E}})$, we are being slightly imprecise when stating that Theorem~\ref{banach_frame_theorem} provides a single Banach frame for all the coorbit spaces $\mathcal{C}o_{p,w}(G)$ for $1 \leq p \leq \infty$.
\end{itemize}
\end{remark}

\begin{example}
Consider for $\alpha, \beta > 0$ the uniform covering $\mathcal{Q}_{\alpha, \beta}$ of $\mathbb{R}^{2n}$ given by \[\mathcal{Q}_{\alpha, \beta} := ((\alpha k, \beta l) + Q)_{k,l \in \mathbb{Z}^{n}}, \qquad Q := [-\alpha, \alpha]^{n} \times [-\beta, \beta]^{n}.\] Then for $g \in \mathcal{S}(\mathbb{R}^{n})$ and sufficiently small $\alpha, \beta$ we have that 
\[\mathcal{E} = \{M_{\beta l}T_{\alpha k}g\}_{k,l \in \mathbb{Z}^{n}}, \qquad B_{\mathcal{E}} = l_{v_{s}}^{p}(\mathbb{Z}^{2n}),\] is a Banach frame for the modulation space $M_{s}^{p}(\mathbb{R}^{n}) := M_{v_s}^{p,p}(\mathbb{R}^{n}),$ where $v_{s}$ for $s \geq 0$ is the polynomial weight given in Example \ref{weighted_Feichtinger_space}. The collection $\mathcal{E}$ is often called a \textit{Gabor system} in the literature. Hence we have the norm-equivalence 
\[\|f\|_{M_{s}^{p}(\mathbb{R}^{n})} \asymp \left(\sum_{k,l \in \mathbb{Z}^{n}}\left|V_{g}f(\alpha k, \beta l)\right|^{p}(1 + |\alpha k|^2 + |\beta l|^2)^{\frac{sp}{2}}\right)^{\frac{1}{p}}.\]
\end{example}

\subsection{A Kernel Theorem for Coorbit Spaces}
\label{sec: A Kernel Theorem for Coorbit Spaces}

The \textit{Schwartz kernel theorem} is one of the most influential results in distribution theory. It states that any continuous linear operator $A:\mathcal{S}(\mathbb{R}^{n}) \to \mathcal{S}'(\mathbb{R}^{n})$ can be represented by a unique distributional kernel $K \in \mathcal{S}'(\mathbb{R}^{2n})$ in the sense that 
\begin{equation}
\label{eq:kernel_theorem_Schwartz}
    \langle Af, g \rangle = \langle K, f \otimes g \rangle, \qquad f,g \in \mathcal{S}(\mathbb{R}^{n}).
\end{equation}
If $K$ is a locally integrable function, then we have that $K$ is indeed an integral kernel in the sense that 
\[\langle Af, g \rangle = \int_{\mathbb{R}^{n}}K(x,y)f(y)\overline{g(x)}\, dy \, dx, \qquad f,g \in \mathcal{S}(\mathbb{R}^{n}).\] \par
In \cite{fei80}, Hans Georg Feichtinger showed that a kernel theorem is also valid for the modulation spaces. More precisely, he showed that any continuous linear operator $A: M^{1}(\mathbb{R}^{n}) \to M^{\infty}(\mathbb{R}^{n})$ can be represented as in \eqref{eq:kernel_theorem_Schwartz} with $K \in M^{\infty}(\mathbb{R}^{2n})$. This result shows that kernel theorems are also possible in the Banach space setting. Building on this, the authors in \cite{speck19} have recently extended Feichtinger's kernel theorem to coorbit spaces. In this section, we showcase their results without weights for simplicity and refer the reader to the well-written paper \cite{speck19} for more information.\par 
Consider two integrable representations $\pi_{1}:G_{1} \to \mathcal{U}(\mathcal{H}_{1})$ and $\pi_{2}:G_{2} \to \mathcal{U}(\mathcal{H}_{2})$. From this we obtain the corresponding coorbit spaces $\mathcal{C}o_{p}^{\pi_{1}}(G_{1})$ and $\mathcal{C}o_{q}^{\pi_{2}}(G_{2})$ for all $1 \leq p,q \leq \infty$. The goal is to represent any continuous and linear operator $A: \mathcal{C}o_{1}^{\pi_{1}}(G_{1}) \to \mathcal{C}o_{\infty}^{\pi_{2}}(G_{2})$ though a distributional kernel $K$ in an appropriate sense. The first step is to identify which space the distributional $K$ should be taken from. To do this, we briefly review tensor products of representations.

\begin{definition}
We can consider the \textit{tensor product representation} $\pi_{1} \otimes \pi_{2}$ from $G_{1} \times G_{2}$ to unitary operators on the tensor product $\mathcal{H}_{2} \otimes \mathcal{H}_{1}$ given on simple tensors $\psi_{2} \otimes \psi_{1} \in \mathcal{H}_{2} \otimes \mathcal{H}_{1}$ by 
\[(\pi_{1} \otimes \pi_{2})(g_{1}, g_{2})(\psi_{2} \otimes \psi_{1}) := \pi_{2}(g_{2})\psi_{2} \otimes \pi_{1}(g_{1})\psi_{1}, \qquad (g_{1}, g_{2}) \in G_{1} \times G_{2}.\]
\end{definition}

It is straightforward to verify that if $\psi_1 \in \mathcal{H}_{1}$ and $\psi_{2} \in \mathcal{H}_{2}$ are integrable vectors for respectively $\pi_{1}$ and $\pi_{2}$, then $\psi_{2} \otimes \psi_{1} \in \mathcal{H}_{2} \otimes \mathcal{H}_{1}$ is an integrable vector for the tensor product representation $\pi_{1} \otimes \pi_{2}$. As such, it makes sense to consider the coorbit space \[\mathcal{C}o_{p}(G_{1} \times G_{2}) \coloneqq \mathcal{C}o_{p}^{\pi_{1} \otimes \pi_{2}}(G_{1} \times G_{2}), \qquad 1 \leq p \leq \infty,\] associated to the tensor product representation $\pi_{1} \otimes \pi_{2}$. The following result from \cite[Theorem 3]{speck19} shows that a kernel theorem is valid for general coorbit spaces.

\begin{theorem}[Coorbit Kernel Theorem]
\label{general_kernel_theorem}
Let $\pi_{i}:G_{i} \to \mathcal{U}(\mathcal{H}_{i})$ for $i = 1,2$ be two integrable representations. There is a bijective correspondence between bounded linear operators \[A: \mathcal{C}o_{1}^{\pi_{1}}(G_{1}) \to \mathcal{C}o_{\infty}^{\pi_{2}}(G_{2})\] and elements $K \in \mathcal{C}o_{\infty}(G_{1} \times G_{2})$ given by 
\begin{equation}
\label{eq:kernel_theorem_general}
    \langle Af, g \rangle = \langle K, f \otimes g \rangle,
\end{equation}
where $f \in \mathcal{C}o_{1}^{\pi_{1}}(G_{1})$ and $g \in \mathcal{C}o_{1}^{\pi_{2}}(G_{2})$. Moreover, we have the norm-equivalence 
\[\|A\|_{Op} \asymp \|K\|_{\mathcal{C}o_{\infty}(G_{1} \times G_{2})}.\]
\end{theorem}

The reader is referred to \cite[Section 5]{speck19}
for concrete applications of Theorem \ref{general_kernel_theorem} regarding mappings between Besov spaces and modulation spaces. In light of Theorem \ref{general_kernel_theorem}, it makes sense to refer to $K$ in \eqref{eq:kernel_theorem_general} as the \textit{distributional kernel} of the operator $A: \mathcal{C}o_{1}^{\pi_{1}}(G_{1}) \to \mathcal{C}o_{\infty}^{\pi_{2}}(G_{2})$. The authors in \cite{speck19} go on to use Theorem \ref{general_kernel_theorem} to deduce properties of $A$ based on knowledge of its distributional kernel $K$. In particular, they show the following elegant result in \cite[Theorem 9]{speck19}.

\begin{corollary}
In the notation of Theorem \ref{general_kernel_theorem}, the operator $A$ defines a bounded linear operator $A: \mathcal{C}o_{\infty}^{\pi_{1}}(G_{1}) \to \mathcal{C}o_{1}^{\pi_{2}}(G_{2})$ when its distributional kernel $K$ satisfies $K \in \mathcal{C}o_{1}(G_{1} \times G_{2})$.
\end{corollary} 

\section{Examples and Recent Developments}
\label{sec: Chapter 4}

Now that all the main features of coorbit spaces have been discussed, we will briefly outline in Sections~\ref{sec: Shearlet_Spaces} - \ref{sec: Coorbit_Spaces_on_Nilpotent_Lie_Groups} examples from different areas of modern analysis. The goal here is not to give a comprehensive exposition on each topic, nor to give a comprehensive account of all the applications of coorbit theory. We rather strive to convince the reader that coorbit theory is an active research topic that unifies seemingly different branches of modern analysis. We will in Sections \ref{sec: Shearlet_Spaces} - \ref{sec: Coorbit_Spaces_on_Nilpotent_Lie_Groups} provide references for further reading so that the reader can look more into the most eye-caching example themselves. Finally, in subsection \ref{sec: Where_To_Go_Next?} we give references to modern directions in coorbit theory, as well as suggestions for where the reader can learn more about coorbit theory.

\subsection{Shearlet Spaces}
\label{sec: Shearlet_Spaces}

For image analysis and image processing, the continuous wavelet transform given in \eqref{classical_wavelet_transform} has been extensively used. However, the continuous wavelet transform can fall short if one wishes to extract directional information. Several approaches have been developed to provide an alternative to the continuous wavelet transform, e.g. \textit{ridgelets} and \textit{curvelets} \cite{ candes1999ridgelets, candes2000curvelets, zhang2008wavelets}. The most examined alternative, namely \textit{shearlets}, does have a description that allows the theory of coorbit spaces to be applied. We refer the reader to \cite{guo2006sparse, labate2005sparse} for the origins of shearlets and to \cite{kutyniok2012introduction} for a general introduction to shearlets. In this section, we will describe the shearlet transform and the underlying shearlet group in two dimensions following \cite{dahlke2009shearlet}. The extension to higher dimensions was given in \cite{dahlke2010continuous}. \par 
To begin describing the shearlet group we first need two matrices: For $a \in \mathbb{R}^{*} := \mathbb{R} \setminus \{0\}$ the \textit{parabolic scaling matrix} $A_{a}$ is given by \[A_{a} := \begin{cases}
\begin{pmatrix} a &\phantom{-} 0 \\  0 &\phantom{-} \sqrt{a}\phantom{-} \end{pmatrix}, \, \textrm{ when } a > 0, \\ \vspace{-0.3cm} \\ 
\begin{pmatrix} a & \phantom{-} 0 \\ 0 & -\sqrt{-a} \end{pmatrix}, \, \textrm{ when } a < 0.
\end{cases}\]
Hence $A_{a}$ for $a > 0$ scales the first axis with the squared length of the scaling of the second axis. For $s \in \mathbb{R}$ the \textit{shear matrix} $S_{s}$ is given by \[S_{s} := \begin{pmatrix} 1 & s \\ 0 & 1 \end{pmatrix}.\]
Using these matrices, we can define the shearlet group as follows.
\begin{definition}
The \textit{(full) shearlet group} $\mathbb{S}$ is defined to be $\mathbb{R}^{*} \times \mathbb{R} \times \mathbb{R}^2$ with the group operation 
\begin{equation}
\label{shearlets_group_operation}
    (a, s, t) \cdot_{\mathbb{S}} (a', s', t') := (aa', s + s'\sqrt{|a|}, t + S_{s}A_{a}t').
\end{equation} 
\end{definition}

It is straightforward to check that \eqref{shearlets_group_operation} is in fact a group operation with identity element $(1, 0, 0) \in \mathbb{S}$, see \cite[Lemma 2.1]{dahlke2008uncertainty} for details. Notice that $\mathbb{S}$ has two connected components; the identity component $\mathbb{S}^{+}$ is called the \textit{connected shearlet group}. The left Haar measure $\mu_{L}$ on the shearlet group $\mathbb{S}$ is given by \[\mu_{L}(a, s, t) = \frac{da \, ds \, dt}{|a|^3}, \quad (a, s, t) \in \mathbb{S}.\] \par
Given an invertible matrix $M \in GL(2, \mathbb{R})$ we can consider the \textit{generalized dilation operator} $D_{M}$ acting on $f \in L^{2}(\mathbb{R}^{2})$ by the formula 
\begin{equation}
\label{generalized_dilation_operator}
D_{M}f(x) := \frac{1}{\sqrt{|\textrm{det}(M)|}}f(M^{-1}x), \qquad x \in \mathbb{R}^2.    
\end{equation}
Notice that \eqref{generalized_dilation_operator} is a two-dimensional generalization of the dilation operator given in \eqref{usual_dilation_operator}. 

\begin{definition}
The \textit{(continuous) shearlet representation} $\pi: \mathbb{S} \to \mathcal{U}(L^{2}(\mathbb{R}^{2}))$ is given by \[\pi(a, s, t)f(x) := T_{t}D_{S_{s}A_{a}}f(x) = \frac{1}{\sqrt{|\textrm{det}(S_{s}A_{a})|}}f\left(A_{a}^{-1}S_{s}^{-1}(x - t)\right),\]
where $f \in L^{2}(\mathbb{R}^{2})$ and $(a,s,t) \in \mathbb{S}$.
\end{definition}

One can view the unitary representation $\pi$ as a two-dimensional version of the continuous wavelet representation in \eqref{classical_wavelet_representation}.
The representation $\pi$ is irreducible since we are considering the full shearlet group $\mathbb{S}$ instead of the connected group $\mathbb{S}^{+}$, see \cite[Theorem 2.2]{dahlke2009shearlet}. Moreover, \cite[Theorem 2.2]{dahlke2009shearlet} also shows that $\pi$ is square integrable. More precisely, a function $g \in L^{2}(\mathbb{R}^{2})$ is admissible if and only if 
\begin{equation}
\label{eq:admissibility_shearlets}
    \int_{\mathbb{R}^{2}}\frac{|\mathcal{F}(g)(\omega_1,\omega_2)|^2}{\omega_{1}^2} \, d\omega_1 \, d\omega_2 = 1.
\end{equation}
We refer to the elements $g \in L^{2}(\mathbb{R}^{2})$ satisfying \eqref{eq:admissibility_shearlets} as \textit{(continuous) shearlets}. Although it is common in the literature to denote the wavelet transform corresponding to the shearlet representation by $\mathcal{SH}$, we will stick with our predefined notation $\mathcal{W}$ for consistency. Hence for a shearlet $g \in L^{2}(\mathbb{R}^{2})$ and $f \in L^{2}(\mathbb{R}^{2})$ the orthogonality relation \eqref{wavelet_transform_orthogonality} shows that
\[\int_{\mathbb{S}} |\mathcal{W}_{g}f(a, s, t)|^2 \, \mu_{L}(a, s, t) = \int_{\mathbb{S}} |\langle f, \pi(a, s, t)g \rangle|^2 \, \frac{da \, ds \, dt}{|a|^{3}} = \|f\|_{L^{2}(\mathbb{R}^2)}^2.\] \par
Let us for simplicity consider the polynomial weights \[v_{\alpha}(a, s, t) := (1 + a^2 + s^2)^{\alpha/2}\] for $(a, s, t) \in \mathbb{S}$ and $\alpha \geq 0$. The existence of a $v_{\alpha}$-integrable vector is guaranteed by \cite[Theorem~4.2]{dahlke2009shearlet}. Thus we obtain the space of $v_{\alpha}$-test vectors $\mathcal{H}_{v_{\alpha}}^{1}$ and the $v_{\alpha}$-reservoir $\mathcal{R}_{1/v_{\alpha}}$, see Section \ref{sec: Weighted_Versions} for details. The following definition is inevitable.

\begin{definition}
The \textit{shearlet coorbit spaces} $\mathcal{C}o_{p,v_{\alpha}}(\mathbb{S})$ for $1 \leq p \leq \infty$ are defined to be 
\[\mathcal{C}o_{p,v_{\alpha}}(\mathbb{S}) := \{f \in \mathcal{R}_{1/v_{\alpha}} \, : \, \mathcal{W}_{g}f \in L_{v_{\alpha}}^{p}(\mathbb{S})\},\]
where $g$ is any $v_{\alpha}$-integrable vector.
\end{definition}

We invoke Theorem \ref{summary_theorem} to deduce that the shearlet coorbit spaces $\mathcal{C}o_{p,v_{\alpha}}(\mathbb{S})$ constitute a well-behaved class of Banach spaces. In \cite[Theorem 4.7]{dahlke2009shearlet} it is shown that the shearlet coorbit spaces contain many smooth functions of rapid decay. We refer the reader to \cite[Section 4.2]{dahlke2009shearlet} for results regarding atomic decompositions and Banach frames for the shearlet coorbit spaces.

\subsection{Bergman Spaces and the Blaschke Group}
\label{sec: Bergman_Spaces_and_the_Blaschke_Group}

We will now describe an application of coorbit spaces to the realm of classical complex analysis, namely the Bergman spaces. The connection with Bergman spaces was to our knowledge initially pointed out in \cite[Section 7.3]{feichtinger1988unified}. To introduce this topic in a brief and succinct manner, we will give an outline of the definitions and results given in \cite{pap2012properties} and \cite{feichtinger2014coorbit}. We encourage the reader to seek out the more recent and technical paper \cite{christensen2015} for interesting results in higher dimensions.

Let us first recall the Bergman spaces in classical complex analysis. We denote the unit disk in the complex plane by $\mathbb{D}$ and consider for $\alpha > -1$ the weighted area measure \[dA_{\alpha}(z) = \frac{\alpha + 1}{\pi}\left(1 - |z|^2\right)^{\alpha} \, dz, \quad z \in \mathbb{D}.\] We let $A_{\alpha}^{p} := A_{\alpha}^{p}(\mathbb{D})$ denote the \textit{(weighted) Bergman space} consisting of analytic functions $f:\mathbb{D} \to \mathbb{C}$ such that \[\int_{\mathbb{D}}|f(z)|^{p} \, dA_{\alpha}(z) < \infty.\] For $p = 2$ we have a natural Hilbert space structure on $A_{\alpha}^{2}$ given by the inner product \[\langle f, g \rangle_{\alpha} := \int_{\mathbb{D}}f(z)\overline{g(z)} \, dA_{\alpha}(z).\]
It not difficult to verify that $A_{\alpha}^{2}$ is a reproducing kernel Hilbert space with reproducing kernel
\[K_{\alpha}(z,w) = \frac{1}{(1 - z\overline{w})^{\alpha + 2}}, \qquad z,w \in \mathbb{D}.\] \par
We will now describe a group that acts unitarily on $A_{\alpha}^{2}$: For $\mathbb{B} := \mathbb{D} \times \mathbb{T}$ we say that a function on the form \[B_{a}(z) := \epsilon \frac{z - b}{1 - z\overline{b}}, \qquad z \in \mathbb{C}, \, \, a := (b, \epsilon) \in \mathbb{B}, \, \, z\overline{b} \neq 1,\] is called a \textit{Blaschke function}. The Blaschke functions allows us to define a group operation on $\mathbb{B}$ by the formula $a_1 \circ a_2 = a_{3}$ if and only if $B_{a_1} \circ B_{a_2} = B_{a_3}$. The locally compact group $(\mathbb{B}, \circ)$ is a unimodular, non-commutative group known as the \textit{Blaschke group}.

\begin{remark}
The terminology is motivated by the Blaschke products in complex analysis: A sequence $(a_{n})_{n \in \mathbb{N}}$ in $\mathbb{D}$ satisfies the \textit{Blaschke condition} when
\[\sum _{n = 1}^{\infty}(1-|a_{n}|) < \infty.\]
Given such a sequence, we  define the \textit{Blaschke product} as the infinite product \[B(z)=\prod _{n = 1}^{\infty}B(a_{n},z), \qquad B(a,z) = \frac{|a|}{a}\frac{a-z}{1-\overline{a}z},\]
with the convention that $B(0,z) = z$. Then $B$ is an analytic function in $\mathbb{D}$ vanishing precisely at the points $(a_{n})_{n \in \mathbb{N}}$.
\end{remark}

Introduce the functions \[F_{a}(z) := \frac{\sqrt{\epsilon(1-|b|^2)}}{1 - z\overline{b}}, \qquad a = (b, \epsilon) \in \mathbb{B}, \, z \in \mathbb{D}.\] We obtain for each $\alpha \geq 0$ a unitary representation $U_{\alpha}:\mathbb{B} \to \mathcal{U}(A_{\alpha}^{2})$ given by \[U_{\alpha}(a)f(z) = \left[F_{a^{-1}}(z)\right]^{\alpha + 2}f\left(B_{a}^{-1}(z)\right) = \left[F_{a^{-1}}(z)\right]^{\alpha + 2}f\left(B_{a^{-1}}(z)\right), \qquad f \in A_{\alpha}^{2}, \, a \in \mathbb{B}, \, z \in \mathbb{D}.\] The representation $U_{\alpha}$ is square integrable and any $g \in A_{\alpha}^{2}$ satisfying $\|g\|_{\alpha} = \pi^{-1}\sqrt{\alpha + 1}$ is admissible. For the wavelet transform $\mathcal{W}_{g}^{\alpha}f(a) := \langle f, U_{\alpha}(a)g\rangle_{\alpha}$ with $f,g \in A_{\alpha}^{2}$ and $g$ admissible we have by Theorem \ref{reproducing_formula} that \[\mathcal{W}_{g}^{\alpha}f = \mathcal{W}_{g}^{\alpha}f *_{\mathbb{B}} \mathcal{W}_{g}^{\alpha}g.\] Moreover, we can by Corollary \ref{reconstruction_corollary} reconstruct any $f \in \mathcal{A}_{\alpha}^{2}$ through the weak integral formula 
\begin{align*}
   f(z) & = \int_{\mathbb{B}}\mathcal{W}_{g}^{\alpha}f(a) (U_{\alpha}(a)g)(z) \, d\mu_{L}(a)
   \\ & =
   \frac{1}{2\pi} \int_{-\pi}^{\pi}\int_{\mathbb{D}}\frac{\mathcal{W}_{g}^{\alpha}f(b, e^{it}) (U_{\alpha}(b, e^{it})g)(z)}{(1 - |b|^2)^2} \, db \, dt.
\end{align*}

A straightforward computation shows that for $g \equiv 1 \in \mathcal{A}_{\alpha}^{2}$ we have $\|\mathcal{W}_{g}^{\alpha}g\|_{L^{1}(\mathbb{B})} = 2/\alpha$. Hence for $\alpha > 0$ we can conclude that the representation $U_{\alpha}$ is integrable. More generally, it is showed in \cite[Theorem 3.2.2]{pap2012properties} that any non-zero analytic function $g$ on the unit disk that can be written as \[g(z) = \sum_{j = 0}^{\infty}\lambda_{j}\frac{z - b_{j}}{1 - z\overline{b_{j}}}\] with $|b_{j}| \leq 1$ for all $j \geq 0$ and \[\sum_{j = 0}^{\infty}|\lambda_{j}| < \infty,\]
is an integrable vector for the representation $U_{\alpha}$ for $\alpha > 0$. For $\alpha > 0$ we define the space of test vectors $\mathcal{H}_{\alpha}^{1} \subset A_{\alpha}^{2}$ and the reservoir space $\mathcal{R}_{\alpha}$ as usual, see Section \ref{sec: Test_Spaces_and_Distributional_Spaces} and \ref{sec: Reservoirs_and_the_Extended_Wavelet_Transform} respectively for details. As such, we can define coorbit spaces associated to $U_{\alpha}$.

\begin{definition}
The \textit{Blaschke coorbit spaces} $\mathcal{C}o_{p,\alpha}(\mathbb{B})$ for $1 \leq p \leq \infty$ and $\alpha > 0$ are defined to be 
\[\mathcal{C}o_{p,\alpha}(\mathbb{B}) := \{f \in \mathcal{R}_{\alpha} \, : \, \mathcal{W}_{g}^{\alpha}f \in L^{p}(\mathbb{B})\},\]
where $g$ is any integrable vector for $U_{\alpha}$.
\end{definition}

By the theory we have developed, we can automatically deduce all the consequences in Theorem \ref{summary_theorem} for the Blaschke coorbit spaces $\mathcal{C}o_{p,\alpha}(\mathbb{B})$. For discretization results, the reader can first consult \cite[Section 3.3]{pap2012properties} and proceed to \cite[Theorem 3.14]{christensen2015} where the classical atomic decomposition results for Bergman spaces by Coifman and Rochberg are deduced through coorbit theory.

\subsection{Coorbit Spaces on Nilpotent Groups}
\label{sec: Coorbit_Spaces_on_Nilpotent_Lie_Groups}

It is clear from Example \ref{sec: Orthogonality_of_the_Wavelet_Transform} that the modulation spaces are intrinsically linked with the Heisenberg group. The Heisenberg group fits in with a large class of well-behaved locally compact groups known as \textit{nilpotent Lie groups}. We refer the reader to \cite{fischer2016quantization} for the definition of a nilpotent Lie group. In view of this observation, it makes sense to try to define coorbit spaces analogous to the modulation spaces for other nilpotent groups. This is a recent idea that was first seriously considered in \cite{fischer2018heisenberg} and recently expanded on in \cite{grchenig2020new}. We will outline basic definitions and results in this direction following \cite{grchenig2020new}. The interested reader should consult \cite{fischer2018heisenberg, grchenig2020new} for more details and interesting examples. \par

Let $G$ be a simply connected nilpotent Lie group with \textit{center} 
\[\mathcal{Z} := \mathcal{Z}(G) := \{x \in G \, : \, xy = yx \textrm{ for all } y \in G\}.\] We will consider the quotient group $G/\mathcal{Z}$ with its Haar measure $\mu_{G/\mathcal{Z}}$. An irreducible unitary representation $\pi:G \to \mathcal{U}(\mathcal{H}_{\pi})$ is said to be \textit{square integrable modulo the center} if there exists an element $g \in \mathcal{H}_{\pi}$ such that 
\begin{equation}
\label{square_integrable_modulo_center}
    \int_{G/\mathcal{Z}}|\mathcal{W}_{g}g(\overline{x})|^2 \, d\mu_{G/\mathcal{Z}}(\overline{x}) < \infty,
\end{equation}
where as usual $\mathcal{W}_{g}f(x) := \langle f, \pi(x)g \rangle$ for $f,g \in \mathcal{H}_{\pi}$ and $x \in G$. Since $\pi|_{\mathcal{Z}}(x) = \chi(x) \cdot \textrm{Id}_{\mathcal{H}_{\pi}}$ where $\chi$ is a character of the commutative group $\mathcal{Z}$, it follows that the integrand in \eqref{square_integrable_modulo_center} is a well defined function on the quotient group $G/\mathcal{Z}$. We remind the reader that the reduction from $G$ to the quotient group $G/\mathcal{Z}$ is precisely what we did in Example \ref{example_Heisenberg_reduction} to make the Schr\"{o}dinger representation square integrable. Hence we can say that the Schr\"{o}dinger representation $\rho:\mathbb{H}^{n} \to \mathcal{U}(L^{2}(\mathbb{R}^{n}))$ given in \eqref{Schrodinger_representation} is square integrable modulo the center.\par
To proceed, we first need a good choice for a well-behaved \textquote{window function} $g \in \mathcal{H}_{\pi}$. Since $G$ is a Lie group it has a smooth structure and it makes sense to ask for a fixed $g \in \mathcal{H}_{\pi}$ whether the function 
\begin{equation}
\label{smooth_map}
G \ni x \mapsto \pi(x)g \in \mathcal{H}_{\pi}    
\end{equation}
is a smooth map from $G$ to $\mathcal{H}_{\pi}$. Details for this can be found in \cite[Chapter 1.7]{fischer2016quantization}. We refer to the elements $g \in \mathcal{H}_{\pi}$ such that \eqref{smooth_map} is a smooth map as the \textit{smooth vectors} of the representation $\pi$ and denote them by $\mathcal{H}_{\pi}^{\infty}$. It is a general fact that $\mathcal{H}_{\pi}^{\infty}$ is dense in $\mathcal{H}_{\pi}$, see \cite[Proposition 1.7.7]{fischer2016quantization}. 

\begin{definition}
Let $G$ be a simply connected nilpotent Lie group with center $\mathcal{Z}$. Assume we have a square integrable representation modulo the center $\pi: G \to \mathcal{U}(\mathcal{H}_{\pi})$ and let $\mathcal{H}_{\pi}^{\infty}$ denote the corresponding smooth vectors. We define the \textit{coorbit space} $\mathcal{C}o_{p}(G/\mathcal{Z})$ for $1 \leq p < \infty$ to be the completion of the subspace of elements $f \in \mathcal{H}_{\pi}^{\infty}$ such that \[\|f\|_{\mathcal{C}o_{p}(G/\mathcal{Z})} = \left(\int_{G/\mathcal{Z}}|\mathcal{W}_{g}f(\overline{x})|^p \, d\mu_{G/\mathcal{Z}}(\overline{x})\right)^{\frac{1}{p}} < \infty,\]
where $g \in \mathcal{H}_{\pi}^{\infty}$ is a fixed non-zero smooth vector.
\end{definition}

\begin{remark}
The case of $\mathcal{C}o_{\infty}(G/\mathcal{Z})$ can be handled by considering weak closures, but we restrict ourselves to $1 \leq p < \infty$ for simplicity. Moreover, we also refrain from considering weighted version of $\mathcal{C}o_{p}(G/\mathcal{Z})$ so that we can focus on the essential features.
\end{remark}

Although the representation space $\mathcal{H}_{\pi}$ has an abstract flavor in general, it can be shown that for nilpotent groups one can always realize $\mathcal{H}_{\pi}$ as $L^{2}(\mathbb{R}^{s})$ in a natural way. We point out that the parameter $s$ generally satisfies $s < \textrm{dim}(G)$. The identification of $\mathcal{H}_{\pi}$ with $L^{2}(\mathbb{R}^{s})$ uses Kirillov's theory of coadjoint orbits (not to be confused with coorbit theory). We refer the reader to the standard reference \cite[Chapter 3]{kirillov2004lectures} for a more careful explanation of this phenomenon. \par
One important problem for coorbit spaces on nilpotent groups is whether the new spaces are identical to the classical modulation spaces. If this was the case, then coorbit spaces on nilpotent groups would just be a more complicated view of the usual modulation spaces and offer little of value. The following example, taken from \cite[Example 3.2]{grchenig2020new}, illustrates that this can actually happen.

\begin{example}
We consider the nilpotent group $G$ with the concrete realization $G \simeq (\mathbb{R}^{6},\cdot)$ where 
\[x \cdot y := (x_{1} + y_{1} + x_{5}y_{3} + x_{6}y_{4}, x_{2} + y_{2} + x_{6}y_{5}, x_{3} + y_{3}, x_{4} + y_{4}, x_{5} + y_{5}, x_{6} + y_{6}).\] A square integrable representation modulo the center is $\pi: G \to \mathcal{U}(L^{2}(\mathbb{R}^{2}))$ given by \begin{align*}
    \pi(x_{1}, \dots, x_{6})g(s,t) = e^{2\pi i(x_1 - x_{3}s - x_{4}t)}g(s - x_{5}, t - x_{6}) = e^{2 \pi i x_1}M_{(-x_{3}, -x_{4})}T_{(x_{5}, x_{6})}g(s,t),
\end{align*}
where $T$ and $M$ are the translation operator and modulation operator given in \eqref{time_shift_and_frequency_shift}. As our goal is to investigate the integrability of the corresponding wavelet transform, we henceforth drop the phase factor $e^{2\pi i x_{1}}$ as this will be insignificant. We identify $G/\mathcal{Z} \simeq \mathbb{R}^{4}$ and write \[\overline{x} = (0, 0, x_{3}, x_{4}, x_{5}, x_{6}) \in G/\mathcal{Z}.\] The wavelet transform $\mathcal{W}_{g}f$ for $f \in L^{2}(\mathbb{R}^{2})$ and a non-zero $g \in \mathcal{H}_{\pi}^{\infty}$ is given by 
\[\mathcal{W}_{g}f(\overline{x}) = V_{g}f((x_{5}, x_{6}), (-x_{3}, -x_{4})),\]
where $V_{g}f$ is the STFT. From this it follows that for $1 \leq p < \infty$ we have $\mathcal{C}o_{p}(G/\mathcal{Z}) = M^{p}(\mathbb{R}^{2})$ since \[\|f\|_{\mathcal{C}o_{p}(G/\mathcal{Z})} \simeq \|f\|_{M^{p}(\mathbb{R}^{2})}.\] 
\end{example}

In light of the previous example, one might fear that coorbit spaces associated with nilpotent groups never produce anything other than the classical modulation spaces. However, in \cite{grchenig2020new} several examples are given of coorbit spaces on nilpotent groups that are not equal to any of the classical modulation spaces. The first example of this phenomenon was presented in \cite[Theorem 7.6]{fischer2018heisenberg}. The group in question was the \textit{Dynin-Folland group}, and the techniques used to prove distinctness came from the theory of \textit{decomposition spaces}. Distinctness of a class of decomposition spaces on two-step nilpotent groups was proved in \cite[Theorem 5.6]{berge2019modulation}. 

\subsection{At the Finishing Line}
\label{sec: Where_To_Go_Next?}

Phew! You're still here? Good. Hopefully you have been convinced that coorbit theory is an exciting research topic. You now understand the main ideas of coorbit theory along with several concrete examples. If you are satisfied, then congratulations; you know the basics of coorbit theory. However, if you are interested in doing research in coorbit theory, then the journey has just started. \par 
A great way to get more familiar with technical aspects of coorbit theory is by reading the Ph.D. thesis of Felix Voigtlaender \cite{felix_thesis}. We also recommend seeking out the original papers on coorbit theory \cite{feichtinger1988unified, feichtinger1989banach1, feichtinger1989banach2}. Reading these sources is will improve your fundamental knowledge of coorbit theory. A good idea is to find a problem in coorbit theory that you want to solve. This forces you to work through details that is tempting to skip when reading other peoples work. Below we have given some references for two directions that have received much attention in recent decades: 
\begin{itemize}
    \item Consider two integrable representations $\pi_{1}:G_{1} \to \mathcal{U}(\mathcal{H}_{1})$ and $\pi_{2}:G_{2} \to \mathcal{U}(\mathcal{H}_{2})$ and two parameters $1 \leq p,q \leq \infty$. A natural question to answer is whether there exists a continuous embedding \[\phi: \mathcal{C}o_{p}^{\pi_{1}}(G_{1}) \to \mathcal{C}o_{q}^{\pi_{2}}(G_{2})\] between different coorbit spaces corresponding to (possibly) different groups. This question has been considered in many concrete settings, see e.g.\ \cite{ char18, Toft12} for the modulation spaces and Besov spaces, and \cite{dahlke2011shearlet} for embeddings between shearlet coorbit spaces. The embedding question is often more easily tackled if the coorbit spaces in question can be given a \textit{decomposition space} structure. Decomposition spaces originate from \cite{Hans_Grobner} and many general embedding results between decomposition spaces can be found in \cite{voigtlaender2019embeddings}. We refer the authors to \cite{FUHR201580} where the authors show that a large class of wavelet spaces can be given a decomposition space structure. In \cite{voigtlaender2016embeddings} several embedding results from decomposition spaces into Sobolev spaces and BV spaces are given. Specific embeddings between decomposition spaces with a geometric flavor have recently been investigated in \cite{berge2019modulation, eiriklargescale}. Finally, recent results regarding embeddings of shearlet coorbit spaces into Sobolev spaces can be found in \cite{rene2020}.
    \item There are plenty of directions where coorbit theory can be generalized: As previously mentioned, one can instead of $L^{p}(G)$ for $1 \leq p \leq \infty$ in the definition of $\mathcal{C}o_{p}(G)$ consider $\mathcal{C}o(Y)$, where $Y$ is a \textit{solid} and \textit{translation invariant} Banach space of functions on $G$, see \cite{feichtinger1988unified, feichtinger1989banach1, feichtinger1989banach2}. We refer the reader to \cite{Rauhut05, felix_thesis, Quasiandinhomogeneous} for results regarding coorbit spaces in the quasi-Banach setting. The paper \cite{Dahlke2019} considers coorbit spaces associated with representations that are not necessarily integrable, while \cite{fuhr2020coorbit} considers certain representations that are not necessarily irreducible. In \cite{christensen1996} it is shown that atomic decompositions are valid even for projective representations. Coorbit theory for homogeneous spaces have been investigated, and we suggest to start with the papers \cite{dahlke2008generalized, dahlke2009shearlet, dahlke2007frames}. We highly recommend the recent work \cite{romero2020dual} where the authors derive discretization improvements and, in their own words, \textquote{bridge a
    gap between what is achievable with abstract and concrete methods}. Finally, we refer the reader to \cite{fornasier2005continuous} where a generalization of the coorbit space theory is used to derive atomic decompositions and Banach frames for a wide range of Banach spaces.
\end{itemize}

If you have found a typographical or mathematical error when reading this survey, it would be very much appreciated if you would let me know. The same goes if some work on coorbit theory you believe deserves to be mentioned has been omitted.

\Addresses



\bibliographystyle{plain}
\bibliography{main.bib}

\end{document}